\documentclass[a4paper, 12pt]{article}
\usepackage{amsmath,amssymb,amsthm,graphicx,epstopdf,indentfirst,mathrsfs,geometry,enumitem,verbatim,tcolorbox,pifont,graphicx,enumitem,txfonts}
\usepackage[runin]{abstract}

\setenumerate[1]{itemsep=0pt,partopsep=0pt,parsep=\parskip,topsep=5pt}
\setitemize[1]{itemsep=0pt,partopsep=0pt,parsep=\parskip,topsep=5pt}
\setdescription{itemsep=0pt,partopsep=0pt,parsep=\parskip,topsep=5pt}

\abslabeldelim{\ \ }
\newtheorem{theorem}{Theorem}[section]

\newtheorem{lemma}[theorem]{Lemma}

\newtheorem{remark}[theorem]{Remark}

\newtheorem{proposition}[theorem]{Proposition}

\newenvironment{e*}{\begin{equation*}}{\end{equation*}}

\def\slashii#1{\setbox0=\hbox{$#1$}             
	\dimen0=\wd0                                 
	\setbox1=\hbox{\sl/} \dimen1=\wd1            
	\ifdim\dimen0>\dimen1                        
	\rlap{\hbox to \dimen0{\hfil\sl/\hfil}}   
	#1                                        
	\else                                        
	\rlap{\hbox to \dimen1{\hfil$#1$\hfil}}   
	\hbox{\sl/}                               
	\fi}                                         %
\def\slashiii#1{\setbox0=\hbox{$#1$}#1\hskip-\wd0\hbox to\wd0{\hss\sl/\/\hss}}
\begin{document}
	\pagestyle{plain}
	
	\title{On the super-Liouville equations on the sphere}
	
	\author{\small{Mingyang Han}\\
		\small{School of Mathematical Sciences, Shanghai Jiao Tong University}\\
		\small{Shanghai, 200240, China}\\
		\small{hanmingyang@sjtu.edu.cn}\\\\
		\small{Chunqin Zhou}
		\footnote{The second author was partially supported by NSFC Grant 12031012, NSFC Grant 12571223 and  STCSM Grant 24ZR1440700.}
		\\
		\small{School of Mathematical Sciences, Shanghai Jiao Tong University}\\
		\small{Shanghai, 200240, China}\\
		\small{cqzhou@sjtu.edu.cn}}

	\date{}
	\maketitle
	
	\begin{abstract}
		In this paper, we investigate the existence of nontrivial least-energy solutions for the super-Liouville equation with positive coefficient functions on the two-dimensional sphere. Firstly, we derive a global Pohozaev-type identity by analyzing the behavior of solutions under conformal transformations, which generalizes the classical Kazdan-Warner obstruction for the two-dimensional Nirenberg problem. Secondly, by exploiting conformal symmetry, we establish a pointwise estimate that bounds the norm of the spinor component by the scalar component, and show that the $H^1 \times H^{1/2}$ energy of the spinor part remains uniformly bounded. As a byproduct of our analysis, parallel techniques are applied to the Dirac-Einstein equations on the 3-sphere, demonstrating that nontrivial solutions are uniformly bounded away from the trivial solution in the $H^1 \times H^{1/2}$ norm. Moreover, the compactness of the solution space is also analyzed from two perspectives: in the low-energy regime, and modulo the action of the M\"obius group. Finally, by introducing a new natural constraint $\mathcal{A}$ and employing variational methods, we obtain a supersymmetric generalization of the Moser–Trudinger–Onofri inequality and  establish the existence of least-energy solutions for even coefficient functions. In particular, these solutions are shown to be nontrivial provided that a certain spectral parameter associated with the coefficients satisfies $\lambda_1(h_2, h_1) < 1$. Concurrently, we provide a complete classification of nontrivial least-energy solutions in the case of positive constant coefficients.
	\end{abstract}
	
	{\bf Keywords:}  Super-Liouville equation; Dirac operator; conformal invariance; Kazdan–Warner obstruction; supersymmetric functional inequalities; least-energy solutions; sharp \(H^{1/2}\)-Sobolev inequality; compactness; classification
	
	
	\section{Introduction}

	The following two-dimensional Liouville-type equation
	$$\Delta u + K{e^{2u}} = 0$$
	which frequently appears in mathematics and physics, can be traced back to Liouville and Monge  \cite{Cai_xiaohan,liouville1853equation}. On surfaces, it usually takes the form
	\begin{equation*}
		- \Delta u = {K_1}{e^{2u}} - {K_2}
	\end{equation*}	
	which arises in various problems, such as solving mean field equations or the problem of prescribed conformal Gaussian curvature on surfaces. Foremost among these classical endeavors is the celebrated two-dimensional Nirenberg problem: given a function $K$ on the standard sphere $(\mathbb{S}^2,g_0)$, one seeks a metric within the pointwise conformal equivalence class of the metric $g_0$ (i.e., $g = e^{2u}g_0$, where $u$ is a smooth function on the sphere), such that the curvature related to $g$ is the given function $K$. Analytically, the Nirenberg problem is equivalent to finding a solution to the following equation on the sphere $\left(\mathbb{S}^2,g_0 \right) $
	\begin{equation}\label{Liouville}
		- \Delta_{g_0} u = K{e^{2u}} - 1,
	\end{equation}
	where $	 \Delta_{g_0} $ is the Laplace-Beltrami operator. This problem has been a focal point of geometric analysis, as seen in \cite{MR1417436,MR1261723,MR1131392,changMR0908146,chang1988conformal,MR896027,MR1338474,MR339258}. As is well known, equation (\ref{Liouville}) is not solvable for some functions $K$. More specifically,  Kazdan and Warner first proposed the following identity
	\begin{equation}\label{Kaz-War}
		\int_{{\mathbb{S}^2}} {\left\langle {\nabla K,\nabla {x_j}} \right\rangle {e^{2u}}{\text{d}}{v_{g_0}}  = 0,\  j = 1,2,3,}
	\end{equation}	
	where ${x_j},j = 1,2,3$ are the coordinate functions of the standard embedding of the sphere into three-dimensional Euclidean space. A solution $u$ of equation (\ref{Liouville}) needs to satisfy identity (\ref{Kaz-War}). However, for certain functions $K$, such as the coordinate functions $x_j$, this identity (\ref{Kaz-War}) is not satisfied. Later, Bourguignon and Ezin \cite{MR882712} generalized equation (\ref{Kaz-War}) to
	\begin{equation*}
		\int_{{\mathbb{S}^2}} {X\left( K \right){e^{2u}}} {\text{d}}{v_{g_0}} = 0,
	\end{equation*}	
	where $X$ is any conformal vector field on $\mathbb{S}^2$. In particular, Futaki \cite{MR718940} extended identity (\ref{Kaz-War}) to a complex variable version.
	
	For the Nirenberg problem, The seminal work of Moser \cite{MR339258} provided an existence result that if the function $K$ on the sphere is even, i.e.
	\begin{equation*}
		K(-x) = K(x),\ \forall x\in \mathbb{S}^2,
	\end{equation*}
	then equation (\ref{Liouville}) has a solution. Moser used the maximization method for the functional
	\begin{equation*}
		F\left( u \right) = \log \left( {\frac{1}{{4\pi }}\int_{{\mathbb{S}^2}} {K{e^{2u}}} {\text{d}}{v_{g_0}}} \right) - \frac{1}{{4\pi }}\int_{{\mathbb{S}^2}} {{{\left| {\nabla u} \right|}^2}} {\text{d}}{v_{g_0}} - \frac{1}{{2\pi }}\int_{{\mathbb{S}^2}} u {\text{d}}{v_{g_0}},
	\end{equation*}
	where the crucial technique relied on the sharp constant of the famous Moser-Trudinger inequality. Chang and Yang \cite{chang1988conformal}, Hong \cite{MR845999}, Chen and Ding\cite{MR896027}, Xu and Yang \cite{MR1087058} and other authors provided other solvability conditions on functions $K$ with certain symmetries. Without any symmetry assumptions on functions $K$, Chang and Yang \cite{changMR0908146} first achieved a breakthrough using the minimax method and the blow-up technique. Subsequently, numerous studies have been conducted on this problem, yielding new results. In a parallel development, Struwe \cite{MR2137948} initiated the flow approach to this problem, while Han \cite{MR1084455} and Chang and Liu \cite{chang1993nirenberg} employed the method of Morse theory with boundary to arrive at the same results as Chang and Yang \cite{changMR0908146}. The prescribed conformal curvature problem has also been extensively studied on surfaces $\Sigma$ with $\chi(\Sigma) \le 0$, which differs significantly from the spherical case ($\chi(\mathbb{S}^2) > 0$); see \cite{MR295261,MR1937411,MR375153,kazdan1974curvature,MR2137948}.
	
	In physics, the Liouville equation serves as a significant example in two-dimensional integrable models, conformal field theories and quantum gravity. Especially in Liouville conformal field theory (LCFT), conformal bootstrap and the DOZZ formula provide a method for the exact computation of this model, thereby deepening the understanding of the Liouville theory and its applications in mathematical physics. On a Riemann surface \((\Sigma,g)\) with complex coordinates \((z, \bar{z})\), the known Lagrangian action of the LCFT is defined as follows:
	\[
	S_{L}[\phi] = \int_{\Sigma} \mathrm{d}^{2}z \sqrt{g} \left( \frac{1}{4\pi} g^{\alpha\beta} \partial_{\alpha}\phi \partial_{\beta}\phi + 2QK  \phi + \mu e^{2b\phi} \right).
	\]	
	Here, \( b \) is the coupling constant, \( Q \) is the background charge, \(\mu\) is referred to as the Liouville cosmological constant, and \(K\) is the Gaussian curvature defined on the surface \(\Sigma\). Its historical development spans multiple fields, including quantum field theory, string theory and stochastic geometry. For detailed information, see \cite{alday2010liouville,MR757857,dorn1994two,MR4816634,MR1171758,MR4060417,MR623209,polyakov1987gauge,teschner2001liouville,MR1413469}. In particular, for the physical background of the Liouville theory on the Riemann sphere, see \cite{david2016liouville}. In 1981, Polyakov \cite{polyakov1981quantum}, following his work on the bosonic string in quantum theory of the two-dimensional Liouville lagrangian, introduced a supersymmetric generalization for the fermionic string. This involved the supersymmetric Liouville theory, which effectively sums over fermionic surfaces. The corresponding Lagrangian action with trivial topology is given by
	\[
	W =  - \frac{D}{{8\pi }}\int {\left[ {\frac{1}{2}{{({\partial _i}\phi)}^2} + \frac{1}{2}i\overline \psi  \slashiii{D}\psi  + \frac{1}{2}{\mu ^2}{e^{2\phi}} + \frac{1}{2}\mu {e^\phi}(\overline \psi  {\gamma _5}\psi )} \right]}
	\]
	where \( D \) is the spacetime dimension, \( \phi \) is a scalar superfield, \( \partial_i \) denotes the spacetime derivative, \( \psi \) is a spinor, $\slashiii{D}$ is the Dirac operator, \( \mu \) is a mass scale, and \( \gamma_5 \) is a gamma matrix. For more articles on the supersymmetric Liouville theory in physics, see \cite{MR1372587,MR1028444,fukuda2002super,MR1608487,strominger1998black}. Based on Polyakov's work, Jost, Wang and Zhou \cite{jost2007super} conducted research from a mathematical perspective and extended the supersymmetric Liouville theory to topologically nontrivial Riemann surfaces. They investigated the following functional on a Riemann surface $(\Sigma,g)$:
	$$
	E_0(u,\psi ) = \int_\Sigma {\left\{ {\frac{1}{2}{{\left| {\nabla u} \right|}^2} + {K_g}u + \left\langle { \slashiii{D}\psi ,\psi } \right\rangle  - {e^{2u}} + {e^u}{{\left| \psi  \right|}^2}} \right\}} {\text{d}}{v_g}.
	$$
	The Euler-Lagrange equation associated with the functional of the supersymmetric Liouville theory is known as the super-Liouville equation and typically takes the following form
	\begin{equation}\label{prev-sup-Liou}
		\begin{cases}
			{ - \Delta_g u =h_1{e^{2u}}- {K_g}+h_2 {e^{u}}{{\left| \psi    \right|^2}}}  ,&{{\text{on}}\;{\Sigma},} \\
			{\slashiii{D}_g\psi    = h_2{e^{u}}\psi  },&{{\text{on}}\;{\Sigma},}
		\end{cases}
	\end{equation}
	where $\Sigma$ is a closed Riemann surface, $K_g$ is the curvature of the metric $g$, $\slashiii{D}_g$ is the Dirac operator, and $\psi$ is a spinor, see Section 2 for details. When both $h_1$ and $h_2$ are constants, the equation is conformally invariant under $\widetilde{g}=e^{2v}g$. Specifically, by setting
	$$
	\begin{cases}
		{\tilde u = u - v,} \\
		{\tilde \psi  = {e^{-\frac{v}{2}}}\psi ,}
	\end{cases}
	$$
	equation (\ref{prev-sup-Liou}) becomes
	\begin{equation*}
		\begin{cases}
			{ - {\Delta _{\tilde g}}\tilde u = {h_1}{e^{2\tilde u}} - {K_{\tilde g}}+ {h_2}{e^{\tilde u}}{{\left| {\tilde \psi} \right|}^2} ,} &{{\text{on}}\;{\Sigma},} \\
			{{\slashiii{D}_{\tilde g}}\tilde \psi  = {h_2}{e^{\tilde u}}\tilde \psi ,} &{{\text{on}}\;{\Sigma}.}
		\end{cases}
	\end{equation*}	
	See \cite{jost2007super} for more details.
	
	When ${h_1} \equiv 2$ and ${h_2}  \equiv  - 1$, Jost et al. \cite{jost2007super,jost2009energy} investigated the geometric and analytic properties of equation (\ref{prev-sup-Liou}), including the regularities and the blow-up phenomenons of weak solutions. In particular, when \( \Sigma \) is a sphere, the authors demonstrated that the blow-up points are confined to at most a single point and obtained the precise blow-up value of \( 4\pi \).
	
	When $h_1 \equiv1$ and $h_2 \equiv \rho>0$ is a parameter, Jevnikar et al. \cite{MR4206467} studied the existence of solutions on the sphere and obtained perturbation results by using bifurcation method and Morse theory in variational method.
	
	When $h_1 \equiv-1$, $h_2 \equiv \rho\in\mathbb{R}$ and the genus $\gamma$ of the surface $M$ is greater than 1,  Jevnikar et al. \cite{jevnikar2020existence1} used the mountain pass theorem and the linking theorem to show that equation (\ref{prev-sup-Liou}) has nontrivial solutions (which refers to solutions satisfying $\psi\not \equiv0$) provided that $\rho$ does not belong to the spectrum of the Dirac operator $\slashiii{D}$.
	
	When $h_1$ or $h_2$ in equation (\ref{prev-sup-Liou}) is not a constant, the equation is not conformally invariant. In this case, unlike the prescribed Gaussian curvature problem, there are relatively few results on the existence of nontrivial solutions. In a previous work \cite{han2024existence}, we established the existence of nontrivial solutions on closed surfaces with genus $\gamma > 1$, under the conditions that $h_1 < 0$ and $h_2 \equiv \rho \in \mathbb{R}$, provided that $\rho$ does not belong to the spectrum of the weighted Dirac operator.
	
	For the research on super-Liouville equations on more generally closed surfaces and surfaces with nonempty boundary, see \cite{jevnikar2020existence1,MR4206467,jost2007super,jost2009energy,jost2014boundary,jost2014qualitative,jost2015local}. 	 The generalization of the super-Liouville equation to higher dimensions (\(n \geq 3\)) is known as the Dirac-Einstein equation, which has been the subject of many studies. For more details, we refer to \cite{MR4267625,MR4505161,MR1709232,MR1793008,MR1738150,MR3906258,MR4844577,MR4707830}.
	
	In this paper, we focus on the existence of solutions to the following equation on the standard sphere $(\mathbb{S}^2,g_0)$
	\begin{equation}\label{key-Sphere}
		\begin{cases}
			{ - \Delta u = {h_1}(x){e^{2u}} - 1 + {h_2(x)}{e^{u}}{{\left| \psi  \right|}^2},}&{{\text{on}}\;{\mathbb{S}^2},} \\
			{\slashiii{D}\psi  = h_2(x){e^{u}}\psi ,}&{{\text{on}}\;{\mathbb{S}^2}.}
		\end{cases}
	\end{equation}
	Here, $\Delta$, $\slashiii{D}$, and \( \left|  \cdot  \right| \) are all considered with respect to the standard metric \( g_0 \) on $\mathbb{S}^2$ and $h_i(x),i=1,2$ are nonnegative functions. In particular, to simplify the proof, we consider in this paper the case where \( u \) is a real-valued function and \( \psi \) is a real spinor. For convenience, unless otherwise specified, we write \( {\text{d}{v}} \) for \({\text{d}}{v_{{g_0}}} \), the volume form associated with the standard metric on the sphere. The minimum and maximum values of the functions $h_i,i=1,2$ are denoted by $h_{i,\min}$ and $h_{i,\max}$, respectively.
	
	A further motivation for considering such sign choices comes from recent developments in timelike \(N=1\) super-Liouville theory. M\"uhlmann et al. \cite{muhlmann2026three} derived the three-point structure constants of timelike \(N=1\) Liouville conformal field theory and showed that, in a suitable normalization, they are expressed by inverse-type formulae in terms of the corresponding spacelike structure constants. They also showed that, at degenerate values of the momenta, these constants reproduce the structure constants of \(N=1\) superconformal minimal models. Although the present paper is concerned with a classical geometric PDE system rather than with the quantum bootstrap construction, the constant-coefficient equation studied here may be regarded, at least formally, as a conformally invariant super-Liouville-type model with a timelike sign structure.
	
	Recall that for the prescribed Gaussian curvature equation (\ref{Liouville}) on the two-dimensional sphere, there exists a Pohozaev-type obstruction known as the Kazdan Warner identity (\ref{Kaz-War}). For the Dirac equation
	\[
	\slashiii{D}\psi = V(x)\psi,\quad \text{on}\ {\mathbb{S}^2} ,
	\]
	with a smooth function $V$, a similar argument, an argument analogous to the proof of Lemma 4.3 by Jevnikar et al. \cite{MR4206467} readily yields
	\[
	\int_{{\mathbb{S}^2}} V \nabla |\psi|^2 \nabla x_i \, {\text{d}}{v} = \int_{{\mathbb{S}^2}} V |\psi|^2x_i  \, {\text{d}}{v} = \int_{{\mathbb{S}^2}} \langle \nabla V, \nabla x_i \rangle |\psi|^2 \, {\text{d}}{v},\quad i=1,2,3.
	\]
	Consequently, for equation (\ref{key-Sphere}), we obtain the following Pohozaev-type identity.
	\begin{proposition}\label{Kazdan-w}
		If $h_1(x)$ and $h_2(x)$ are two smooth functions on ${\mathbb{S}^2}$ and $(u,\psi)$ is a smooth solution to equation (\ref{key-Sphere}), then
		\[\int_{{\mathbb{S}^2}} {\langle \nabla {h_1},\nabla {x_i}\rangle } {e^{2u}}{\mkern 1mu} {\text{d}}v + 2\int_{{\mathbb{S}^2}} {\langle \nabla {h_2},\nabla {x_i}\rangle } {e^u}|\psi {|^2}{\mkern 1mu} {\text{d}}v = 0,\quad i = 1,2,3.\]
	\end{proposition}
	When dealing with equation (\ref{key-Sphere}), two fundamental difficulties arise. The first is analogous to that encountered in the Nirenberg problem and is related to the non-compactness of the conformal symmetry group of the sphere. The second difficulty arises from the nature of the Dirac operator. The second difficulty arises from the strongly indefinite nature of the Dirac operator, whose algebraic structure limits straightforward coordinate expressions. For the Dirac operator, a particularly important formula is the Lichnerowicz formula (\ref{Lichnerowicz}), see Section 2. Inspired by the proof of Lemma 3.1 by Maalaoui and Martino \cite{MR4423145} and by exploring conformal transformations with the Lichnerowicz formula for the super-Liouville equation on an arbitrary closed surfaces $(M,g)$, we obtain that the spinor part \(\psi\) of the solution can be controlled by the scalar function \(u\). Specifically, for the following equation
	\begin{equation}\label{general}
		\begin{cases}
			{ - {\Delta _g}u = h_1(x){e^{2u}} - {K_g} + h_2(x){e^u}{{\left| \psi  \right|}^2},}&{{\text{on}}\;{M},} \\
			{{\slashiii{D}_g}\psi  = h_2(x){e^u}\psi ,}&{{\text{on}}\;{M}.}
		\end{cases}
	\end{equation}
	where \( K_g \) is the Gaussian curvature of the surface \(({M}, g)\), we have the following result.
	\begin{theorem}\label{thm-principle}
		Assume that \( (M,g) \) is a closed Riemann surface with its Gauss curvature $K_g$ and $h_1(x)$ and $h_2(x)$ are two smooth functions on $M$. In addition, assume that \( h_2 > 0 \) and \( h_1 \leq 2h_2^2 \). If $(u,\psi)\in H^1(M)\times H^{1/2}( \mathcal{S}M )$ is a weak solution of (\ref{general}), then there is a constant $C=\mathop {\max }\limits_{x \in {\mathbb{S}^2}} \left\{ {\left( {2h_2^2 - {h_1}} \right)/{h_2}} \right\}>0$ such that	\[{\left| \psi  \right|} \leq C{e^{u/2}}.\]
	\end{theorem}
	\begin{remark}
		As a direct consequence of the estimate in Theorem \ref{thm-principle}, any blow-up point of \(\psi\) must also be a blow-up point of \(u\), i.e., the blow-up set of \(\psi\) is contained in that of \(u\).
	\end{remark}
	On the other hand, on a sphere $(\mathbb{S}^2,g)$, the strongly indefinite nature of the Dirac operator may also provide some additional information. For instance, by exploiting the strongly indefinite nature, we derive the first result, which concerns a priori estimate for the spinor component of the solution. Specifically, we show that the Sobolev norm of spinors, see Section \ref{Sobolev-spin}, of all solutions is bounded.
\begin{theorem}\label{thm1}
	Assume that $M=(\mathbb{S}^2,g)$ with \( K_g \geq 0 \). If $h_1(x)>0$ and $h_2(x)\geq 0$ are two smooth functions on $\mathbb{S}^2$ and $(u,\psi)\in H^1(\mathbb{S}^2)\times H^{1/2}( \mathcal{S}\mathbb{S}^2 )$ is a weak solution of (\ref{general}), then there is a constant $C=C(h_1,h_2)>0$ such that	$\left\| \psi  \right\|_{{H^{1/2}}\left( \mathcal{S}\mathbb{S}^2 \right)} \leq C.$
\end{theorem}
\begin{remark}
	By applying similar techniques to the three-dimensional Dirac-Einstein equation on the sphere, we obtain a series of results exhibiting distinctive features that have not been previously documented in the literature. We consider the following coupled system on the three-dimensional sphere $(\mathbb{S}^3, g)$:
	\begin{equation}\label{DE}
		\begin{cases}
			{ - {\Delta _g}u + \frac{{{R_g}}}{8}u = Ku{{\left| \psi  \right|}^2},}&{{\text{on}}\;{\mathbb{S}^3},} \\
			{{\slashiii{D}_g}\psi  = K{u^2}\psi ,}&{{\text{on}}\;{\mathbb{S}^3},}
		\end{cases}
	\end{equation}
	where $u \in H^1(\mathbb{S}^3)$ is a positive scalar field, $\psi \in H^{1/2}(\Sigma\mathbb{S}^3)$ is a spinor field, and $K$ is a positive function. For further background on this equation, see \cite{MR4234090}. We recall that on a 3-manifold, the Sobolev critical exponent for the fractional space $H^{1/2}$ is $2^* = 3$.  Because the sphere $\mathbb{S}^3$ carries no harmonic spinors, it is clear that system \eqref{DE} admits no semi-trivial solutions of the form $(u, 0)$ or $(0, \psi)$. Then we have the following theorem.
	\begin{theorem}\label{DE-th}
		Assume that the scalar curvature \(R_g\) and the function \(K\) are both positive and smooth, then there exists a constant $C=C\left(R_g,K \right)>0 $ such that any solution $\left( u, \psi\right) \not \equiv\left(0,0 \right) $ to the equation \eqref{DE} satisfies the uniform lower bound
		$$
		\|u\|_{H^1} \geq C,\quad {\left\| \psi  \right\|_{{H^{1/2}}}} \geq C,
		$$
		and
		$$\frac{C}{\|u\|_{H^1}} \leq \|\psi\|_{{H^{1/2}}} \leq C \|u\|_{H^1}^2.$$
	\end{theorem}
\end{remark}

With the help of the above uniform estimate for the super-Liouville equations, we can study the compactness of solutions. For this purpose, we assume that ${h_{i,n}} \to h_i$ in $C^1({\mathbb{S}^2}),\ i=1,2$, and $(u_n,\psi_n)\in H^1(\mathbb{S}^2)\times H^{1/2}( \mathcal{S}\mathbb{S}^2 )$ satisfy
\begin{equation}\label{jinsi}
	\begin{cases}
		{ - \Delta u_n = {h_{1,n}}{e^{2u_n}} - 1 + {h_{2,n}}{e^{u_n }}{{\left| \psi_n   \right|}^2},}&{{\text{on}}\;{\mathbb{S}^2},} \\
		{\slashiii{D}\psi_n   ={h_{2,n}}{e^{u_n}}\psi_n  ,}&{{\text{on}}\;{\mathbb{S}^2}.}
	\end{cases}
\end{equation}
Then, we will show that the solutions are uniformly bounded under the small energy condition.
\begin{theorem}\label{thm1.3}
	Assume that \((u_n, \psi_n) \in H^1(\mathbb{S}^2)\times H^{1/2}( \mathcal{S}\mathbb{S}^2 )\) is a sequence of weak solutions for (\ref{jinsi}) with
	\begin{equation}\label{coefficient-bound}
		0 < m_i \leq {h_{i,n}}\left( x \right) \leq M_i,\quad i=1,2
	\end{equation}
	and
	\begin{equation}\label{concentrate}
		\mathop {\lim }\limits_{r \to 0} \mathop {\lim }\limits_{n \to \infty } \int_{{B_r}(x)} {\left( {{h_{1,n}}{e^{2{u_n}}} +{h_{2,n}} {e^{{u_n}}}{{\left| \psi_n  \right|}^2}} \right)} {\mkern 1mu} \mathrm{d}{v_{{g_0}}} < 2\pi .
	\end{equation}
	Then \(\left(  u_n,\psi_n \right) \) is bounded in  \( C^k(\mathbb{S}^2)\times {C^k}\left( \mathcal{S}\mathbb{S}^2 \right) ,\forall k>0\).
\end{theorem}
\begin{remark}
	Regarding the small-energy compactness result for the super-Liouville equation (\ref{prev-sup-Liou}) with a negative coupling coefficient, a case distinct from our consideration, Jost et al. \cite[Theorem 4.5]{jost2007super} established the following conclusion under the parameter choice $h_1=2$ and \( h_{2}=-1 \).
	\begin{proposition}[Theorem 4.5, \cite{jost2007super}]\label{JZW}
		Let $M$ be a closed Riemann Surface. Assume that \((u_n, \psi_n)\) is a sequence of solutions for (\ref{prev-sup-Liou})
		with		
		\[
		\int_M e^{2u_n} \mathrm{d}v < \varepsilon_0, \quad \text{and} \quad \int_M |\psi_n|^4 \mathrm{d}v < C
		\]		
		for some positive constant \(\varepsilon_0 < \pi\) and \(C\). Then we have		
		\[
		\|u_n\|_{C^k(B_{\frac{1}{8}}(x))} + \|\psi_n\|_{C^k(B_{\frac{1}{8}}(x))} \leq C.
		\]		
		for any geodesic ball \(B_{\frac{1}{8}}(x)\) of \(M\).
	\end{proposition}
	In Proposition \ref{JZW}, only a sufficiently small energy of the scalar function \( u \) is required. This improvement is made possible by a favorable a priori estimate that arises specifically when the coupling coefficient is negative. In contrast, Theorem \ref{thm1.3} deals with a positive coupling coefficient, rendering the techniques used in Proposition \ref{JZW} inapplicable.
\end{remark}
Next, we will present an alternative compactness result. We will adopt some definitions from the work of Chang and Yang \cite{MR1131392, changMR0908146}. Firstly, let us recall the conformal invariant quantities
\[
S\left[ u \right]: = \fint_{{\mathbb{S}^2}} {{{\left| {\nabla u} \right|}^2}{\text{d}{v}_{g_0}} + 2\bar u},
\]
which was introduced by Chang and Yang  \cite{changMR0908146}. To address the loss of compactness caused by conformal transformations of the sphere, let us define the set
\begin{equation}\label{zhixin}
	{\mathscr{S}}: = \left\{ {u \in {H^1}\left( {{\mathbb{S}^2}} \right):\fint_{{\mathbb{S}^2}} {{e^{2u}}{x_j}{\text{d}{v_{g_0}}}}  = 0,j=1,2,3} \right\}.
\end{equation}
Under the constraint on the centroid, a new compactness result is established for the super-Liouville equation (\ref{jinsi}).
\begin{proposition}\label{prop1.4}
	Assume that \(\left\{ {\left( {{u_n},{\psi _n}} \right)} \right\} \in H^1(\mathbb{S}^2)\times H^{1/2}( \mathcal{S}\mathbb{S}^2 )\) is a sequence of weak solutions for (\ref{jinsi}) with $0 < m_i \leq {h_{i,n}}\left( x \right) \leq M_i,\ i=1,2$. If $u_n\in \mathscr{S}$ and ${\left\| \psi_n  \right\|_{{L^{\frac{{32}}{7}}}(\mathcal{S}\mathbb{S}^2)}} \leq C$, then we have
	$$\left\| {\left( {u_n,\psi_n } \right)} \right\|_{H^1(\mathbb{S}^2)\times H^{1/2}( \mathcal{S}\mathbb{S}^2 )} \leq C.$$
	In particular, we have
	$$ (i)~ u_n \geq -C, ~~~~ (ii)~ \int_{\mathbb{S}^2}|\nabla u_n|^2 \mathrm{d}v_{g_0} \leq C.$$
\end{proposition}

\begin{remark}From $(i)$ and $(ii)$ we have $S[u_n]\leq C$. 
	From Theorem \ref{thm-principle} and the proof of Theorem \ref{thm1.3}, it can be seen that the compactness of the spinor part can be largely controlled by the scalar part. However, Proposition \ref{prop1.4} tells us that on the sphere, the compactness of the scalar part is actually also constrained by the spinor part.
\end{remark}

Next, we will investigate the existence of nontrivial solutions on the sphere. We briefly review the existence results that have been established for equation (\ref{prev-sup-Liou}) with $K_g \equiv K_{g_0}\equiv1$. In the case of $h_1\equiv2$ and $h_2\equiv-1$, Jost et al. \cite{jost2007super} identified that, under these circumstances, equation (\ref{prev-sup-Liou}) admits two fundamental types of solutions. One type is \((u_\phi, 0)\), where \(u_\phi = \frac{1}{2}\log \frac{1}{2} + \frac{1}{2}\log \det\left(  {d\phi } \right) \) and $\phi$ is a conformal diffeomorphism of ${\mathbb{S}^2}$.
The other type is \((0, \psi_0)\), where \(\psi_0\) is a Killing spinor with $\left| {{\psi _0}} \right| \equiv1$. When $h_1\equiv1$ and $h_2\equiv \rho\in\mathbb{R}$ is a parameter, Jevnikar et al. \cite{MR4206467} pointed out that
\begin{enumerate}[label=\roman*), leftmargin=*]
	\setcounter{enumi}{0}
	\item if $(u',\psi')$ is a solution to (\ref{prev-sup-Liou}) and $\psi'$ is not identically zero, then $\rho>1$.
	\item equation (\ref{prev-sup-Liou}) has solutions
	$\left( {u'_\rho,\psi'_\rho } \right) = \left( { - \ln \rho ,\frac{{\sqrt {{\rho ^2} - 1} }}{\rho }{\varphi _1}} \right)$
	for $\forall\rho \geq 1$, where $\varphi_1$ is the first positive eigenspinor of $\slashiii{D}$ as well as $\left| {{\varphi _1}} \right| = 1$;
	\item if $\rho$ is in the spectrum of the Dirac operator $\slashiii{D}$, then $\rho$ is a bifurcation point, thus yielding some nontrivial solutions of (\ref{prev-sup-Liou}) when $\rho$ is near the eigenvalues of $\slashiii{D}$.
\end{enumerate}
\begin{remark}
	Based on the above information, it is particularly worth mentioning that for the conformally invariant equation
	\[\begin{cases}
		{ - \Delta u = {e^{2u}} - 1 + {e^u}{{\left| \psi  \right|}^2},}&{{\text{on}}\;{\mathbb{S}^2},} \\
		{\slashiii{D}\psi  = {e^u}\psi ,}&{{\text{on}}\;{\mathbb{S}^2},}
	\end{cases}\]
	all solutions must necessarily be of the form \((u_\phi, 0)\). In other words, the solutions to this equation are completely classified. Regarding the form that first proposed by Jost-Wang-Zhou,
	\[\begin{cases}
		{ - \Delta u = {e^{2u}} - 1 - {e^u}{{\left| \psi  \right|}^2},}&{{\text{on}}\;{\mathbb{S}^2},} \\
		{\slashiii{D}\psi  =  - {e^u}\psi ,}&{{\text{on}}\;{\mathbb{S}^2},}
	\end{cases}\]
	there exist two types of semi-trivial solutions, \((u_\phi, 0)\) and \((0,\psi_0)\), even though the two equations differ only by the sign in front of the coupling term.
\end{remark}

Let us return to equation (\ref{key-Sphere}). We first seek some functions $h_i,i=1,2$ such that equation (\ref{key-Sphere}) admits special solutions with both non-zero components. Let $h_2$ be a positive constant. If the solution $u$ is a constant $\bar u$, then by the second linear equation of (\ref{key-Sphere}), we may choose that ${e^{\overline u }} = {\lambda _i}/{h_2},\ i\geq 1$, where $\lambda_i$ is the eigenvalue of the Dirac operator. In this case, the spinor can be chosen to be $\psi=c\varphi_i$ for any $c\in \mathbb{R}$. Thus from the first equation we have  $0 = {\left( \lambda _i/{h_2}\right) ^2}{h_1} - 1 + {\lambda _i}{c^2}\left| {{\varphi _i}} \right|^2.$	Therefore $0 \leq \left| {{\varphi _i}} \right|^2 = \frac{1}{{{c^2}}}\left( {\frac{1}{{{\lambda _i}}} -\frac{{\lambda _i}{h_1}}{h^2_2}} \right)$.  From section 2 we know the first positive eigenvalue $\lambda$ of Dirac operator on the sphere is 1. It follows that in this case, \(h_1 \leq h_2^2\). This indicates that for some \(h_1 \leq h_2^2\), there exist nontrivial solutions with $\psi  \not \equiv 0$.

Actually, for the case in which  $h_1$ and $h_2$ are positive functions, by applying recent results on zero mode inequalities established by Wang and Zhang \cite{wang2025conformal}, we can determine the conditions under which equation (\ref{key-Sphere}) admits solutions with a nontrivial spinor component.
\begin{theorem}\label{trivial-condation}
	Assume $h_i(x) \in C^{\infty}(\mathbb{S}^2),i=1,2,$ are positive smooth functions on $\mathbb{S}^2$. If equation (\ref{key-Sphere}) admits a solution \((u,\psi)\) with \(\psi \not\equiv 0\), then there must be \( h_{1,\min} < h_{2,\max}^2 \).
\end{theorem}

To establish the general existence of solutions, we employ variational methods to obtain  least-energy solutions for equation (\ref{key-Sphere}) when the functions \( h_i,\ i=1,2 \) are even. Consider the functional \(E: {H^1}({\mathbb{S}^2}) \times {H^{\frac{1}{2}}}(\mathcal{S} {\mathbb{S}^2})\to \mathbb{R} \) defined by
\begin{equation}\label{Functional}
	E\left( {u,\psi } \right) = \int_{{\mathbb{S}^2}} {\left[ {{{\left| {\nabla u} \right|}^2} + 2u - h_1{e^{2u}} + 2\left( {\left\langle {\slashiii{D}\psi ,\psi } \right\rangle  - h_2{e^{u}}{{\left| \psi  \right|}^2}} \right)} \right]{\text{d}{v}_{g_0}}}  + 4\pi.
\end{equation}
Equation (\ref{key-Sphere}) represents the Euler-Lagrange equation corresponding to this functional. The solutions we seek correspond to critical points of the functional \(E(u,\psi)\) that attain the minimal energy level, mathematically equivalent  to the physical concept of ground state solutions. For terminological consistency in the variational framework, we shall exclusively use the term \textit{least-energy solutions} throughout the work.

The primary challenge stems from the fact that the functional \(E(u,\psi) \) is neither bounded above nor below. Owing to the spectral properties of the Dirac operator, the functional exhibits strong indefiniteness. Drawing inspiration from the work of Jevnikar et al. \cite{jevnikar2020existence1,MR4206467}, we introduce a novel natural constraint \( \mathcal{A} \), under which the critical points of the functional \(E(u,\psi) \) coincide with its unconstrained critical points. Utilizing the Moser-Trudinger inequality, we demonstrate that the functional becomes bounded below and coercive when restricted to this constraint. Using the same ideas and techniques, we prove a supersymmetric inequality, see Theorem \ref{SMT}, which can be regarded as a supersymmetric generalization of the classical Moser–Trudinger–Onofri inequality.

Another additional difficulty involves excluding semi-trivial solutions and ensure the nontriviality of the spinor component. Given the even symmetry of  \(h_i\), we employ the well-known even symmetric Moser Trudinger inequality combined with a direct variational approach to establish that the functional \(4\pi S[u]\) attains its minimum on the constraint
\[
\mathcal{A}_0 = \left\{ u \in H^1(\mathbb{S}^2) : \int_{\mathbb{S}^2} h_1 e^{2u} \, \text{d}v = 4\pi, \; h_1(x) = h_1(-x) \right\},
\]
and this minimizer corresponds to a solution of equation (\ref{Liouville}) on the sphere. Denoting  this solution by \(u_{h_1}^*\), we observe that equation (\ref{key-Sphere}) always admits a semi-trivial solution of the form \((u_{h_1}^*,0)\). To exclude such solutions, we introduce a criterion based on the first eigenvalue of the weighted Dirac operator ${\left( {{h_2}{e^{{u_{h_1}^*}}}} \right)^{ - 1}}\slashiii{D}$. We write $\omega \left( {{h_1},{h_2}} \right) = {h_2}{e^{u_{{h_1}}^*}}>0$, and define $\lambda _1({h_2},{h_1})$ as

\[{\lambda _1}({h_2},{h_1}): = \mathop {\inf }\limits_{\psi \not  \equiv 0} \left\{ {\frac{{\int_{{\mathbb{S}^2}} {\left\langle {\slashiii{D}\psi ,\psi } \right\rangle } {\text{d}}v}}{{\int_{{\mathbb{S}^2}} {\omega \left( {{h_1},{h_2}} \right){{\left| \psi  \right|}^2}{\text{d}}v} }}:\psi  \in {H^{1/2}}(\mathcal{S}{\mathbb{S}^2}),\, \psi { \bot _\omega }{\text{Ker}}{{(\slashiii{D} - \mu \omega \left( {{h_1},{h_2}} \right)\mathbb{I})}_{\mu  < 0}}} \right\}\]
Here \(\bot _\omega\) is weighted orthogonal, see Lemma \ref{Eigen-positve} for details. By Lemma \ref{Eigen-positve}, we have \(\lambda_1(h_2, h_1)  > 0\). Hence, we have the following Theorem.
\begin{theorem}\label{main-Thm}
	Let $h_i(x) $ be positive smooth functions on $\mathbb{S}^2$ satisfying $
	h_i(-x) = h_i(x)$ for all $x\in \mathbb{S}^2$, $i=1,2$. Then there exists a least-energy solution \((u, \psi) \in H_{even}^1(\mathbb{S}^2)\times H^{1/2}( \mathcal{S}\mathbb{S}^2 )\) to equation (\ref{key-Sphere}). If the spectral condition \(\lambda_1(h_2, h_1) < 1\) holds, then the least-energy solution is nontrivial, i.e., \(\psi \not\equiv 0\).
\end{theorem}
\begin{remark}
	If both \(h_1\) and \(h_2\) are positive constants, then from the fact that the smallest positive eigenvalue of the Dirac operator on the standard sphere is \(1\), we obtain
	\[
	\lambda_1(h_2, h_1) < 1 \quad \Leftrightarrow \quad h_1 < h_2^2.
	\]
\end{remark}

Finally, we provide a complete classification of nontrivial least-energy solutions when $h_1 \equiv A$ and $h_2 \equiv B$ are constants satisfying $A < B^2$, in which case system \eqref{key-Sphere} reads
\begin{equation}\label{A-B-key-Sphere}
	\begin{cases}
		{ - \Delta_{g_0} u = A{e^{2u}} - 1 + B{e^{u}}{{\left| \psi  \right|}^2},}&{{\text{on}}\;{\mathbb{S}^2},} \\
		{\slashiii{D}_{g_0}\psi  = B{e^{u}}\psi ,}&{{\text{on}}\;{\mathbb{S}^2}.}
	\end{cases}
\end{equation}
\begin{theorem}\label{Classification}
	Assume that \(A<B^2\). Let \((u,\psi)\) be a nontrivial least-energy solution of (\ref{A-B-key-Sphere}). Then, up to the action of the M\"obius group of \(\mathbb{S}^2\),
	\[
	u\equiv -\log B,\qquad \psi=\sqrt{1-\frac{A}{B^2}}\;\varphi_1,
	\]
	where \(\varphi_1\) is a first positive eigenspinor of the standard Dirac operator on \((\mathbb{S}^2,g_0)\), normalized by
	\[
	\slashiii{D}_{g_0}\varphi_1=\varphi_1,\qquad |\varphi_1|\equiv1.
	\]
	
	Equivalently, every nontrivial least-energy solution is of the form
	\[
	u=
	-\log B+\frac12\log \det(d\Phi) \quad \text{and}\quad
	\psi=\sqrt{1-\frac{A}{B^2}}\,\Phi^*\varphi_1,
	\]
	for some conformal transformation \(\Phi\in\mathrm{Conf}(\mathbb{S}^2)\).
\end{theorem}
\begin{remark}
	Via the stereographic projection $\pi:\mathbb{S}^2 \setminus \{N\}\to\mathbb R^2$,
	the least-energy solution obtained in Theorem \ref{Classification} corresponds, up to Euclidean translations and dilations, to the standard entire solution on \(\mathbb R^2\).
	
	More precisely, after pulling back by stereographic projection, the conformal factor becomes
	\[
	u_0(x)
	=\log\frac{2}{B(1+|x|^2)},
	\qquad x\in\mathbb R^2.
	\]
	Moreover, the spinor component takes the form
	\[
	\psi_0(x) = \sqrt{1-\frac{A}{B^2}}\, \frac{\sqrt{2}}{1+|x|^2}\Bigl(\mathbf 1-i\,x\cdot \Bigr)\Psi,
	\]
	where \(\Psi\in\Sigma_2\cong\mathbb C^2\) is a constant spinor satisfying $|\Psi|=1$, and \(x\cdot\) denotes Clifford multiplication in \(\mathbb R^2\).
	Consequently, the two nonlinear energies can be computed explicitly:
	\[
	A\int_{\mathbb R^2}e^{2u_0}{\text{d}}x
	=A\int_{\mathbb R^2} \frac{4}{B^2(1+|x|^2)^2}{\text{d}}x
	=\frac{4\pi A}{B^2},
	\]
	and
	\[
	B\int_{\mathbb R^2}e^{u_0}|\psi_0|^2{\text{d}}x
	=\left(1-\frac{A}{B^2}\right) \int_{\mathbb R^2} \frac{4}{(1+|x|^2)^2}{\text{d}}x
	=4\pi\left(1-\frac{A}{B^2}\right).
	\]
\end{remark}	
The remainder this article is structured as follows. Section 2 established the analytical foundation by reviewing essential spin geometry concepts and compiling key properties of Dirac operators on manifolds with boundary. Building upon these preliminaries, Section 3 explores the conformal invariance of system \eqref{key-Sphere} under M\"obius group transformations, leading to the derivation of a global Pohozaev-type identity that serves as the fundamental algebraic constraint for subsequent asymptotic analysis. The technical heart of the paper lies in Sections 4 and 5, where we develop a series of refined local estimates and perform a precise bubbling analysis, ultimately proving our principal compactness results (Theorems \ref{thm-principle}, \ref{thm1}, \ref{DE-th}, \ref{thm1.3} and Proposition \ref{prop1.4}). In Section 6, we show that the functional is bounded below on the set \(\mathcal{A}\), and we obtain a supersymmetric version of the Moser–Trudinger–Onofri inequality. Section 7 leverages these compactness theorems in conjunction with the intrinsic topological constraints of the underlying functional to construct a variational minimax scheme, thereby establishing the existence of a least-energy solution and completing the proof of Theorem \ref{main-Thm}. The final Section (Section 8) examines the constant coefficient case, providing an explicit classification of nontrivial least-energy solutions as asserted in Theorem \ref{Classification}. In Appendix A, we provide the sharp constant for the \(H^{1/2}\)-Sobolev inequality when \(n \geq 2\) and prove that it is not attained.
\section{Geometric and analytic settings}

\subsection{Spinors}
We first provide a brief overview of the theory of spinors and the Dirac operator on general Riemann surfaces. For more detailed information, see \cite{MR2509837,MR1031992,MR1867733}. Let \((\Sigma,g)\) be a Riemann surface of genus \(\gamma\) equipped with a fixed spin structure. The spinor bundle \(\mathcal{S}\Sigma\) on \(\Sigma\) is equipped with a natural Hermitian product \(\langle\cdot, \cdot\rangle_h\) induced by the metric \(g\). A spinor \(\psi \in \Gamma(\mathcal{S}\Sigma)\) is a section of the spinor bundle \(\mathcal{S}\Sigma\).

The Clifford algebra \(\text{Cl}(T_x\Sigma, g_x)\) is generated by the tangent space \(T_x\Sigma\) and the metric \(g_x\) at any point \(x \in \Sigma\), satisfying the Clifford relation:
\begin{equation}\label{Clifford}
	X_i \cdot X_j + X_j \cdot X_i = -2g_x(X_i, X_j).
\end{equation}
We denote the bundle of Clifford algebras over \(\Sigma\) by
\[
\text{Cl}(\Sigma,g) = \coprod_{x \in \Sigma} \text{Cl}(T_x\Sigma, g_x).
\]
There exists a representation
\[
\rho: T\Sigma \to \text{End}(\mathcal{S}\Sigma),
\]
which can be extended to \(\text{Cl}(\Sigma,g)\):
\[
\rho: \text{Cl}(\Sigma,g) \otimes \mathcal{S}\Sigma \to \mathcal{S}\Sigma, \quad \sigma \otimes \psi \mapsto \rho(\sigma) \cdot \psi.
\]
We will use the notation \(X \cdot \psi\) to represent \(\rho(X) \cdot \psi\). The representation \(\rho\) is compatible with the connection \(\nabla\) and with the Hermitian metrics
${\left\langle { \cdot , \cdot } \right\rangle _h}$. Specifically, if \(\{e_1, e_2\}\) is a local orthonormal frame on \(T\Sigma\) satisfying
\[
e_i \cdot e_j \cdot \psi + e_j \cdot e_i \cdot \psi = -2\delta_{ij} \psi,
\]
then the following properties hold:
\[
\nabla_{e_i}(e_j \cdot \psi) = (\nabla_{e_i} e_j) \cdot \psi + e_j \cdot (\nabla_{e_i} \psi),
\]
\[
\langle \psi, \varphi \rangle_h = \langle e_i \cdot \varphi, e_i \cdot \psi \rangle_h,
\]
and
\begin{equation}\label{Clofford-inner-product}
	{\left\langle {X \cdot {\psi _1},X \cdot {\psi _2}} \right\rangle _h} = {g}\left( {X,X} \right){\left\langle {{\psi _1},{\psi _2}} \right\rangle _h},\quad\forall X \in \Gamma \left( {T\Sigma } \right),\forall {\psi _1},\ {\psi _2} \in \Gamma \left( {S\Sigma } \right).
\end{equation}
The Dirac operator \(\slashiii{D}\) is defined by
\[
\slashiii{D} \psi := \sum_{\alpha=1}^2 e_\alpha \cdot \nabla_{e_\alpha} \psi.
\]
Regarding the Dirac operator, there is the following important Lichnerowicz formula
\begin{equation}\label{Lichnerowicz}
	\slashiii{D}_g^2\psi  =  - {\Delta _g}\psi  + \frac{{{R_g}}}{4}\psi,
\end{equation}
where \( R_g \) denotes the scalar curvature of the metric \( g \). Moreover, on the sphere, there holds Bar's inequality \cite{MR1162671} concerning the eigenvalues of the Dirac operator.
\begin{proposition}
	Let $g$ be any metric on $\mathbb{S}^2$ conformal to the standard metric $g_0$, and let
	$\lambda_1(\slashiii{D}_g)>0$
	denote the first positive eigenvalue of the Dirac operator $D_g$.
	Then
	\begin{equation}\label{Bar-inq}
		\lambda_1(\slashiii{D}_g)^2\operatorname{Area}(\mathbb{S}^2,g)\ge 4\pi.
	\end{equation}
	Moreover, equality holds if and only if $(\mathbb{S}^2,g)$ is a round sphere, equivalently,
	\[
	g=c\,\Phi^*g_0
	\]
	for some constant $c>0$ and some conformal transformation
	$\Phi\in \mathrm{Conf}(\mathbb{S}^2)$.
	In this case, the eigenspinors corresponding to $\lambda_1(\slashiii{D}_g)$ are precisely the Killing spinors.
\end{proposition}
\begin{proof}
	This follows from the sharp eigenvalue estimate of B\"ar
	(or equivalently Hijazi's inequality in dimension two), see \cite{MR1162671,MR834486}.
\end{proof}

For further details on spinors, we refer the reader to \cite{MR1343997, MR0358873, jost2008riemannian}.
\subsection{Conformal symmetry and transformations}
For clarity, we explicitly label various geometric objects with the metric \(g\). The material that follows is primarily drawn from the exposition in \cite[Section 2]{MR4239839}. First, we introduce the behavior of the Dirac operator under a pointwise conformal transformation, i.e., when \(g_v = e^{2v}g\) with \(v\in C^{\infty}(\Sigma)\). Then there is an isometric map
\[
\begin{split}
	b : (T\Sigma,g) &\to (T\Sigma,e^{2v}g) \\
	X &\mapsto e^{-v}X
\end{split}
\]
between \((T\Sigma,g)\) and \((T\Sigma,e^{2v}g)\). The map \(b\) induces an isomorphic isometric map on the spinor bundle
\[
\beta\equiv\beta_{g,g_v} : (\mathcal{S}_g\Sigma,h)\to(\mathcal{S}_{g_v}\Sigma,h_v).
\]
The map \(\beta\) is independent of the spinor connection. Under this map, there is a well-known two-dimensional conformal transformation formula as follows
\[
\slashiii{D}_{g_v}\beta\left(e^{-\frac{1}{2}v}\psi\right)=e^{-\frac{3}{2}v}\beta(\slashiii{D}_g\psi).
\]

Next, we introduce the transformation induced by a conformal diffeomorphism from surfaces. Let \(\phi:\Sigma\to\Sigma\) be a diffeomorphism that preserves the orientation and the spinor structure. Let \(\phi^{*}g\) denote the pull-back metric on \(T\Sigma\). Then the tangent map \(T\phi:(T\Sigma,\phi^{*}g)\to(T\Sigma,g)\) is an isometric map, and it also preserves the Levi-Civita connection. There exists an induced map \(\Phi\), which also covers the map \(\phi\), in the sense that the following diagram is commutative:
\[\begin{array}{*{20}{c}}
	{({\mathcal{S} _{{{\phi}^*g}}}\Sigma,{\phi}^*h)}&{\xrightarrow{\Phi }}&{({\mathcal{S} _g}\Sigma,{h})} \\
	\downarrow &{}& \downarrow  \\
	{(\Sigma,{{\phi}^*g})}&{\xrightarrow{{\phi}}}&{(\Sigma,g)}
\end{array}\]
The map \(\Phi\) preserves the spinor connection and is an isometric map of vector bundles, and thus preserves the Dirac operator. Therefore for any \(\psi\in\Gamma(\mathcal{S}_g\Sigma)\), denoting \(\phi^{*}\psi=\Phi^{-1}\circ\psi\circ\phi\in\Gamma(\mathcal{S}_{\phi^{*}g}\Sigma)\) as the pulled-back spinor, then
\begin{equation}\label{2.22}
	\Phi\left(\slashiii{D}_{\phi^{*}g}(\phi^{*}\psi)(x)\right)=\slashiii{D}_g\psi(\phi(x)),\quad x\in\Sigma.
\end{equation}
and
\begin{equation}\label{2.23}
	{\left\langle {{\phi ^*}\psi ,{\phi ^*}\psi } \right\rangle _{\phi^*h}} = {\left\langle {\psi \circ \phi ,\psi \circ \phi } \right\rangle _{h\circ \phi }}.
\end{equation}
For the sake of notational convenience in the sequel, we omit the subscript in the Hermitian inner product \( \langle \cdot, \cdot \rangle _h\) for spinors.
\subsection{Sobolev spaces of spinors}\label{Sobolev-spin}
We now introduce the fractional Sobolev spaces \( H^{1/2}(\mathcal{S}\Sigma) \) for spinors on general Riemann surfaces. For foundational material on Sobolev spaces and fractional Sobolev spaces, we refer to [4, 20].

We first describe the spectrums of the Dirac operator $\slashiii{D}$, denoted by Spec($\slashiii{D}$).
In $L^2({\mathcal{S}\Sigma })$ we can write $\operatorname{Spec}(\slashiii{D})$ as $\operatorname{Spec} (\slashiii{D}) = {\left\{ {{\lambda _k}} \right\}_{k \in {\mathbb{Z}_*}}} \cup \left\lbrace 0\right\rbrace $, where $\mathbb{Z}_*$ = $\mathbb{Z} \backslash\{0\}$.  ${\left( {{\lambda _k}} \right)_{k \in {\mathbb{Z}_*}}}$ are the non-zero eigenvalues. A spinor corresponding to the eigenvalue 0 is referred to as a harmonic spinor.
We denote that the dimension of vector space of harmonic spinors is $h^ 0 $, and put the non-zero eigenvalues in an increasing order (in absolute value) and counted with multiplicities to get
$$
-\infty \leftarrow \cdots \leq \lambda_{-l-1} \leq \lambda_{-l} \leq \cdots \leq \lambda_{-1} < 0 < \lambda_1 \leq \cdots \leq \lambda_k \leq \lambda_{k+1} \leq \cdots \rightarrow+\infty .
$$
Let $\varphi_k$ be the eigenspinors corresponding to $\lambda_k, k \in \mathbb{Z}_*$ with $\left\|\varphi_k\right\|_{L^2(\mathcal{S}\Sigma)}=1$. Similarly, \((\varphi_{0,j})_{1 \le j \le h^0}\) are mutually \(L^2\)-orthogonal harmonic spinors satisfying \(\|\varphi_{0,j}\|_{L^2(\mathcal{S}\Sigma)} = 1\). These eigenspinors together form a complete orthonormal basis of $L^2(\mathcal{S}\Sigma)$ : any spinor $\psi \in \Gamma(\mathcal{S}\Sigma)$ can be expressed in terms of this basis as
$$
\psi=\sum_{k \in \mathbb{Z}_*} a_k \varphi_k+\sum_{1 \leq j \leq h^0} a_{0, j} \varphi_{0, j},
$$
and the Dirac operator acts as
$$
\slashiii{D} \psi=\sum_{k \in \mathbb{Z}_*} \lambda_k a_k \varphi_k.
$$
For any $s>0$, the operator $|\slashiii{D}|^s: \Gamma(\mathcal{S}\Sigma) \rightarrow \Gamma(\mathcal{S}\Sigma)$ is defined as
$$
|\slashiii{D}|^s \psi=\sum_{k \in \mathbb{Z}_*}\left|\lambda_k\right|^s a_k \varphi_k,
$$
provided that the right-hand side belongs to $L^2(\mathcal{S}\Sigma)$. In particular, $|\slashiii{D}|^1$ defined in $L^2(\mathcal{S}\Sigma)$ is abbreviated as $|\slashiii{D}|$ with
$$
\operatorname{Spec}(|\slashiii{D}|)=\left\{\left|\lambda_k \right| : k \in \mathbb{Z}_*\right\}.
$$
Similarly,  the spectrum of $|\slashiii{D}|^{1 / 2}$ is
$$
\operatorname{Spec}\left(|\slashiii{D}|^{1 / 2}\right)=\left\{\left|\lambda_k \right|^{1 / 2} : k \in \mathbb{Z}_*\right\}.
$$
Let
$$
{H^{\frac{1}{2}}}(\mathcal{S}\Sigma) = Dom\left( {{{\left| \slashiii{D} \right|}^{\frac{1}{2}}}} \right) = \left\{ {\psi  \in {L^2}(\mathcal{S}\Sigma)\left| {\int_\Sigma {\left\langle {{{\left| \slashiii{D} \right|}^{\frac{1}{2}}}\psi ,{{\left| \slashiii{D} \right|}^{\frac{1}{2}}}\psi } \right\rangle } {\text{d}{v}} < \infty } \right.} \right\}
$$
denote the domain of the operator $|\slashiii{D}|^{1 / 2}$.  It is clear that $|\slashiii{D}|^{1 / 2}$ is self-adjoint in $L^2(\mathcal{S} \Sigma)$. Define an inner product
\begin{equation}\label{original_norm}
	{\langle \psi ,\phi \rangle _{{H^{\frac{1}{2}}}}} = {\langle \psi ,\phi \rangle _{{L^2}}} + {\left\langle {|\slashiii{D}{|^{\frac{1}{2}}}\psi ,|\slashiii{D}{|^{\frac{1}{2}}}\phi } \right\rangle _{{L^2}}}. \quad \forall \psi, \varphi \in H^{\frac{1}{2}}(\mathcal{S}\Sigma).
\end{equation}
Let $\|\cdot\|_{H^{\frac{1}{2}}(\mathcal{S}\Sigma)}=\langle  \cdot , \cdot \rangle _{{H^{\frac{1}{2}}}(\mathcal{S}\Sigma)}^{\frac{1}{2}}$. Then $(H^{\frac{1}{2}}(\mathcal{S}\Sigma),\|\cdot\|_{H^{\frac{1}{2}}(\mathcal{S}\Sigma)})$ is a Hilbert space. For this space and $q\leq 4$, there exists a continuous embedding
\begin{equation}\label{embedding}
	H^{\frac{1}{2}}(\mathcal{S}\Sigma) \hookrightarrow L^q(\mathcal{S}\Sigma).
\end{equation}
Furthermore, for \( q < 4 \), the embedding is compact, for more details, see, for instance, reference \cite{jevnikar2020existence1,MR4206467}.

Let \( L^{2,+}(\mathcal{S}\Sigma) \) and \( L^{2,-} (\mathcal{S}\Sigma)\) denote the subspaces spanned by $\left\{ {{\varphi _k}:k > 0} \right\}$ and $\left\{ {{\varphi _k}:k < 0} \right\}$ respectively, while $L^{2,0}$ denotes the space of harmonic spinors in $L^2$. If the spin structure of $\Sigma$  is chosen such that $0\notin \operatorname{Spec}(\slashiii{D}_g)$, we can split spinors into the positive and negative parts according to the spectrum of $\slashiii{D}=\slashiii{D}_g$, i.e  we have the decomposition
\begin{equation}\label{2.2}
	H^{\frac{1}{2}}(\mathcal{S}\Sigma)=H^{\frac{1}{2},{+}}(\mathcal{S}\Sigma) \oplus H^{\frac{1}{2},{-}}(\mathcal{S}\Sigma),
\end{equation}
where
$$
H^{\frac{1}{2},{+}}(\mathcal{S}\Sigma)=H^{\frac{1}{2}}(\mathcal{S}\Sigma) \cap L^{2,+}(\mathcal{S}\Sigma),\quad H^{\frac{1}{2},{-}}(\mathcal{S}\Sigma)=H^{\frac{1}{2}}(\mathcal{S}\Sigma) \cap L^{2,-}(\mathcal{S}\Sigma).
$$
Furthermore, if we write $\psi=\psi^{+}+\psi^{-}$, which is decomposed according to (\ref{2.2}), then we have
\begin{equation}\label{spin-eqv-norm1}
	\int_\Sigma\left\langle\slashiii{D} \psi^{+}, \psi^{+}\right\rangle {\text{d}{v}}=\int_\Sigma\left\langle|\slashiii{D}|^{\frac{1}{2}} \psi^{+},|\slashiii{D}|^{\frac{1}{2}} \psi^{+}\right\rangle {\text{d}{v}} \geq \lambda_1\left(\slashiii{D} _g\right)\left\|\psi^{+}\right\|_{L^2(\Sigma)}^2,
\end{equation}
where $\lambda_1$ is the first positive eigenvalue of $\slashiii{D}=\slashiii{D}_g$. Hence
\begin{equation}\label{spin-eqv-norm2}
	\begin{aligned}
		\left\|\psi^{+}\right\|_{H^{\frac{1}{2}}}^2 & =\left\|\psi^{+}\right\|_{L^2}^2+\left\||\slashiii{D}|^{\frac{1}{2}} \psi^{+}\right\|_{L^2}^2
		\leq\left(\lambda_1\left(\slashiii{D}_g\right)^{-1}+1\right)\left\||\slashiii{D}|^{\frac{1}{2}} \psi^{+}\right\|_{L^2}^2\\ &\leq\left(\lambda_1\left(\slashiii{D}_g\right)^{-1}+1\right)\left\|\psi^{+}\right\|_{H^{\frac{1}{2}}}^2.
	\end{aligned}
\end{equation}
That is, for a given $g$, the integral $\int_\Sigma\left\langle\slashiii{D} \psi^{+}, \psi^{+}\right\rangle {\text{d}{v}}$ defines a norm on $H^{\frac{1}{2},+}(\mathcal{S} \Sigma)$ which is equivalent to the Hilbert's. Similarly, on $H^{\frac{1}{2},-}(\mathcal{S} \Sigma)$ there is an equivalent norm given by
$$
-\int_\Sigma\left\langle\slashiii{D} \psi^{-}, \psi^{-}\right\rangle {\text{d}{v}}=\left\||\slashiii{D}|^{\frac{1}{2}} \psi^{-}\right\|_{L^2}^2 .
$$
Consequently,
\begin{equation}\label{2.3}
	\int_\Sigma\left[\left\langle\slashiii{D} \psi^{+}, \psi^{+}\right\rangle-\left\langle\slashiii{D} \psi^{-}, \psi^{-}\right\rangle\right] {\text{d}{v}}=\left\||\slashiii{D}|^{\frac{1}{2}} \psi^{+}\right\|_{L^2}^2+\left\||\slashiii{D}|^{\frac{1}{2}} \psi^{-}\right\|_{L^2}^2
\end{equation}
defines a norm equivalent to the $H^{\frac{1}{2}}(\mathcal{S} \Sigma)$-norm.

It is worthy to mention that the functional \(E: {H^1}({\mathbb{S}^2}) \times {H^{\frac{1}{2}}}(\mathcal{S} {\mathbb{S}^2})\to \mathbb{R} \),
\begin{equation*}
	E\left( {u,\psi } \right) = \int_{{\mathbb{S}^2}} {\left[ {{{\left| {\nabla u} \right|}^2} + 2u - h_1{e^{2u}} + 2\left( {\left\langle {\slashiii{D}\psi ,\psi } \right\rangle  - h_2{e^{u}}{{\left| \psi  \right|}^2}} \right)} \right]{\text{d}{v}_{g_0}}}  + 4\pi
\end{equation*}
is well defined and is of class $C^1$. By using the similar arguments of Jost et al. \cite{jost2007super}, the critical points of $E(u,\psi)$, which are weak solutions of (\ref{key-Sphere}), are actually smooth.
\subsection{Spinors on the sphere}
Finally, in particular, we provide a brief introduction to spinors on the sphere. Consider the spin bundle \( \mathcal{S}\mathbb{S}^2 \to \mathbb{S}^2 \) associated with the unique spin structure on \( \mathbb{S}^2 \), and let \( \slashiii{D}_{g_0} \) denote the Dirac operator on \( (\mathbb{S}^2,g_0) \). The spectrum of the Dirac operator \( \slashiii{D}_{g_0} \) on \( \mathbb{S}^2 \) consists of the eigenvalues \( \pm(k + 1) \) for \( k \in \mathbb{N} \), with the (real) multiplicity of the eigenvalues \( \pm(k + 1) \) being \( 4(k + 1) \). Regarding harmonic spinors on the sphere, the Lichnerowicz formula yields the following celebrated result.
\begin{theorem}[Corollary 8.9, \cite{MR1031992}]\label{Lich}
	Any compact spin manifold of positive scalar curvature admits no harmonic spinors. In fact, the same conclusion holds if the scalar curvature is $\geq$ 0 and $>$ 0 at some points.
\end{theorem}
In particular, there are no harmonic spinor fields on \( \left( \mathbb{S}^2,g_0 \right) \) hence
${h^0} = \dim {\text{ker}}(\slashiii{D}_{g_0}) = 0.$	And the first positive eigenvalue is \( 1 \), with the corresponding eigen-spinors given by Killing spinors of constant length. According to Friedrich \cite{MR1653146}, the existence of such spinors with constant length is related to the embedding of the surface into three-dimensional Euclidean space. For further details, we refer to \cite{ammann2003variational,MR1777332,MR2509837,MR1031992}.
\section{Conformal transformation and an obstruction}
Let \(Aut(\mathbb{S}^2)\) denote the group of conformal diffeomorphisms from \( \mathbb{S}^2 \) to \( \mathbb{S}^2 \), then there is \( Aut(\mathbb{S}^2) = PSL(2,\mathbb{C}) \) on the Riemann surface $\mathbb{S}^2=\mathbb{C}\cup \{\infty\}$.  Consider a conformal mapping \( \phi \in Aut(\mathbb{S}^2) \), then we have \( {\phi ^*}{g_0} = \det (d\phi) {g_0} \). Let $v=\frac{1}{2}\log \det \left(d \phi\right) $, then \( {\phi ^*}{g_0} = e^{2v} {g_0} \). We employ the notation introduced in Section 2.2. Under the transformation, we have the following commutative diagram
\[\begin{array}{*{20}{c}}
	{({\mathcal{S}_{g_0}}\Sigma ,h)}&{\xrightarrow{\beta }}&{({\mathcal{S}_{{\phi ^*}g_0}}\Sigma ,{\phi ^*}h)}&{\xrightarrow{\Phi }}&{({\mathcal{S}_{g_0}}\Sigma ,h)} \\
	\downarrow &{}& \downarrow &{}& \downarrow  \\
	{(\Sigma ,{g_0})}&{\xrightarrow{{{\text{id}}}}}&{(\Sigma ,{\phi ^*}{g_0} = {e^{2v}}{g_0})}&{\xrightarrow{\phi }}&{(\Sigma ,{g_0})}
\end{array}\]
For the pair \((u, \psi)\) in equation (\ref{key-Sphere}), we perform the following transformation
\begin{equation}\label{trans}
	\begin{cases}
		{{u_\phi }\left( x \right): = u \circ \phi \left( x \right) + v\left( x \right),} \\
		{{\psi _\phi }\left( x \right): = {\beta ^{ - 1}}\left( {{e^{\frac{1}{2}v\left( x \right)}}{\Phi ^{ - 1}} \circ \psi  \circ \phi \left( x \right)} \right),}
	\end{cases}
\end{equation}
then by (\ref{2.22}) and (\ref{2.23}) we have
\[\begin{aligned}
	- {\Delta _{{g_0}}}{u_\phi } =&  - {\Delta _{{g_0}}}u \circ \phi  - {\Delta _{{g_0}\left( x \right)}}v
	=  - {e^{2v}}{\Delta _{{\phi ^*}{g_0}}}\left( {{\phi ^*}u} \right) + {e^{2v}} - 1  \\
	= &- {e^{2v}}\left( {{\Delta _{{g_0}}}u} \right) \circ \phi  + {e^{2v}} - 1
	= {e^{2v}}\left( {h_1{e^{2u}} - 1 + h_2{e^u}\left| \psi  \right|_h^2} \right) \circ \phi  + {e^{2v}} - 1  \\
	=& \left( {h_1 \circ \phi } \right){e^{2u \circ \phi  + 2v}} - 1 +\left(h_2\circ \phi\right) {e^{u \circ \phi  + v}}{e^v}\left| {{\Phi ^{ - 1}}\psi } \right|_h^2 \\
	=& \left( {h_1 \circ \phi } \right){e^{2{u_\phi }}} - 1 + \left(h_2\circ \phi\right){e^{{u_\phi }}}\left| {{e^{\frac{v}{2}}}{\beta ^{ - 1}}{\Phi ^{ - 1}}\psi } \right|_h^2  \\
	= &\left( {h_1 \circ \phi } \right){e^{2{u_\phi }}} - 1 +\left(h_2\circ \phi\right) {e^{{u_\phi }}}\left| {{\beta ^{ - 1}}{e^{\frac{v}{2}}}{\Phi ^{ - 1}}\psi } \right|_h^2  \\
	= &\left( {h_1 \circ \phi } \right){e^{2{u_\phi }}} - 1 +\left(h_2\circ \phi\right) {e^{{u_\phi }}}{\left| {{\psi _\phi }} \right|^2} ,
\end{aligned} \]
and
\[\begin{aligned}
	{\slashiii{D}_{{g_0}}}{\psi _\phi }(x) =& {\slashiii{D}_{{g_0}}}\left( {{\beta ^{ - 1}}{e^{\frac{1}{2}v}}{\Phi ^{ - 1}}\left( {\psi  \circ \phi } \right)} \right)  \\
	=& {\beta ^{ - 1}}{e^{\frac{3}{2}v}}{\slashiii{D}_{{\phi ^*}{g_0}}}{e^{ - \frac{1}{2}v}}\beta \left( {{\beta ^{ - 1}}{e^{\frac{1}{2}v}}\left( {{\Phi ^{ - 1}}\left( {\psi  \circ \phi } \right)} \right)} \right)  \\
	=& {\beta ^{ - 1}}{e^{\frac{3}{2}v}}{\slashiii{D}_{{\phi ^*}{g_0}}}{\phi ^*}\psi  = {e^{\frac{3}{2}v}}{\beta ^{ - 1}}{\Phi ^{ - 1}}\left( {\left( {{\slashiii{D}_{{g_0}}}\psi } \right) \circ \phi } \right)  \\
	=& {e^{\frac{3}{2}v}}{\beta ^{ - 1}}\left( {{\Phi ^{ - 1}}\left( {h_2{e^u}\psi } \right) \circ \phi } \right)  = \left(h_2\circ \phi\right){e^{u \circ \phi  + v}}{\beta ^{ - 1}}{e^{\frac{1}{2}v}}{\Phi ^{ - 1}}\left( {\psi  \circ \phi } \right)  \\
	=& \left(h_2\circ \phi\right){e^{{u_\phi }}}{\psi _\phi }
\end{aligned} \]
In summary, \((u_{\phi}, \psi_{\phi})\) satisfies \begin{equation}
	\begin{cases}
		{ - \Delta_{g_0} u_\phi = {(h_1\circ\phi)}{e^{2u_\phi}} - 1 + \left(h_2\circ \phi\right)}{e^{u_\phi}{{\left| \psi_\phi  \right|}^2},}&{{\text{on}}\;{\mathbb{S}^2},} \\
		{\slashiii{D}_{g_0}\psi_\phi  =\left(h_2\circ \phi\right)  {e^{u_\phi}}\psi_\phi ,}&{{\text{on}}\;{\mathbb{S}^2}.}
	\end{cases}
\end{equation}
Next, we employ M\"obius transformations on the sphere to derive a Pohozaev-type identity.

\

\noindent\textbf{Proof of Proposition \ref{Kazdan-w}. }We employ the method in \cite{MR1131392,changMR0908146}. First, it is known that $\int_{{\mathbb{S}^2}} \left( |\nabla u|^2 + 2u \right) \text{d}v_{g_0}$ is invariant under M\"obius transformations. A straightforward computation
\[
\begin{aligned}
	{\int_{{\mathbb{S}^2}} {\left\langle {{\slashiii{D}_{{g_0}}}{\psi _\phi },{\psi _\phi }} \right\rangle } _h}{\text{d}{v}_{g_0}}  =& {\int_{{\mathbb{S}^2}} {\left\langle {{e^{\frac{3}{2}v}}{\beta ^{ - 1}}{\Phi ^{ - 1}}\left( {{\slashiii{D}_{{g_0}}}\psi } \right) \circ \phi ,{\beta ^{ - 1}}{e^{\frac{1}{2}v}}{\Phi ^{ - 1}} \circ \psi  \circ \phi } \right\rangle } _h}{e^{ - 2v}}{e^{2v}}{\text{d}{v}_{g_0}}  \\
	=& {\int_{{\mathbb{S}^2}} {\left\langle {{\beta ^{ - 1}}{\Phi ^{ - 1}}\left( {{\slashiii{D}_{{g_0}}}\psi } \right) \circ \phi ,{\beta ^{ - 1}}{\Phi ^{ - 1}} \circ \psi  \circ \phi } \right\rangle } _h}{e^{2v}}{\text{d}{v}_{g_0}}  \\
	=& {\int_{{\mathbb{S}^2}} {\left\langle {{\Phi ^{ - 1}}\left( {{\slashiii{D}_{{g_0}}}\psi } \right) \circ \phi ,{\Phi ^{ - 1}} \circ \psi  \circ \phi } \right\rangle } _h}\text{d}{v_{{\phi ^*}{g_0}}}  \\
	=& {\int_{{\mathbb{S}^2}} {\left\langle {{\slashiii{D}_{{g_0}}}\psi ,\psi } \right\rangle } _h}{\text{d}{v}_{g_0}}
\end{aligned} \]
shows that $\int_{\mathbb{S}^2} \left\langle {\slashiii{D}_{g_0}} \psi, \psi \right\rangle \, \text{d}v_{g_0}$ is also  M\"obius conformally invariant. Let \( \pi_Q \) denote the stereographic projection from $\mathbb{S}^2$ to \( \mathbb{C}^2 \) with \( Q \) as the north pole, and let \( \tau_t \) be a dilation transformation from \( \mathbb{C}^2 \) to \( \mathbb{C}^2 \) with $\tau_t=tz,\ z\in \mathbb{C}^2$. Define ${\phi _{Q,t}} = {\pi _Q}^{ - 1} \circ t \circ {\pi _Q}$, which  is a conformal map from $\mathbb{S}^2$ to $\mathbb{S}^2$. Therefore, we have form that $(u,\psi) $ is a solution of equation (\ref{key-Sphere})
\[\begin{aligned}
	0 =& \left\langle {dE\left( {{u},{\psi }} \right),{{\left. {\frac{d}{{dt}}} \right|}_{t = 1}}\left( {{u_{{\phi _{Q,t}}}},{\psi _{{\phi _{Q,t}}}}} \right)} \right\rangle
	= {\left. {\frac{d}{{dt}}} \right|_{t = 1}}E\left( {{u_{{\phi _{Q,t}}}},{\psi _{{\phi _{Q,t}}}}} \right) \hfill \\
	=& {\left. { - \frac{d}{{dt}}} \right|_{t = 1}}\left( {\int_{{\mathbb{S}^2}} {\left( {{h_1}\circ \phi _{Q,t}^{ - 1}} \right){e^{2{u}}}{\text{d}}v + 2\int_{{\mathbb{S}^2}} {\left( {{h_2}\circ \phi _{Q,t}^{ - 1}} \right){e^{{u}}}{{\left| {{\psi }} \right|}^2}{\text{d}}v} } } \right) \hfill \\
	=&  - \left( {\int_{{\mathbb{S}^2}} {{{\left. {\frac{d}{{dt}}} \right|}_{t = 1}}\left( {{h_1}\circ \phi _{Q,t}^{ - 1}} \right){e^{2{u}}}{\text{d}}v}  + 2\int_{{\mathbb{S}^2}} {\frac{d}{{dt}}\left( {{h_2}\circ \phi _{Q,t}^{ - 1}} \right){e^{{u}}}{{\left| {{\psi }} \right|}^2}{\text{d}}v} } \right) \hfill \\
	=&  - \int_{{\mathbb{S}^2}} {\left\langle {\nabla {h_1},\nabla \left\langle {\vec x \cdot Q} \right\rangle } \right\rangle {e^{2{u}}}{\text{d}}v}  - 2\int_{{\mathbb{S}^2}} {\left\langle {\nabla {h_2},\nabla \left\langle {\vec x \cdot Q} \right\rangle } \right\rangle {e^{{u}}}{{\left| {{\psi }} \right|}^2}{\text{d}}v} .
\end{aligned} \]
The proof is completed by taking \( Q = e_1, e_2, e_3 \).\qed
\section{Bounds and estimates for compactness}
In this section, we primarily perform estimates on the solutions and utilize them to establish the compactness results.

\

\noindent\textbf{Proof of Theorem \ref{thm-principle}.} Notice that the weak solutions of (\ref{general}) are smooth. Let \( (u,\psi) \) be a solution of the equation. Under the following transformation
\[
g_u = e^{2u}g,\qquad \tilde{\psi} = e^{-u/2}\beta\left( \psi\right) ,
\]
equation (\ref{general}) transforms into
\begin{equation}\label{conformal-general}
	\begin{cases}
		K_{g_u} = h_1 + h_2 \, \left| {\tilde \psi } \right|^2, & \text{on } M, \\[4pt]
		\slashiii{D}_{g_u} \tilde{\psi} = h_2 \tilde{\psi}, & \text{on } M.
	\end{cases}
\end{equation}
In dimension two we have the following Lichnerowicz formula from (\ref{Lichnerowicz})
\begin{equation}\label{Lichnerowicz-Gauss}
	\slashiii{D}_{g_u}^2 \tilde{\psi} = \nabla_{g_u}^*\nabla_{g_u}\tilde{\psi} + \frac{K_{g_u}}{2}\tilde{\psi}.
\end{equation}
Substituting the Dirac equation \( \slashiii{D}_{g_u}\tilde{\psi}=h_2\tilde{\psi} \) into (\ref{Lichnerowicz-Gauss}) gives
\[
\nabla_{g_u}^*\nabla_{g_u}\tilde{\psi}
= \slashiii{D}_{g_u}(h_2\tilde{\psi}) - \frac{K_{g_u}}{2}\tilde{\psi}
= h_2^2\tilde{\psi} + (\nabla_{g_u}h_2)\cdot\tilde{\psi} - \frac{K_{g_u}}{2}\tilde{\psi}.
\]
Using the expression for \( K_{g_u} \) from (\ref{conformal-general}), we obtain
\begin{equation}\label{Laplace-spin}
	-\Delta_{g_u}\tilde{\psi} + \frac{h_1+h_2\left| {\tilde \psi } \right|^2}{2}\,\tilde{\psi}
	= h_2^2\tilde{\psi} + (\nabla_{g_u}h_2)\cdot\tilde{\psi}.
\end{equation}

Define
\[
f := \frac12 \left| {\tilde \psi } \right|^2.
\]
Applying the Laplacian to \( f \) and using (\ref{Laplace-spin}), we have
\begin{equation}\label{struc-control}
	\begin{aligned}
		-\Delta_{g_u} f
		&= \left\langle { -\Delta_{g_u}\tilde{\psi},\,\tilde{\psi}} \right\rangle
		- |\nabla_{g_u}\tilde{\psi}|^2 \\
		&\le \left\langle { -\Delta_{g_u}\tilde{\psi},\,\tilde{\psi}} \right\rangle  \\
		&= \left\langle
		-\frac{h_1+h_2\left| {\tilde \psi } \right|^2}{2}\tilde{\psi}
		+ h_2^2\tilde{\psi} + (\nabla_{g_u}h_2)\cdot\tilde{\psi},
		\;\tilde{\psi}\right\rangle .
	\end{aligned}
\end{equation}
Using the Clifford relation (\ref{Clifford}) and its compatibility with the inner product (\ref{Clofford-inner-product}), we obtain
\[\begin{aligned}
	{\left\langle {X \cdot \psi ,\psi } \right\rangle _h} =& \frac{1}{{g\left( {X,X} \right)}}{\left\langle {X \cdot X\psi ,X \cdot \psi } \right\rangle _h} \hfill \\
	=& \frac{1}{{g\left( {X,X} \right)}}{\left\langle { - g\left( {X,X} \right)\psi ,X \cdot \psi } \right\rangle _h} \hfill \\
	=&  - {\left\langle {\psi ,X \cdot \psi } \right\rangle _h} .
\end{aligned} \]
Therefore $\left\langle {\left( {{\nabla _{{g_u}}}{h_2}} \right) \cdot \psi ,\psi } \right\rangle  = 0$ and
\begin{equation}\label{laplace}
	- {\Delta _{{g_u}}}f + \left( {{h_1} + {h_2}{{\left| {\tilde \psi } \right|}^2}} \right)f \leq 2h_2^2f.
\end{equation}

Let \(x_0\in M\) be a point where \(f\) attains its maximum. At such a point we necessarily have
\[
-\Delta_{g_u} f\left( {{x_0}} \right) \ge 0.
\]
Insert this inequality into (\ref{laplace}) evaluated at \(x_0\) to get
\[
\left( {{h_1}\left( {{x_0}} \right) + {h_2}\left( {{x_0}} \right){{\left| {\tilde \psi \left( {{x_0}} \right)} \right|}^2}} \right)f\left( {{x_0}} \right) \leq 2 {h_2^2\left( {{x_0}} \right)} f\left( {{x_0}} \right).
\]
Since \(f\left( {{x_0}} \right)>0\) unless \(\tilde{\psi}\equiv0\), we may divide by \(f\left( {{x_0}} \right)\) and rearrange to obtain
\[
h_1\left( {{x_0}} \right)+h_2\left( {{x_0}} \right){\left| {\tilde \psi \left( {{x_0}} \right)} \right|}^2
\le 2h_2^2\left( {{x_0}} \right).
\]
Consequently,
\begin{equation}\label{change-control}
	{\left| {\tilde \psi \left( {{x}} \right)} \right|}^2\le {\left| {\tilde \psi \left( {{x_0}} \right)} \right|}^2
	\le \frac{2h_2^2\left( {{x_0}} \right)-h_1\left( {{x_0}} \right)}{h_2\left( {{x_0}} \right)}.
\end{equation}
Recalling that \(\left| {\tilde \psi } \right|^2 = e^{-u}\left| { \psi } \right|^2\), inequality (\ref{change-control}) implies
\begin{equation}\label{orig-control}
	\left| { \psi } \right|^2 \le
	\Biggl\{\frac{2h_2^2\left( {{x_0}} \right)-h_1\left( {{x_0}} \right)}{h_2\left( {{x_0}} \right)}\Biggr\}\,e^{u}.
\end{equation}
Thus the spinor \(\psi\) is pointwise bounded by an expression involving \(u\) and the values of \(h_1\) and \(h_2\). This yields the desired control of \(\psi\) by the scalar function \(u\).\qed
\begin{remark}
	Recall the Sobolev embedding theorem (\ref{embedding}) for spinors, where \(q=4\) is the critical integrable exponent. By Theorem \ref{thm-principle}, \(\|\psi\|_{L^4}\) is controlled by \(\int_M e^{2u}{\text{d}{v}}\). When \(M\) is the sphere \(\mathbb{S}^2\), as can be seen from the following computation, the quantity \(\int_{\mathbb{S}^2} e^{2u}{\text{d}{v}}\) plays a particularly important role in the estimates.
\end{remark}

Next we employ the spectral properties of the Dirac operator to derive estimates for solutions to equation (\ref{general}) on the sphere $\mathbb{S}^2$, with particular attention to the positive and negative parts of the spinor separately, thereby obtaining a uniform Sobolev bound for the spinor component of solutions.

\

\noindent\textbf{Proof of Theorem \ref{thm1}.}  Integrate both sides of the first equation
\begin{equation*}
	{ - {\Delta _g}u = h_1(x){e^{2u}} - {K_g} + h_2(x){e^u}{{\left| \psi  \right|}^2}}
\end{equation*}
of (\ref{general}), and apply Gauss-Bonnet formula to yield
\begin{equation}\label{stokes}
	\int_{{\mathbb{S}^2}} {\left( {{h_1}{e^{2u}} + h_2{e^{u}}{{\left| \psi  \right|}^2}} \right)} {\text{d}{v}} = 4\pi .
\end{equation}
Therefore we have from $h_1>0$
\begin{equation}\label{ou-he}
	0 \leq \int_{{\mathbb{S}^2}} h_2{{e^u}} {\left| \psi  \right|^2}{\text{d}}v \leq 4\pi,\quad \int_{{\mathbb{S}^2}} {{e^{2u}}}  {\text{d}}v < C\left(h_1 \right) .
\end{equation}
Multiply both sides of the second equation
\begin{equation*}
	\slashiii{D}_g\psi  = h_2{e^{u}}\psi.
\end{equation*}
by the spinor $\psi$, and then take integration to obtain
\begin{equation}\label{nehari}
	\int_{{\mathbb{S}^2}} {\left\langle {\slashiii{D}_g\psi ,\psi } \right\rangle } {\text{d}{v}} = \int_{{\mathbb{S}^2}}h_2 {{e^{u}}{{\left| \psi  \right|}^2}} {\text{d}{v}}.
\end{equation}
Since there are no harmonic spinors on the sphere, the spectral decomposition allows us to write \( \psi = \psi^+ + \psi^- \). From (\ref{spin-eqv-norm1}) and (\ref{spin-eqv-norm2}), it follows that there exist positive constants $A_1$ and $A_2$ such that
\begin{equation}\label{pos-spin}
	{A_1}\left\| {{\psi ^ + }} \right\|_{{H^{1/2}}}^2 \leq \int_{{{\mathbb{S}^2}}} {\left\langle {\slashiii{D}_g{\psi ^ + },{\psi ^ + }} \right\rangle } {\text{d}}v \leq {A_2}\left\| {{\psi ^ + }} \right\|_{{H^{1/2}}}^2.
\end{equation}
Similarly, there exist positive constants \( B_1 \) and \( B_2 \) such that
\begin{equation}\label{neg-spin}
	-{B_1}\left\| {{\psi ^ - }} \right\|_{{H^{1/2}}}^2 \leq \int_{{{\mathbb{S}^2}}} {\left\langle {\slashiii{D}_g{\psi ^ - },{\psi ^ - }} \right\rangle } {\text{d}}v \leq {-B_2}\left\| {{\psi ^ - }} \right\|_{{H^{1/2}}}^2.
\end{equation}
Then by (\ref{ou-he}) and (\ref{nehari}) we have
\[\begin{aligned}
	{A_1}\left\| {{\psi ^ + }} \right\|_{{H^{1/2}}}^2 - {B_1}\left\| {{\psi ^ - }} \right\|_{{H^{1/2}}}^2
	&\leq \int_{{{\mathbb{S}^2}}} {\left\langle {\slashiii{D}_g{\psi ^ + },{\psi ^ + }} \right\rangle } {\text{d}}v + \int_{{{\mathbb{S}^2}}} {\left\langle {\slashiii{D}_g{\psi ^ - },{\psi ^ - }} \right\rangle } {\text{d}}v  \\
	&= \int_{{\mathbb{S}^2}} {\left\langle {\slashiii{D}_g\psi ,\psi } \right\rangle } {\text{d}}v = \int_{{\mathbb{S}^2}} h_2{{e^u}} {\left| \psi  \right|^2}{\text{d}}v \leq 4\pi,
\end{aligned}
\]
\[\begin{aligned}
	0 \leq \int_{{\mathbb{S}^2}} h_2{{e^u}} {\left| \psi  \right|^2}{\text{d}}v = \int_{{\mathbb{S}^2}} {\left\langle {\slashiii{D}_g\psi ,\psi } \right\rangle } {\text{d}}v &= \int_{{{\mathbb{S}^2}}} {\left\langle {\slashiii{D}_g{\psi ^ + },{\psi ^ + }} \right\rangle } {\text{d}}v + \int_{{{\mathbb{S}^2}}} {\left\langle {\slashiii{D}_g{\psi ^ - },{\psi ^ - }} \right\rangle } {\text{d}}v  \\
	&\leq {A_2}\left\| {{\psi ^ + }} \right\|_{{H^{1/2}}}^2 - {B_2}\left\| {{\psi ^ - }} \right\|_{{H^{1/2}}}^2.
\end{aligned} \]
Therefore,
\begin{equation}\label{pos-neg}
	B_2{\left\| {{\psi ^ - }} \right\|^2_{{H^{1/2}}}} \leq A_2{\left\| {{\psi ^ + }} \right\|^2_{{H^{1/2}}}} \leq \frac{4\pi A_2}{A_1}  + \frac{A_2B_1}{A_1}{\left\| {{\psi ^ - }} \right\|^2_{{H^{1/2}}}}.
\end{equation}
Multiply both sides of the second equation of (\ref{general}) by the spinor $\psi^+$, and then take integration to obtain
\begin{equation}\label{nonlinear}
	\begin{aligned}
		A_1{\left\| {{\psi ^ + }} \right\|^2_{{H^{1/2}}}}  &\leq \int_{{{\mathbb{S}^2}}} {\left\langle {\slashiii{D}_g\psi ,{\psi ^ + }} \right\rangle } {\text{d}}v = \int_{{\mathbb{S}^2}}h_2 {{e^{u}}\left\langle {\psi ,{\psi ^ + }} \right\rangle } {\text{d}{v}} \\
		&\leq\int_{{\mathbb{S}^2}}h_2 {{e^{u}}\left| \psi  \right|\left| {{\psi ^ + }} \right|} {\text{d}{v}} \\
		&\leq {\left( {\int_{{\mathbb{S}^2}} h_2{{e^{u}}} {{\left| {{\psi ^ + }} \right|}^2}{\text{d}{v}}} \right)^{\frac{1}{2}}}{\left( {\int_{{\mathbb{S}^2}}h_2 {{e^{u}}{{\left| \psi  \right|}^2}} {\text{d}{v}}} \right)^{\frac{1}{2}}} \\
		&\leq C\left( h_2\right) {\left( {\int_{{\mathbb{S}^2}} {{e^{2u}}} {\text{d}{v}}} \right)^{\frac{1}{4}}}{\left( {\int_{{\mathbb{S}^2}} {{{\left| {{\psi ^ + }} \right|}^4}{\text{d}{v}}} } \right)^{\frac{1}{4}}}{\left( {\int_{{\mathbb{S}^2}}h_2 {{e^{u}}{{\left| \psi  \right|}^2}} {\text{d}{v}}} \right)^{\frac{1}{2}}} \\
		&\leq C(h_1,h_2)\left\| {{\psi ^ + }} \right\|_{{H^{1/2}}}.
	\end{aligned}
\end{equation}
The last inequality follows from (\ref{ou-he}) and the Sobolev embedding inequality. As a result, we have
\begin{equation*}
	{\left\| \psi  \right\|^2_{{H^{1/2}}}}=\left\| {{\psi ^ + }} \right\|_{{H^{1/2}}}^2 + \left\| {{\psi ^ - }} \right\|_{{H^{1/2}}}^2 \leq C.
\end{equation*}
Thus, we have completed the proof.\qed

\

\noindent\textbf{Proof of Theorem \ref{DE-th}.} Given $R_g>0$, we utilize the equivalent norm of ${H^1(\mathbb{S}^3)}$, which is
\[
\|u\|_{H^1(\mathbb{S}^3)}^2 := \int_{\mathbb{S}^3} |\nabla u|^2 \, \text{d}v_g + \int_{\mathbb{S}^3} \frac{R_g}{8} u^2 \, \text{d}v_g.
\]
Testing the first equation in \eqref{DE} by $u$ yields the identity
\begin{equation}\label{DE_first_eq}
	\int_{\mathbb{S}^3} \left( |\nabla_g u|^2 + \frac{R_g}{8} u^2 \right)\text{d}v_g = \|u\|_{H^1}^2 = \int_{\mathbb{S}^3} K u^2 |\psi|^2\text{d}v_g,
\end{equation}
which directly implies that
\begin{equation}\label{3-coup}
	\int_{\mathbb{S}^3} u^2 |\psi|^2 \, \text{d}v_g \leq C \|u\|_{H^1}^2.
\end{equation}
Next, testing the Dirac equation with $\psi$ we obtain
\[ \int_{{\mathbb{S}^3}} {\left\langle {{\slashiii{D}_g}\psi ,\psi } \right\rangle } {\mkern 1mu} {\text{d}}v= \int_{\mathbb{S}^3} K u^2 |\psi|^2 \text{d}v_g= \|u\|_{H^1}^2. \]
By decomposing the spinor $\psi$ into its positive and negative spectral parts, $\psi = \psi^+ + \psi^-$, we obtain the following control from \eqref{DE_first_eq} and (\ref{pos-spin})
\begin{equation}\label{u-psi}
	\left\| u \right\|_{{H^1}}^2 \leq C\left\| {{\psi ^ + }} \right\|_{{H^{1/2}}}^2,
\end{equation}
and
\begin{equation}\label{psi-psi}
	\left\| {{\psi ^ - }} \right\|_{{H^{1/2}}}^2 \leq C\left\| {{\psi ^ + }} \right\|_{{H^{1/2}}}^2.
\end{equation}			
To establish the reverse inequality, we test the Dirac equation with $\psi^+$:
\[
\begin{aligned}
	C \|\psi^+\|^2_{H^{1/2}} &= \int_{\mathbb{S}^3} K u^2 \langle \psi, \psi^+ \rangle \text{d}v_g\leq C \int_{\mathbb{S}^3} u^2 |\psi| |\psi^+| \text{d}v_g\\
	&\leq C \left( \int_{\mathbb{S}^3} u^2 |\psi|^2 \text{d}v_g\right)^{1/2} \left( \int_{\mathbb{S}^3} u^2 |\psi^+|^2\text{d}v_g \right)^{1/2}.
\end{aligned}
\]
Invoking the Sobolev embeddings $H^1(\mathbb{S}^3) \hookrightarrow L^6(\mathbb{S}^3)$ and $H^{1/2}(\Sigma\mathbb{S}^3) \hookrightarrow L^3(\Sigma\mathbb{S}^3)$, we estimate the last term
\[ \left( \int_{\mathbb{S}^3} u^2 |\psi^+|^2 \text{d}v_g\right)^{1/2} \leq \|u\|_{L^6} \|\psi^+\|_{L^3} \leq C \|u\|_{H^1} \|\psi^+\|_{H^{1/2}}. \]
Substituting this estimate and \eqref{3-coup} back into the inequality, we find:
\[ \|\psi^+\|_{H^{1/2}}^2 \leq C \|u\|_{H^1} \cdot \|u\|_{H^1} \|\psi^+\|_{H^{1/2}} = C \|u\|_{H^1}^2 \|\psi^+\|_{H^{1/2}}, \]
which implies $\|\psi^+\|_{H^{1/2}} \leq C \|u\|_{H^1}^2$. Combining this with \eqref{psi-psi}, we deduce
\[{\left\| \psi  \right\|_{{H^{1/2}}}} \leq C\left\| u \right\|_{{H^1}}^2.\]
Furthermore, as $u\not \equiv 0$ and there are no harmonic spinors, from \eqref{u-psi} we obtain
\[{\left\| \psi  \right\|_{{H^{1/2}}}} \geq C\quad and\quad {\left\| u \right\|_{{H^1}}} \geq C.\]
Finally, by HÃ¶lder's inequality and the Sobolev embedding:
\[ C\leq\|u\|_{H^1}^2 = \int_{\mathbb{S}^3} K u^2 |\psi|^2 \text{d}v_g\leq C \|u\|_{L^6}^2 \|\psi\|_{L^3}^2 \leq C \|u\|_{H^1}^2 \|\psi\|_{H^{1/2}}^2. \]
This concludes the proof, establishing that $\frac{C}{\|u\|_{H^1}^2} \leq \|\psi\|_{H^{1/2}}^2 \leq C \|u\|_{H^1}^4$.
\section{Compactness results under certain conditions}
In this section, we establish compactness results for the super-Liouville equation under certain assumptions. A key feature of the equation is the coupling term ${e^u}{\left| \psi \right|^2}$, which critically links the scalar and spinor components, a fact that is also reflected in the detailed calculations and estimates. For clarity in the subsequent discussion, we provide here the relevant indices associated with this coupling term.
\begin{lemma}\label{holder}
	Under the condition of Theorem \ref{thm1}, if $\left( u,\psi \right) $ is a solution to (\ref{general}) and ${\left\| {{e^{u}}} \right\|_{{L^\alpha }\left( {{\mathbb{S}^2}} \right)}} \leq C,\;\alpha  > 1$, then there exists a $C'$ such that for any $0\leq \beta  \leq \frac{{2\alpha }}{{ \alpha+2 }},$
	\[
	{\left\| {{e^{u}}{{\left| \psi  \right|}^2}} \right\|_{{L^\beta }\left( {{\mathbb{S}^2}} \right)}} \leq C'.
	\]
	In particular, since we have already known \( \int_{{\mathbb{S}^2}} {{e^{2u}}} \mathrm{d}v \leq \frac{4\pi}{{{h_{1,\min }}}} \), then we can take $\alpha=2$ to get $\beta=1$, i.e. we have \( \int_{{\mathbb{S}^2}} {{e^{u}}{{\left| \psi  \right|}^2}}{\mathrm{d}v}  \leq C \).
\end{lemma}
\begin{proof}
	According to the embedding theorem of spinors, the spinor parts of all solutions are bounded in $L ^ 4(\mathbb{S}^2) $:\[{\left\| \psi  \right\|_{{L^4({\mathbb{S}^2})}}} \leq C.\]
	Let $p = \frac{{2 + \alpha }}{2}$ and $q = \frac{{2 + \alpha }}{\alpha }$. By H{\"o}lder's inequality we have
	\[\begin{aligned}
		\int_{{\mathbb{S}^2}} {{{\left( {{e^{u}}{{\left| \psi  \right|}^2}} \right)}^{\frac{{2\alpha }}{{2 + \alpha }}}}}{\text{d}{v}}  &\leq {\left( {\int_{{\mathbb{S}^2}} {{e^{\left( {\frac{{2\alpha }}{{2 + \alpha }}} \right)pu}}} }{\text{d}{v}} \right)^{1/p}}{\left( {\int_{{\mathbb{S}^2}} {{{\left| \psi  \right|}^{2\left( {\frac{{2\alpha }}{{2 + \alpha }}}\right)q}}} {\text{d}{v}}} \right)^{^{1/q}}}  \\
		&\leq {\left( {\int_{{\mathbb{S}^2}} {{e^{\alpha u}}} } {\text{d}{v}}\right)^{1/p}}{\left( {\int_{{\mathbb{S}^2}} {{{\left| \psi  \right|}^4}} } {\text{d}{v}}\right)^{1/q}} \leq C  .
	\end{aligned} \]
\end{proof}
\begin{remark}
	For the linear Dirac equation for spinors:
	\[
	\slashiii{D}\psi = \eta, \quad \text{on } \mathbb{S}^2,
	\]
	the space \(L^1\) remains a critical integrable space. Consequently, it is not possible to obtain further estimates via a bootstrap argument, see \cite[Lemma3.2.2, Theorem 3.2.1]{ammann2003variational} for details.
\end{remark}

Next, we prove Theorem \ref{thm1.3}, i.e. show the compactness of the solutions space of (\ref{key-Sphere}) under the small energy  condition.  Recall that $(u_n,\psi_n)$  satisfies equation (\ref{jinsi}) as  follows:
\begin{equation*}
	\begin{cases}
		{ - \Delta u_n = {h_{1,n}}(x){e^{2u_n}} - 1 + {h_{2,n}}{e^{u_n }}{{\left| \psi_n   \right|}^2},}&{{\text{on}}\;{\mathbb{S}^2},} \\
		{\slashiii{D}\psi_n   ={h_{2,n}} {e^{u_n}}\psi_n  ,}&{{\text{on}}\;{\mathbb{S}^2}.}
	\end{cases}
\end{equation*}

\

\noindent\textbf{Proof of Theorem \ref{thm1.3}.} We use the method in the proof of Proposition 1.4 by Brendle \cite{MR1999924}.
Firstly, by Jensen's inequality and Stokes' theorem, we have
\[
4\pi m_1{e^{2{{\bar u}_n}}} \leq m_1\int_{{\mathbb{S}^2}} {{e^{2{u_n}}}} {\text{d}{v}} \leq \int_{{\mathbb{S}^2}} {\left( {{h_{1,n}}{e^{2{u_n}}}{\mkern 1mu}  + {h_{2,n}}{e^{{u_n}}}{{\left| {{\psi _n}} \right|}^2}} \right)} {\text{d}{v}} = 4\pi   .
\]
Therefore, there exists a constant $C>1$ such that
\begin{equation}\label{energy-bound}
	{{\bar u}_n} \leq {C},\quad \int_{{\mathbb{S}^2}} {{e^{2{u_n}}}} {\text{d}{v}} \leq {C}.
\end{equation}
To obtain the compactness result, a main goal is to prove that the inequality \begin{equation}\label{int_finit}
	\int_{\mathbb{S}^2} e^{2 p u_n} \, {\text{d}{v}} \leq C
\end{equation} holds for  some \( p > 1 \). To this purpose, we employ the method of Green's representation  formula. Firstly, from (\ref{coefficient-bound}), (\ref{energy-bound}) and
$$
{h^2_{1,n}}{e^{2{u_n}}} = {e^{ - 2{u_n}}}{\left( 1 - {h_{2,n}}{e^{{u_n}}}{\left| {{\psi _n}} \right|^2} - \Delta {u_n}\right) ^2},
$$
we have
\begin{equation}\label{weishi}
	\int_{\mathbb{S}^2} e^{-2u_n} \left( 1-{h_{2,n}}{e^{u_n}}{\left| \psi_n  \right|^2} -\Delta u_n\right) ^2  {\text{d}{v}} \leq C.
\end{equation}	
By the small energy condition (\ref{concentrate}) we can choose a sufficiently small \( r \) such that
\begin{equation}\label{local-limit}
	\lim_{n \to \infty} \int_{B_r(x)} |\Delta u_n| \, {\text{d}{v}} < 2\pi.
\end{equation}
Set \( I_n = \int_{B_r(x)} |\Delta u_n| \, {\text{d}{v}} \). Recall the Green's representation formula
\begin{equation*}
	u_n(y) - \bar u_n = \int_{\mathbb{S}^2} -\Delta u_n(z) G(y, z) \, {\text{d}{v}}(z).
\end{equation*}
Hence for some $p>1$ we have
\begin{equation*}
	2p (u_n(y) - \bar u_n) \leq \int_{B_r(x)} 2p |\Delta u_n(z)| |G(y, z)| \, {\text{d}{v}}(z) + C,
\end{equation*}
for any \( y \in B_{r/2}(x) \) and for $C=C(r)$.  By Jensen's inequality and
\begin{equation*}
	\frac{1}{I_n} \int_{B_r(x)} |\Delta u_n| \, {\text{d}{v}} = 1,
\end{equation*}
we obtain
\begin{equation}\label{cedu}			
	\begin{aligned}
		{e^{2p({u_n}(y) - {{\bar u}_n})}} &\leq C{e^{\int_{{B_r}(x)} 2 p{I_n}|G(y,z)| \frac{{|\Delta {u_n}(z)|}}{{{I_n}}}{\text{d}{v}}(z)}}\\
		&\leq  C\int_{{B_r}(x)}\frac{{|\Delta {u_n}(z)|}}{{{I_n}}} e^{2p{I_n}|G(y,z)|} {\text{d}{v}}(z)  \\
		&= \frac{C}{I_n} \int_{B_r(x)} |\Delta u_n(z)| e^{2p I_n |G(y, z)|} \, {\text{d}{v}}(z)
	\end{aligned}
\end{equation}
for any \( y \in B_{r/2}(x) \). According to (\ref{local-limit}) we can find a positive number \( p > 1 \), such that
\begin{equation*}
	\lim_{n \to \infty} p I_n < 2\pi.
\end{equation*}
Recall $ G(y, z) \sim -\frac{1}{2\pi}\log |y - z|$, then
\begin{equation*}
	e^{2\pi |G(y,z)|} \sim \frac{1}{|y - z|}.
\end{equation*}
As a result we have
\begin{equation*}
	\int_{{\mathbb{S}^2}} {{e^{2p{I_n}|G(y,z)|}}}  {\text{d}{v}}(y) \leq C.
\end{equation*}
From (\ref{cedu}) we have
\begin{equation*}
	\begin{aligned}
		\int_{{B_{r/2}}(x)} {{e^{2p({u_n}(y) - {{\bar u}_n})}}}  {\text{d}{v}}(y) &\leq \int_{{B_{r/2}}(x)} {\frac{C}{{{I_n}}}\int_{{B_r}(x)} | \Delta {u_n}(z)|{e^{2p{I_n}|G(y,z)|}} {\text{d}{v}}(z)}  {\text{d}{v}}(y)  \\
		&\leq \frac{C}{{{I_n}}}\int_{{B_r}(x)} {\left| {\Delta {u_n}(z)} \right|\int_{{B_{r/2}}(x)}   {e^{2p{I_n}|G(y,z)|}}{\text{d}{v}}(y){\text{d}{v}}(z) }   \\
		&\leq \frac{C}{{{I_n}}}\int_{{B_r}(x)} {\left| {\Delta {u_n}(z)} \right|{\text{d}{v}}(z) }   \\
		&   \leq C  .
	\end{aligned}
\end{equation*}
By covering \( \mathbb{S}^2 \) with a finite number of balls \( B_{r/2}(x) \), we obtain
\begin{equation}\label{3.1}
	\int_{\mathbb{S}^2} e^{2p (u_n - \bar u_n)} \, {\text{d}{v}} \leq C,
\end{equation}
for some \( p > 1 \). Since $\bar {u}_n\leq C$ by (\ref{energy-bound}), we obtain
\begin{equation*}
	\int_{\mathbb{S}^2} e^{2p u_n} \, {\text{d}{v}} \leq C,
\end{equation*} which implies that (\ref{int_finit}) holds. Specifically,  \[
\int_{{{\mathbb{S}^2}}} {{e^{2\left( {{u_n} - {{\overline u }_n}} \right)}}} {\text{d}}v \leq C,
\]
which implies
\begin{equation}\label{3.2}
	\begin{aligned}
		& \int_{{{\mathbb{S}^2}}} {{h_{1,n}}{e^{2\left( {{u_n} - {{\overline u }_n}} \right)}}} {\text{d}}v + \int_{{{\mathbb{S}^2}}} {h_{2,n}}{{e^{\left( {{u_n} - {{\overline u }_n}} \right)}}{{\left| {{\psi _n}} \right|}^2}} {\text{d}}v \\
		& \leq  M_1\int_{{{\mathbb{S}^2}}} {{e^{2\left( {{u_n} - {{\overline u }_n}} \right)}}} {\text{d}}v +M_2 {\left( {\int_{{{\mathbb{S}^2}}} {{e^{2\left( {{u_n} - {{\overline u }_n}} \right)}}{\text{d}}v} } \right)^{\frac{1}{2}}}{\left( {\int_{{{\mathbb{S}^2}}} {{{\left| {{\psi _n}} \right|}^4}} {\text{d}}v} \right)^{\frac{1}{2}}} \leq C.
	\end{aligned}
\end{equation}
On the other hand, from (\ref{energy-bound}) we have
\[
\begin{aligned}
	\int_{{{\mathbb{S}^2}}} {{h_{1,n}}{e^{2\left( {{u_n} - {{\overline u }_n}} \right)}}} {\text{d}}v + \int_{{{\mathbb{S}^2}}} {h_{2,n}}{{e^{\left( {{u_n} - {{\overline u }_n}} \right)}}{{\left| {{\psi _n}} \right|}^2}} {\text{d}}v =& {e^{ - {{\overline u }_n}}}\left( {{e^{ - {{\overline u }_n}}}\int_{{{\mathbb{S}^2}}} {{h_{1,n}}{e^{2{u_n}}}} {\text{d}}v + \int_{{{\mathbb{S}^2}}} {h_{2,n}}{{e^{{u_n}}}{{\left| {{\psi _n}} \right|}^2}} {\text{d}}v} \right)  \\
	\geq& {e^{ - {{\overline u }_n}}}\left( {\frac{1}{C}\int_{{{\mathbb{S}^2}}} {{h_{1,n}}{e^{2{u_n}}}} {\text{d}}v + \int_{{{\mathbb{S}^2}}} {h_{2,n}}{{e^{{u_n}}}{{\left| {{\psi _n}} \right|}^2}} {\text{d}}v} \right)  \\
	\geq& {e^{ - {{\overline u }_n}}}\frac{1}{C}\left( {\int_{{{\mathbb{S}^2}}} {{h_{1,n}}{e^{2{u_n}}}} {\text{d}}v + \int_{{{\mathbb{S}^2}}} {h_{2,n}}{{e^{{u_n}}}{{\left| {{\psi _n}} \right|}^2}} {\text{d}}v} \right)  \\
	\geq& {e^{ - {{\overline u }_n}}}\frac{{4\pi }}{C} .
\end{aligned}
\]
Hence by (\ref{3.2}) we obtain ${{ - {{\overline u }_n}}} \leq C$. Consequently we have $|\bar{u}_n |\leq C.$

Now we use  H\"older's inequality, (\ref{int_finit}) and (\ref{weishi}) to get
\begin{equation*}
	\begin{aligned}
		& \int_{\mathbb{S}^2}\left| 1 -{h_{2,n}}{e^{u_n}}{\left| \psi_n  \right|^2}-\Delta u_n\right| ^{2p/(p+1)} \, {\text{d}{v}}\\
		\leq &\left( \int_{\mathbb{S}^2} e^{-2 u_n} (1-{h_{2,n}}{e^{u_n}}{\left| \psi_n  \right|^2} -\Delta u_n)^2 \, {\text{d}{v}} \right)^{p/(p+1)} \left( \int_{\mathbb{S}^2} e^{2 p u_n} \, {\text{d}{v}} \right)^{1/(p+1)}\\
		\leq& C,
\end{aligned}			\end{equation*}
and
\begin{equation*}
	\begin{aligned}
		& \int_{\mathbb{S}^2}(e^{u_n}|\psi_n|^2)^{2p/(p+1)} \, {\text{d}{v}}\\
		& \leq \left( \int_{\mathbb{S}^2} e^{2p u_n} \, {\text{d}{v}}\right )^{\frac 1{p+1}} \left( \int_{\mathbb{S}^2}|\psi_n|^4 \, {\text{d}{v}} \right)^{p/(p+1)}\\
		& \leq C.
\end{aligned}			\end{equation*}
Hence by Minkowski's inequality we obtain
\[
\begin{aligned}
	{\left( {\int_{{\mathbb{S}^2}} {{{\left| {\Delta {u_n}} \right|}^{2p/\left( {{\text{p + }}1} \right)}}} {\text{d}}v} \right)^{\frac{{p - 1}}{{2p}}}} \leq  & {\left( {\int_{{\mathbb{S}^2}} {{{\left| {1 - {h_{2,n}}{e^{{u_n}}}{{\left| {{\psi _n}} \right|}^2} - \Delta {u_n} - 1 + {h_{2,n}}{e^{{u_n}}}{{\left| {{\psi _n}} \right|}^2}} \right|}^{2p/(p + 1)}}} {\text{d}}v} \right)^{\frac{{p - 1}}{{2p}}}}  \\
	\leq& {\left( {\int_{{\mathbb{S}^2}} {{{\left| {1 - {h_{2,n}}{e^{{u_n}}}{{\left| {{\psi _n}} \right|}^2} - \Delta {u_n}} \right|}^{2p/(p + 1)}}} {\text{d}}v} \right)^{\frac{{p - 1}}{{2p}}}} \\
	& + {\left( {\int_{{\mathbb{S}^2}} {{{\left| {1 - {h_{2,n}}{e^{{u_n}}}{{\left| {{\psi _n}} \right|}^2}} \right|}^{2p/(p + 1)}}} {\text{d}}v} \right)^{\frac{{p - 1}}{{2p}}}}  \\
	\leq& C,  \\
\end{aligned}
\]
which implies that $u_n-\bar{u}_n$ is in $W^{2,\frac {2p}{p+1}}(\mathbb{S}^2)$ for some $p>1$ and have uniformly bounded norm. By the Sobolev embedding theorem and the boundedness of \(|\bar u_n|\), we have ${\left\| {{u_n}} \right\|_{{C^0}}} \leq {\text{C}}$, which means ${\left\| {{e^{{u_n}}}{\psi _n}} \right\|_{{L^4}}} \leq {\text{C}}.$ From the \(L^p\) estimates \cite[Theorem 3.2.3]{ammann2003variational} for the Dirac equation and the embedding theorem \cite[Theorem 3.3.2, Theorem 3.3.5]{ammann2003variational}, we have ${\left\| {{\psi _n}} \right\|_{{C^{0,1/2}}}} \leq C{\left\| {{\psi _n}} \right\|_{{W^{1,4}}}} \leq {\text{C}}$. Consequently, applying the usual elliptic bootstrap argument yields \[{\left\| {{u_n}} \right\|_{{C^k}}} + {\left\| {{\psi _n}} \right\|_{{C^k}}} \leq C.\]  Thus we complete the proof.\qed

\
\

Next we present the proof of Proposition \ref{prop1.4}, which establishes a novel compactness property on the sphere. To effectively constrain the conformal group of the sphere through the centroid condition (\ref{zhixin}), we shall utilize the celebrated Moser-Trudinger inequality. Specifically, our argument will rely on the sharp version of this inequality for the sphere, originally established by Gui and Moradifam in their seminal work \cite{MR3878729}.
\begin{proposition}[Theorem1.2, \cite{MR3878729}]
	If a function $f \in \mathscr{S}$, then
	\[\fint_{{\mathbb{S}^2}} {{e^{2f}}} {\text{d}}{{\text{v}}_g} \leq {e^{\frac{1}{2}\fint_{{\mathbb{S}^2}} {{{\left| {\nabla f} \right|}^2}} {\text{d}}{{\text{v}}_g} + 2\overline f }}\]
\end{proposition}
\noindent\textbf{Proof of Proposition \ref{prop1.4}. }
\begin{enumerate}[label=\roman*), leftmargin=*]
	\setcounter{enumi}{0}
	\item Multiply both sides of the first equation of (\ref{jinsi}) by $2u_n$, and then integrate them to obtain
	$$
	2\int_{{\mathbb{S}^2}} {{{\left| {\nabla u_n} \right|}^2}{\text{d}{v}}}  = 2\int_{{\mathbb{S}^2}} {\left( {{h_{1,n}}u_n{e^{2u_n}} + {h_{2,n}}u_n{e^{u_n}}{{\left| \psi_n  \right|}^2}} \right){\text{d}{v}}}  - 2\int_{{\mathbb{S}^2}} u_n{\text{d}{v}}.
	$$
	Due to
	\[
	\frac{1}{4\pi}\int_{{\mathbb{S}^2}} {\left( {{h_{1,n}}{e^{2{u_n}}} +{h_{2,n}} {e^{u_n}}{{\left| \psi_n  \right|}^2}} \right){\text{d}{v}} = 1}, \]
	Let us take the unit measure on the sphere
	\[
	d\mu=\frac{1}{4\pi} {\left( {{h_{1,n}}{e^{2{u_n}}} +{h_{2,n}} {e^{u_n}}{{\left| \psi_n  \right|}^2}} \right){\text{d}{v}}}.\]
	Then by Jensen inequality with measure $d\mu$
	$$
	e^{\int_{\mathbb{S}^2}2u_nd\mu}\leq \int_{\mathbb{S}^2}e^{2u_n}d\mu,
	$$
	we have
	$$
	2\int_{{\mathbb{S}^2}} {{{\left| {\nabla u_n} \right|}^2}{\text{d}{v}}}  + 2\int_{{\mathbb{S}^2}} u_n{\text{d}{v}} \leq 4\pi\log \frac{1}{4\pi}\int_{{\mathbb{S}^2}} {\left( {{h_{1,n}}{e^{4u_n}} +{h_{2,n}} {e^{3u_n}}{{\left| \psi_n  \right|}^2}} \right){\text{d}{v}}}
	$$
	By Moser-Trudinger inequality and H\"{o}lder inequality we have
	\begin{equation}\label{S-L-Holder}
		\begin{aligned}
			&2\int_{{\mathbb{S}^2}} {{{\left| {\nabla u_n} \right|}^2}{\text{d}{v}}}  + 2\int_{{\mathbb{S}^2}} u_n{\text{d}{v}} \\
			\leq &4\pi\log \frac{1}{4\pi}\int_{{\mathbb{S}^2}} {\left( {{h_{1,n}}{e^{4u_n}} + {h_{2,n}}{e^{3u_n}}{{\left| \psi_n  \right|}^2}} \right){\text{d}{v}}}  \\
			\leq& 4\pi\log \left(\frac{1}{4\pi} {\int_{{\mathbb{S}^2}} {M_1{e^{4u_n}}{\text{d}{v}}}  +M_2 {{\left( {\frac{1}{4\pi}\int_{{\mathbb{S}^2}} {{e^{3pu_n}}{\text{d}{v}}} } \right)}^{\frac{1}{p}}}{{\left( \frac{1}{4\pi}{\int_{{\mathbb{S}^2}} {{{\left| \psi_n  \right|}^{2q}}} {\text{d}{v}}} \right)}^{\frac{1}{q}}}} \right) \\
			\leq &4\pi\log \left( {M_1{e^{\frac{1}{2\pi}\int_{{\mathbb{S}^2}} {{{\left| {\nabla u_n} \right|}^2}{\text{d}{v}}}  + 4\bar u_n}} +\frac{{{M_2}}}{{{{\left( {4\pi } \right)}^{1/q}}}} {{\left\| \psi_n  \right\|^2_{L^{2q}}}}{{\left( {{e^{\frac{{9p^2}}{8}\frac{1}{4\pi}\int_{{\mathbb{S}^2}} {{{\left| {\nabla u_n} \right|}^2}{\text{d}{v}}}  + 3p\bar u_n}}} \right)}^{\frac{1}{p}}}} \right) ,
		\end{aligned}
	\end{equation}
	where $\frac{1}{p}+\frac{1}{q}=1, p>1$. Take  $p=\frac {16}{9}$, then  $q=\frac {16}{7}$. It follows from (\ref{S-L-Holder}) that
	$$\begin{aligned}
		2\int_{{\mathbb{S}^2}} {{{\left| {\nabla u_n} \right|}^2}{\text{d}{v}}}  +2\int_{{\mathbb{S}^2}} u_n{\text{d}{v}} &\leq 4\pi\log \left( {M_1{e^{\frac{1}{2\pi}\int_{{\mathbb{S}^2}} {{{\left| {\nabla u_n} \right|}^2}{\text{d}{v}}}  + 4\bar u_n}} + \frac{{{M_2}}}{{{{\left( {4\pi } \right)}^{1/q}}}}{{\left\| \psi_n  \right\|}_{L^{\frac{32}{7}}}^2}\left( {{e^{\frac{1}{2\pi}\int_{{\mathbb{S}^2}} {{{\left| {\nabla u_n} \right|}^2}{\text{d}{v}}}  +  {3} \bar u_n}}} \right)} \right)\\
		&\leq 2\int_{{\mathbb{S}^2}} {{{\left| {\nabla u_n} \right|}^2}{\text{d}{v}} + 4\pi \log \left( {M_1{e^{4\bar u_n}} + C{e^{3\bar u_n}}} \right)}.
	\end{aligned}
	$$
	Hence we obtain
	\[
	{e^{2\bar u_n}} \leq M_1{e^{4\bar u_n}} + C{e^{3\bar u_n}},
	\]
	which means $\bar u_n$ has a lower bound, i.e.
	\begin{equation}\label{mean-low}
		\bar u_n \geq  - C.
	\end{equation}			
	Putting (\ref{energy-bound}) and (\ref{mean-low}) together, we obtain
	\begin{equation}\label{mean-bound}
		\left| {{{\bar u}_n}} \right| \leq C.
	\end{equation}
	\item According to Green's representation formula we have \[\begin{aligned}
		- u_n\left( x \right) + \overline u_n & = \int_{{\mathbb{S}^2}} {\Delta u_n(y)G(x,y)dy}   \\
		&= \int_{{\mathbb{S}^2}} {\left( {1 - {h_{1,n}}{e^{2u_n}} - {h_{2,n}}{e^{u_n}}\left| \psi_n  \right|^2} \right)G(x,y)dy}   \\
		&\leq \int_{{\mathbb{S}^2}} {G(x,y)dy}.   \\
	\end{aligned} \]
	Therefore we have \[{u_n} \geq  - C.\]
	\item Since ${u_n} \geq  - C$ by ii), we have
	$$
	{h_{1,n}}{e^{2u_n}} +{h_{2,n}} {e^{u_n}}{{\left| \psi_n  \right|}^2}\geq m_1e^{-2C}>0.
	$$
	We choose $\delta \in (0, \min\{1, m_1e^{-2C}\})$, and take the unit measure on the sphere
	\[
	d\mu=\frac{1}{4\pi} {\frac { {{h_{1,n}}{e^{2{u_n}}} +{h_{2,n}} {e^{u_n}}{{\left| \psi_n  \right|}^2}}-\delta }{1-\delta}{\text{d}{v}}}.\]
	
	As similar arguments in i), we multiply both sides of the first equation by $2u_n$, and then integrate them,  by using Jensen's inequality with  measure $d\mu$ to obtain
	\[\begin{aligned}
		2\int_{{\mathbb{S}^2}} {{{\left| {\nabla u_n} \right|}^2}{\text{d}{v}}} & = 2\int_{{\mathbb{S}^2}} {\left( {{h_{1,n}}u_n{e^{2u_n}} +{h_{2,n}} u_n{e^{u_n}}{{\left| \psi_n  \right|}^2}} \right){\text{d}{v}}}  - 2\int_{{\mathbb{S}^2}} u_n{\text{d}{v}}  \\
		&= 2\int_{{\mathbb{S}^2}} {u_n\left( {{h_{1,n}}{e^{2u_n}} +{h_{2,n}} {e^{u_n}}{{\left| \psi_n  \right|}^2}} \right){\text{d}{v}}}  - 2\bar u_n\int_{{\mathbb{S}^2}} {\left( {{h_{1,n}}{e^{2u_n}} +{h_{2,n}} {e^{u_n}}{{\left| \psi_n  \right|}^2}} \right){\text{d}{v}}}   \\
		&= 2\int_{{\mathbb{S}^2}} {\left[ {\left( {{h_{1,n}}{e^{2u_n}} +{h_{2,n}} {e^{u_n}}{{\left| \psi_n  \right|}^2}} \right) - \delta } \right]\left( {u_n - \bar u_n} \right){\text{d}{v}}}   \\
		&= \left( {1 - \delta } \right)\int_{{\mathbb{S}^2}} {\frac{{\left[ {\left( {{h_{1,n}}{e^{2u_n}} +{h_{2,n}} {e^{u_n}}{{\left| \psi_n  \right|}^2}} \right) - \delta } \right]}}{{1 - \delta }}2\left( {u_n - \bar u_n} \right){\text{d}{v}}}   \\
		&\leq \left( {1 - \delta } \right)4\pi\log\frac{1}{4\pi} \int_{{\mathbb{S}^2}} {\frac{{\left[ {\left( {{h_{1,n}}{e^{2u_n}} +{h_{2,n}} {e^{u_n}}{{\left| \psi_n  \right|}^2}} \right) - \delta } \right]}}{{1 - \delta }}{e^{2\left( {u_n - \bar u_n} \right)}}{\text{d}{v}}}.
	\end{aligned} \]
	As (\ref{S-L-Holder}) we have
	\[\begin{aligned}
		2\int_{{\mathbb{S}^2}} {{{\left| {\nabla u_n} \right|}^2}{\text{d}{v}}}  								 &\leq 4\pi{\left( {1 - \delta } \right)}\log \frac{1}{4\pi}\int_{{\mathbb{S}^2}} {\frac{{\left[ {\left( {{h_{1,n}}{e^{4u_n}} +{h_{2,n}} {e^{3u_n}}{{\left| \psi_n  \right|}^2}} \right){e^{ - 2\bar u_n}} - \delta {e^{2\left( {u_n - \bar u_n} \right)}}} \right]}}{{1 - \delta }}{\text{d}{v}}}   \\
		&\leq4\pi {\left( {1 - \delta } \right)}\log \frac{1}{{4\pi }}\int_{{\mathbb{S}^2}} {\frac{{\left[ {\left( {M_1{e^{\frac{1}{2\pi}\int_{{\mathbb{S}^2}} {{{\left| {\nabla {u_n}} \right|}^2}{\text{d}}v} }}{e^{2{{\bar u}_n}}} + C\left\| {{\psi _n}} \right\|_{{L^{\frac{{32}}{7}}}}^2\left( {{e^{\frac{1}{{2\pi }}\int_{{\mathbb{S}^2}} {{{\left| {\nabla {u_n}} \right|}^2}{\text{d}}v} }}} \right){e^{{{\bar u}_n}}}} \right) - \delta {e^{2\left( {{u_n} - {{\bar u}_n}} \right)}}} \right]}}{{1 - \delta }}{\text{d}}v} \\
		&\leq4\pi {\left( {1 - \delta } \right)}\log \frac{1}{{4\pi }}\int_{{\mathbb{S}^2}} {\frac{{M_1{e^{\frac{1}{{2\pi }}\int_{{\mathbb{S}^2}} {{{\left| {\nabla {u_n}} \right|}^2}{\text{d}}v} }}{e^{2{{\bar u}_n}}} + C\left\| {{\psi _n}} \right\|_{{L^{\frac{{32}}{7}}}}^2\left( {{e^{\frac{1}{{2\pi }}\int_{{\mathbb{S}^2}} {{{\left| {\nabla {u_n}} \right|}^2}{\text{d}}v} }}} \right){e^{{{\bar u}_n}}}}}{{1 - \delta }}{\text{d}}v} \\
		&\leq 4\pi{\left( {1 - \delta } \right)}\log {e^{\frac{1}{2\pi}\int_{{\mathbb{S}^2}} {{{\left| {\nabla u_n} \right|}^2}{\text{d}{v}}} }} + 4\pi(1-\delta)\log \left( {\frac{M_1}{{1 - \delta }}{e^{2\bar u_n}} + C\left\| \psi_n  \right\|_{L^{\frac{{32}}{7}}}^2 {{e^{\bar u_n}}} } \right)  \\
		& \leq 2{\left( {1 - \delta } \right)}\int_{{\mathbb{S}^2}} {{{\left| {\nabla u_n} \right|}^2}{\text{d}{v}}}  + 4\pi{\left( {1 - \delta } \right)}\log \left( {\frac{M_1}{{1 - \delta }}{e^{2\bar u_n}} + C\left\| \psi_n  \right\|_{L^{\frac{{32}}{7}}}^2 {{e^{\bar u_n}}}} \right).
	\end{aligned} \]
	As a result, \[2\int_{{\mathbb{S}^2}} {{{\left| {\nabla u_n} \right|}^2}{\text{d}{v}}}  \leq 2\left( {1 - \delta } \right)\int_{{\mathbb{S}^2}} {{{\left| {\nabla u_n} \right|}^2}{\text{d}{v}}}  + \left( {1 - \delta } \right)4\pi\log \left( {\frac{M_1}{{1 - \delta }}{e^{2\bar u_n}} + C\left\| \psi_n  \right\|_{L^{\frac{{32}}{7}}}^2{e^{\bar u_n}}} \right).\]
	Hence from (\ref{mean-bound}) we get
	\[\int_{{\mathbb{S}^2}} {{{\left| {\nabla u_n} \right|}^2}{\text{d}{v}}}  \leq C.\]
	Thus, by the equivalent norms of $H^1$  we obtain
	\[\left\| {{u_n}} \right\|_{{H^1}}^2 \leq C\left( {\int_{{{\mathbb{S}^2}}} {{{\left| {\nabla {u_n}} \right|}^2}} {\text{d}}v + {{\left| {{{\bar u}_n}} \right|}^2}} \right) \leq C.\]
	Then we have completed the proof.\qed
\end{enumerate}

\section{Non-existence of Solutions and a natural constraint}
We first prove Theorem \ref{trivial-condation}. This proof is related to the proof of the zero mode inequality in the work of Wang and Zhang \cite{wang2025conformal}.

\

\noindent\textbf{Proof of Theorem \ref{trivial-condation}}. From formula (\ref{stokes}) we know that if $(u,\psi)$ is a solution to (\ref{key-Sphere}) with $\psi\not \equiv0$, then
$$
{h_{1,\min }}\int_{{\mathbb{S}^2}} {{e^{2u}}}  {\text{d}{v}}\leq \int_{{\mathbb{S}^2}} {{h_1}{e^{2u}}} {\text{d}{v}} < 4\pi.
$$
From Remark 3.9 and Theorem 1.2 in [66], we know that for the equation
\[\slashiii{D}\psi  = f\psi ,\quad on\;\Sigma ,\quad f \in {L^{4/3}}\left( \Sigma  \right),\]
if \(\psi \not\equiv 0\) is a solution, then
\[
\|f\|_2^2 > 2\pi \chi(\Sigma).
\]
Therefore we have
\[4\pi  < \int_{{\mathbb{S}^2}} {{{\left| { {h_2}{e^u}} \right|}^2}{\text{d}}v}  \leq h_{2,\max }^2\int_{{\mathbb{S}^2}} {{e^{2u}}} {\text{d}{v}}.\]
If the above two formulas hold simultaneously, then it must be the case that \( h_{1,\min} < h_{2,\max}^2 \).\qed

\
\

Next, we use the variational method to find least energy solutions of equation (\ref{key-Sphere}). Recall the functional
\begin{equation*}
	E\left( {u,\psi } \right) = \int_{{\mathbb{S}^2}} {\left[ {{{\left| {\nabla u} \right|}^2} + 2u - h_1{e^{2u}} + 2\left( {\left\langle {\slashiii{D}\psi ,\psi } \right\rangle  - h_2{e^{u}}{{\left| \psi  \right|}^2}} \right)} \right]{\text{d}{v}}}  + 4\pi,
\end{equation*}
where $\left( {u,\psi } \right) \in {H^1}({\mathbb{S}^2}) \times {H^{\frac{1}{2}}}(\Sigma {\mathbb{S}^2})$.


\subsection{A natural constraint}
In this subsection, we introduce a natural constraint for a functional. We use the constraint
\begin{equation}\label{wu-nehari}
	\left\{ {\left. {\left( {u,\psi } \right)} \right|\int_{{\mathbb{S}^2}} {\left\langle {\slashiii{D}\psi  -{h_2} {e^{u}}\psi ,{\varphi _j}} \right\rangle {\text{d}{v}} = 0} ,\;\forall j < 0} \right\}
\end{equation}
in the work of Jevnikar et al. \cite{jevnikar2020existence1} to eliminate the effects of strong indefiniteness of the functional. Here $\varphi_j$ is the eigenspinor corresponding to the eigenvalue $\lambda_j$ of the Dirac operator $\slashiii{D}$.  Therefore combined with (\ref{stokes}) and (\ref{nehari}), we define
\begin{equation*}\begin{aligned}
		\mathcal{A} = \left\{ {\left. {\left( {u,\psi } \right) \in {H^1}(\mathbb{S}^2) \times {H^{1/2}(\mathcal{S}\mathbb{S}^2)}} \right|} \right.&\int_{{\mathbb{S}^2}} {\left\langle {\slashiii{D}\psi  -{h_2} {e^{u}}\psi ,{\varphi _j}} \right\rangle {\text{d}{v}}  = 0} ,\ \forall j < 0, \\
		& \left. {\int_{{\mathbb{S}^2}} {\left\langle {\slashiii{D}\psi ,\psi } \right\rangle }{\text{d}{v}}  = \int_{{\mathbb{S}^2}}{h_2} {{e^{u}}} {{\left| \psi  \right|}^2}}{\text{d}{v}},\ \int_{{\mathbb{S}^2}} {\left( {h_1{e^{2u}} + {h_2}{e^{u}}{{\left| \psi  \right|}^2}} \right)}{\text{d}{v}}  = 4\pi \right\}.
	\end{aligned}
\end{equation*}
Then on this constraint $\mathcal{A}$,
\begin{equation*}\begin{aligned}
		E\left( {u,\psi } \right) &= \int_{{\mathbb{S}^2}} {\left( {{{\left| {\nabla u} \right|}^2} + 2u - {h_1}{e^{2u}}} \right){\text{d}{v}} + 4\pi }   \\
		&= \int_{{\mathbb{S}^2}} {\left( {{{\left| {\nabla u} \right|}^2} + 2u +{h_2} {e^{u}}{{\left| \psi  \right|}^2}} \right)}{\text{d}{v}} .
	\end{aligned}
\end{equation*}
If a constraint critical point $(u,\psi)$ of ${\left. {E(u,\psi )} \right|_\mathcal{A}}$ satisfies $\psi\not\equiv 0$, then for any $(v,\phi)\in{H^1}({\mathbb{S}^2}) \times {H^{1/2}}\left( {\mathcal{S} {\mathbb{S}^2}} \right)$, there exist Lagrange multipliers $\lambda$, $\tau$ and $\mu_j$ with $j<0$ such that
\begin{equation}\label{grad_u}
	\begin{aligned}
		&2\int_{{\mathbb{S}^2}} {\left( {\nabla u\nabla v + v - h_1{e^{2u}}v - {h_2}{e^{u}}|\psi {|^2}v} \right)}{\text{d}{v}}   \\
		=  &- \sum_{j<0}{\mu _j}\int_{{\mathbb{S}^2}} {h_2}{{e^{u}}\left\langle {\psi ,{\varphi _j}} \right\rangle v}{\text{d}{v}}  - \lambda \int_{{\mathbb{S}^2}} {h_2}{{e^{u}}} |\psi {|^2}v{\text{d}{v}} + \tau \int_{{\mathbb{S}^2}} {\left( {2h_1{e^{2u}} + {h_2}{ e^{u}}|\psi {|^2}} \right)} v {\text{d}{v}}.
	\end{aligned}
\end{equation}
and
\begin{equation}\label{grad_psi}
	\begin{aligned}
		&4\int_{{\mathbb{S}^2}} {\left( {\langle \slashiii{D}\psi ,\phi \rangle  - {h_2}{ e^{u}}\langle \psi ,\phi \rangle } \right)}{\text{d}{v}}  \\
		=	& \sum_{j<0}{\mu _j}\int_{{\mathbb{S}^2}} {\left\langle {\slashiii{D}\phi  - {h_2}{ e^{u}}\phi ,{\varphi _j}} \right\rangle }{\text{d}{v}}  + 2\lambda \int_{{\mathbb{S}^2}} {\left( {\langle \slashiii{D}\psi ,\phi \rangle  - {h_2}{ e^{u}}\langle \psi ,\phi \rangle } \right)}{\text{d}{v}}  + 2\tau \int_{{\mathbb{S}^2}} {\left\langle {{ {h_2}e^{u}}\psi ,\phi } \right\rangle }{\text{d}{v}} .
	\end{aligned}
\end{equation}
Take $\phi=\psi$ in (\ref{grad_psi}) and use (\ref{wu-nehari}) as well as (\ref{nehari}), we have
\begin{equation*}
	0 = 2\tau \int_{{\mathbb{S}^2}} {h_2}{{ e^{u}}{{\left| \psi  \right|}^2}}{\text{d}{v}}. \end{equation*}
Hence $\tau  = 0$. Letting $\phi=\varphi_j ,\ j<0$ in (\ref{grad_psi}) and using (\ref{wu-nehari}), we obtain \begin{equation*}
	\begin{aligned}
		0& = {\mu _j}\int_{{\mathbb{S}^2}} {\left\langle {\slashiii{D}{\varphi _j} - {h_2}{ e^{u}}{\varphi _j},{\varphi _j}} \right\rangle }{\text{d}{v}}   \\
		&= {\mu _j}\left( {\int_{{\mathbb{S}^2}} {\left\langle {\slashiii{D}{\varphi _j},{\varphi _j}} \right\rangle }  - {{\int_{{\mathbb{S}^2}} {h_2}{{ e^{u}}\left| {{\varphi _j}} \right|} }^2}} \right){\text{d}{v}},  \\
		&=  - {\mu _j}|\lambda_j| - {\mu _j}{\int_{{\mathbb{S}^2}} {h_2}{{ e^{u}}\left| {{\varphi _j}} \right|} ^2}{\text{d}{v}}.
	\end{aligned}
\end{equation*}
From this, it can be concluded that $\mu_j=0$. Taking $v=1$ in (\ref{grad_u}) and using (\ref{stokes}), we have
\begin{equation*}
	0 =  - \lambda \int_{{\mathbb{S}^2}} {h_2}{{ e^{u}}} |\psi {|^2}{\text{d}{v}}
\end{equation*}
Therefore $\lambda=0$.	If a constrained critical point $(u,\psi)$ of ${\left. {E(u,\psi )} \right|_\mathcal{A}}$ satisfies $\psi\equiv 0$, it can be concluded that all Lagrange multipliers are zero as above. Therefore the constraint $\mathcal{A} $ is a natural constraint.

We now show that \(\mathcal{A}\) is nontrivial with respect to $\psi$. Take \((u_0, \psi_0) \in H^1 \times H^{1/2}\) and define \((u', \psi') = (\log t + u_0,\, s\psi_0)\). If \((u', \psi') \in \mathcal{A}\), then the following holds:
\begin{equation}\label{ray-1}
	{t^2}\int_{{\mathbb{S}^2}} {{h_1}(x){e^{2{u_0}}}} {\text{d}}v + t{s^2}\int_{{\mathbb{S}^2}} {{h_2}(x){e^{{u_0}}}{{\left| {{\psi _0}} \right|}^2}} {\text{d}}v = 4\pi,
\end{equation}
and
\begin{equation}\label{ray-2}
	{s^2}\int_{{\mathbb{S}^2}} {\left\langle {\slashiii{D}{\psi _0},{\psi _0}} \right\rangle {\text{d}}v}  = {s^2}t\int_{{\mathbb{S}^2}} {{h_2}(x){e^{{u_0}}}{{\left| {{\psi _0}} \right|}^2}{\text{d}}v}.
\end{equation}
By (\ref{ray-2}), it is first necessary to have
\begin{equation}\label{nes_1}
	t = \frac{{\int_{{\mathbb{S}^2}} {\left\langle {\slashiii{D}{\psi _0},{\psi _0}} \right\rangle dV} }}{{\int_{{\mathbb{S}^2}} {{h_2}(x){e^{{u_0}}}{{\left| {{\psi _0}} \right|}^2}dV} }} > 0.
\end{equation}
and
\begin{equation}\label{nes_2}
	{s^2} = \frac{{4\pi }}{{\int_{{\mathbb{S}^2}} {\left\langle {\slashiii{D}\psi_0 ,\psi_0 } \right\rangle dV} }} - \frac{{\int_{{\mathbb{S}^2}} {{h_1}(x){e^{2u_0}}} dV\int_{{\mathbb{S}^2}} {\left\langle {\slashiii{D}\psi_0 ,\psi_0 } \right\rangle dV} }}{{{{\left( {\int_{{\mathbb{S}^2}} {{h_2}(x){e^{u_0}}{{\left| \psi_0  \right|}^2}} dV} \right)}^2}}} \geq 0.
\end{equation}
To satisfy Condition (\ref{nes_1}), we may select \( \psi_0 \in H^{1/2,+} \). For Condition (\ref{nes_2}), we set \( e^{u_0} = \frac{c}{h_2} \), at which point the first constraint on $\mathcal{A}$ is also satisfied. It then suffices to show the existence of a pair \( \psi_0 \) such that	
\[
{\left( {\frac{{4\pi }}{{\int_{{\mathbb{S}^2}} {\frac{{{h_1}(x)}}{{h_2^2(x)}}} {\text{d}}v}}} \right)^{1/2}} \geq \frac{{\int_{{\mathbb{S}^2}} {\left\langle {\slashiii{D}{\psi _0},{\psi _0}} \right\rangle {\text{d}}v} }}{{\int_{{\mathbb{S}^2}} {{{\left| {{\psi _0}} \right|}^2}} {\text{d}}v}}.
\]
This demonstrates that if ${\int_{{\mathbb{S}^2}} {\frac{{{h_1}(x)}}{{h_2^2(x)}}} {\text{d}}v}$ is sufficiently small, less than $4\pi$ for instance, then $\mathcal{A}$ must be nontrivial. As will be shown in the proof of the theorem \ref{main-Thm}, the condition \(\lambda_1(h_2, h_1) < 1\) is sufficient to guarantee the nontriviality of the set \(\mathcal{A}\).

Next, we show that $\mathcal{A} \setminus \left( H^1(\mathbb{S}^2) \times \{0\} \right)$ is a submanifold. To this end, we define the following functionals:$$\begin{aligned}
	G_{1,j}(u,\psi) &= \int_{\mathbb{S}^2} \langle \slashiii{D}\psi - h_2 e^u \psi, \varphi_j \rangle \, \text{d}v, \quad \forall j < 0, \\
	G_2(u,\psi) &= \int_{\mathbb{S}^2} \left( \langle \slashiii{D}\psi, \psi \rangle - h_2 e^u |\psi|^2 \right) \text{d}v, \\
	G_3(u,\psi) &= \int_{\mathbb{S}^2} \left( h_1 e^{2u} + h_2 e^u |\psi|^2 \right) \text{d}v - 4\pi.
\end{aligned}$$The fact that the constraints $G_{1,j}(u,\psi) = 0$ and $G_3(u,\psi) = 0$ define a submanifold has been established by Jevnikar et al. \cite{MR4206467}. Thus, it remains to examine the functional $G_2(u,\psi)$. Given that $\psi \not\equiv 0$, we consider the variation $G_2(-s \log h_2, t\varphi_1) = t^2(\lambda_1 - s)$. This expression reveals that the partial derivative of $G_2$ with respect to the parameters is surjective (non-vanishing). Consequently, zero is a regular value of $G_2$, implying that the level set $\{G_2 = 0\}$ is indeed a submanifold. Combining these observations, we conclude that $\mathcal{A} \setminus \left( H^1(\mathbb{S}^2) \times \{0\} \right)$ inherits the submanifold structure.



\subsection{The lower bound of the functional and a supersymmetric inequality}	
\begin{lemma}\label{lower-bound}
	Assume that \((u, \psi) \in \mathcal{A}\). Then there is a $C\in \mathbb{R}$ such that the following holds:
	\begin{enumerate}[label=\roman*), leftmargin=*]
		\setcounter{enumi}{0}
		\item $\ S\left[ u \right] \geq C$;
		\item $E(u,\psi)$ has a lower bound on $\mathcal{A}$.
	\end{enumerate}
\end{lemma}
\begin{proof}
	Let \((u, \psi) \in \mathcal{A}\). From the condition (\ref{stokes}) in $\mathcal{A}$ we show that there is $\varepsilon>0$ to be fixed later such that \begin{equation}\label{epsilon}
		\begin{aligned}
			4\pi  &= \int_{{\mathbb{S}^2}} {h_1{e^{2u}}}{\text{d}{v}}  + \int_{{\mathbb{S}^2}} {h_2}{{ e^{u}}{{\left| \psi  \right|}^2}}{\text{d}{v}}  \\
			&\leq {h_{1,\max}}\int_{{\mathbb{S}^2}} {{e^{2u}}} {\text{d}{v}} + \int_{{\mathbb{S}^2}} {\left( {{C_\varepsilon }{e^{2u}} + \varepsilon {{\left| \psi  \right|}^4}} \right)}{\text{d}{v}}  \\
			&\leq \left( {{h_{1,\max}} + {C_\varepsilon }} \right)\int_{{\mathbb{S}^2}} {{e^{2u}}} {\text{d}{v}}} + \varepsilon {\left\| \psi  \right\|_{H^{1/2}}^4  .
		\end{aligned}
	\end{equation}
	Then by Moser-Trudinger inequality we deduce \begin{equation*}
		\frac{{4\pi  - \varepsilon {{\left\| \psi  \right\|}_{H^{1/2}}^4}}}{{{h_{1,\max}} + {C_\varepsilon }}} \leq \int_{{\mathbb{S}^2}} {{e^{2u}}} {\text{d}{v}} \leq 4\pi {e^{\frac{1}{{4\pi }}\int_{{\mathbb{S}^2}} {\left( {{{\left| {\nabla u} \right|}^2} + 2 u} \right)} {\text{d}{v}}}}.
	\end{equation*}
	Proceeding analogously to the proof of Theorem \ref{thm1} but with $\psi^+$ replaced by $\psi^-$ in (\ref{nonlinear}) and appropriate adjustments to the sign convention,  the constraint \( \mathcal{A} \) implies the uniform boundedness of the Sobolev norm of the spinor \( \psi \) i.e. ${\left\| \psi  \right\|_{{H^{\frac{1}{2}}}}}\leq C$. According to the boundedness of $\psi$, there are definite constants $\varepsilon $ and $C'$ such that
	\begin{equation*}
		C'\leq 4\pi \log \frac{{1 - \varepsilon {{\left\| \psi  \right\|}_{H^{1/2}}^4}}}{{{h_{1,\max}} + {C_\varepsilon }}} \leq \int_{{\mathbb{S}^2}} \left( {{{\left| {\nabla u} \right|}^2}} + 2 u\right){\text{d}{v}} =4\pi  S\left[ u \right]\leq E(u,\psi) .
	\end{equation*}
	Thus, the functional $E(u,\psi)$ has a lower bound on $\mathcal{A}$.
\end{proof}

\begin{theorem}\label{SMT}
	Let $(\mathbb{S}^2,g_0)$ be the standard round sphere and let $\slashiii{D}$ be the Dirac operator associated with its spin structure. Let
	\[
	u\in H^1(\mathbb{S}^2),\qquad \bar u=0,\qquad
	\psi\in H^{1/2}(\Sigma \mathbb{S}^2)\setminus\{0\}.
	\]
	If
	\[
	\rho>0,\qquad \mu>0,\qquad \rho+\mu\le 4\pi,
	\]
	then
	\[
	\frac{\rho}{2}\log\int_{\mathbb{S}^2}e^{2u}\,\text{d}v
	+
	\mu\log\int_{\mathbb{S}^2}e^u|\psi|^2\,\text{d}v
	\le
	\frac12\int_{\mathbb{S}^2}|\nabla u|^2\,\text{d}v
	+
	\int_{\mathbb{S}^2}\langle |\slashiii{D}|\psi,\psi\rangle\,\text{d}v
	+
	C_{\rho,\mu},
	\]
	where
	\[{C_{\rho ,\mu }} = \frac{\rho }{2}\log 4\pi  + \mu \log 2\mu  - \mu .\]
\end{theorem}

\begin{proof}
	Set
	\[
	A(u):=\int_{\mathbb{S}^2}e^{2u}\,\text{d}v,
	\qquad
	Q(\psi):=\int_{\mathbb{S}^2}\langle |\slashiii{D}|\psi,\psi\rangle\,\text{d}v ,
	\]
	and
	\[
	C_{|D|}
	:=
	\inf_{\psi\in H^{1/2}(\Sigma \mathbb{S}^2)\setminus\{0\}}
	\frac{\displaystyle\int_{\mathbb{S}^2}\langle |\slashiii{D}|\psi,\psi\rangle\,\text{d}v}
	{\displaystyle\left(\int_{\mathbb{S}^2}|\psi|^4\,\text{d}v\right)^{1/2}} .
	\]
	By H\"older's inequality and the sharp Sobolev inequality for spinors,
	\[
	\int_{\mathbb{S}^2}e^u|\psi|^2\,\text{d}v
	\le
	A(u)^{1/2}
	\left(\int_{\mathbb{S}^2}|\psi|^4\,\text{d}v\right)^{1/2}
	\le
	\frac{A(u)^{1/2}Q(\psi)}{C_{|D|}} .
	\]
	Hence
	\[
	\mu\log\int_{\mathbb{S}^2}e^u|\psi|^2\,\text{d}v - Q(\psi)
	\le
	\frac{\mu}{2}\log A(u)
	-\mu\log C_{|D|}
	+\mu\log Q(\psi)-Q(\psi).
	\]
	Since
	\[
	\sup_{t>0}\{\mu\log t-t\}
	=
	\mu\log\mu-\mu,
	\]
	we obtain
	\[
	\mu\log\int_{\mathbb{S}^2}e^u|\psi|^2\,\text{d}v - Q(\psi)
	\le
	\frac{\mu}{2}\log A(u)
	-\mu\log C_{|D|}
	+\mu\log\mu-\mu .
	\]
	Therefore
	\[
	\begin{aligned}
		&\frac{\rho}{2}\log A(u)
		+\mu\log\int_{\mathbb{S}^2}e^u|\psi|^2\,\text{d}v
		-\frac12\int_{\mathbb{S}^2}|\nabla u|^2\,\text{d}v
		-Q(\psi)
		\\
		\le&
		\frac{\rho+\mu}{2}\log A(u)
		-\frac12\int_{\mathbb{S}^2}|\nabla u|^2\,\text{d}v
		-\mu\log C_{|D|}
		+\mu\log\mu-\mu .
	\end{aligned}
	\]
	By the sharp Onofri inequality on the round sphere,
	\[
	\log A(u)	\le	\log4\pi
	+\frac{1}{4\pi}\int_{\mathbb{S}^2}|\nabla u|^2\,\text{d}v,
	\qquad \bar u=0.
	\]
	It follows that
	\[
	\begin{aligned}
		&\frac{\rho}{2}\log A(u)
		+\mu\log\int_{\mathbb{S}^2}e^u|\psi|^2\,\text{d}v-\frac12\int_{\mathbb{S}^2}|\nabla u|^2\,\text{d}v	-Q(\psi)
		\\
		\le&\frac{\rho+\mu}{2}\log4\pi
		+\left(\frac{\rho+\mu}{8\pi}-\frac12\right)\int_{\mathbb{S}^2}|\nabla u|^2\,\text{d}v
		-\mu\log C_{|D|}
		+\mu\log\mu-\mu .
	\end{aligned}
	\]
	Since $\rho+\mu\le 4\pi$, the coefficient of the Dirichlet energy is non-positive. From Theorem \ref{Sharp-const} in the appendix we know that \(C_{|D|} = \sqrt \pi  \), where there we provide the optimal constant in general dimensions (\(n \geq 2\)), and this constant is not attained by any non-zero spinor. Thus
	\[
	\frac{\rho}{2}\log A(u)
	+\mu\log\int_{\mathbb{S}^2}e^u|\psi|^2\,\text{d}v
	-\frac12\int_{\mathbb{S}^2}|\nabla u|^2\,\text{d}v
	-Q(\psi)
	\le
	C_{\rho,\mu}.
	\]
	Rearranging gives the desired inequality.
\end{proof}
\section{Existence of least-energy solutions}
Consequently, we conclude that the functional $E(u,\psi)$ is bounded from below on $\mathcal{A}$. In general, however, this lower bound may not be attained. For the specific functions $h_i(x),\ i=1,2$ considered in Theorem \ref{main-Thm}, the infimum is indeed achieved, which guarantees the existence of minimizers with minimal energy. In particular, if ${\lambda _1}\left( {h_2},{h_1}\right) < 1$, these minimizers are nontrivial. Before proceeding to the proof Theorem \ref{main-Thm}, we examine the parameter ${\lambda _1}\left( {h_2},{h_1}\right)$. The following lemma ensures that ${\lambda _1}\left( {h_2},{h_1}\right)> 0$.
\begin{lemma}\label{Eigen-positve}
	Let $(\mathbb{S}^2,g_0)$ be the standard sphere and let $\slashiii{D}$ be the Dirac operator.
	Assume that $w\in L^\infty(\mathbb{S}^2)$ satisfies
	\[
	0<m\le w(x)\le M<\infty
	\]
	for a.e.\ $x\in \mathbb{S}^2$. Consider the weighted eigenvalue problem
	\[
	\slashiii{D}\psi=\lambda w\psi\quad \text{on } \,\mathbb{S}^2.
	\]
	Define
	\[
	E_w^-:=\bigoplus_{\lambda<0}\ker(\slashiii{D}-\lambda w)
	\]
	and
	\[
	\lambda_1(w)
	:=
	\inf_{\substack{\psi\neq0\\ \psi\perp_w E_w^-}}
	\frac{\displaystyle\int_{\mathbb{S}^2}\langle \slashiii{D}\psi,\psi\rangle\,\text{d}v}
	{\displaystyle\int_{\mathbb{S}^2}w|\psi|^2\,\text{d}v},
	\]
	where $\perp_w$ denotes orthogonality with respect to
	\[
	(\psi,\varphi)_w:=		\int_{\mathbb{S}^2}w\langle\psi,\varphi\rangle\,\text{d}v.
	\]
	Then
	\[
	\lambda_1(w)>0.
	\]
\end{lemma}
\begin{proof}
	Since $w$ is uniformly positive and bounded, the weighted norm
	\[
	\|\psi\|_{L^2_w}^2:=\int_{\mathbb{S}^2}w|\psi|^2\,\text{d}v
	\]
	is equivalent to the standard $L^2$ norm of $\psi$. Define
	\[
	T:=w^{-1}\slashiii{D}.
	\]
	Then for smooth spinors $\psi,\varphi$,
	\[
	(T\psi,\varphi)_w
	=
	\int_{\mathbb{S}^2}\langle \slashiii{D}\psi,\varphi\rangle\,\text{d}v
	=
	\int_{\mathbb{S}^2}\langle\psi,\slashiii{D}\varphi\rangle\,\text{d}v
	=
	(\psi,T\varphi)_w.
	\]
	Hence $T$ is self-adjoint on $L^2_w$. Since $\slashiii{D}$ possesses a compact resolvent on
	$\mathbb{S}^2$ and the multiplication operator $w^{-1}$ is bounded,the operator $T$ consequently inherits compact resolvent. As a result, its spectrum is purely discrete and comprises only real
	eigenvalues.  Moreover,
	\[
	T\psi=0
	\quad\Longleftrightarrow\quad
	\slashiii{D}\psi=0.
	\]
	Since $\mathbb{S}^2$ admits no harmonic spinors, $\ker \slashiii{D}=\{0\}$, hence
	$\ker T=\{0\}$. Consequently,
	\[
	0\notin \operatorname{Spec}(T),
	\]
	and thus the first positive eigenvalue of $T$ is strictly positive.
	By the spectral characterization of self-adjoint operators, this
	eigenvalue is exactly $\lambda_1(w)$. Therefore $	\lambda_1(w)>0.$
\end{proof}

\
\

\noindent\textbf{Proof of Theorem \ref{main-Thm}. }Firstly, since $h_1(-x)=h_1(x)$ by assumptions in Theorem \ref{main-Thm}, by Moser's seminal result \cite{MR339258}, we know that there exists a even function $u(x)$ satisfying
\[ - \Delta u = h_1(x){e^{2u}} - 1,\quad on\;{\mathbb{S}^2}.\]
This implies that the set \({\mathcal{A}} \) is nonempty. From Lemma \ref{lower-bound}, we can define
\begin{equation*}
	c_0 = \mathop {\inf }\limits_{\mathcal{A}} E\left( {u,\psi } \right).
\end{equation*}Then  there is a minimizing sequence $\left\{ {\left( {{u_n},{\psi _n}} \right)} \right\} \subset \mathcal{A} $ such that
\begin{equation}\label{minimizer}
	E\left( {{u_n},{\psi _n}} \right) \to c_0, ~~~ \text{ as } n \to  \infty.
\end{equation}
Notice that  $\left\{ {\left( {{u_n},{\psi _n}} \right)} \right\} $ satisfies
\begin{equation*}
	\int_{{\mathbb{S}^2}} {\left\langle {\slashiii{D}{\psi _n} -{h_2} { e^{{u_n}}}{\psi _n},{\varphi _j}} \right\rangle {\text{d}{v}} = 0} ,\;\forall j < 0.
\end{equation*}
\begin{equation*}
	\int_{{\mathbb{S}^2}} {\left\langle {\slashiii{D}{\psi _n},{\psi _n}} \right\rangle } {\text{d}{v}} = \int_{{\mathbb{S}^2}}{h_2} {{e^{{u_n}}}} {\left| {{\psi _n}} \right|^2}{\text{d}{v}}.
\end{equation*}
\begin{equation}\label{stokes-un}
	\int_{{\mathbb{S}^2}} {\left( {h_1{e^{2{u_n}}} + {h_2}{e^{{u_n}}}{{\left| {{\psi _n}} \right|}^2}} \right)} {\text{d}{v}} = 4\pi .
\end{equation}	
As the similar arguments in the proof of Theorem \ref{thm1}, we deduce that $\left\| \psi^+_n  \right\|_{H^{1/2}},\left\| \psi^-_n  \right\|_{H^{1/2}}$ and $\left\| \psi_n  \right\|_{H^{1/2}}$ are bounded. By (\ref{stokes-un}), we have \begin{equation*}4\pi {h_{1,\min}}{e^{2{{\overline u }_n}}} \leq {h_{1,\min}}\int_{{\mathbb{S}^2}} {{e^{2{u_n}}}{\text{d}{v}}}  \leq \int_{{\mathbb{S}^2}} {{h_1}{e^{2{u_n}}}{\text{d}{v}}}  \leq 4\pi. \end{equation*}
Hence
\begin{equation}\label{upbound}
	2{\overline u _n} \leq \log \frac{1}{{{h_{1,\min}}}}.
\end{equation} 
According to (\ref{minimizer}), without loss of generality, we assume
\begin{equation}\label{lowbound}
	c_0 - 1 \leq \int_{{\mathbb{S}^2}} {\left( {{{\left| {\nabla u_n} \right|}^2} + 2u_n + {h_2}{e^{u_n}}{{\left| \psi_n  \right|}^2}} \right)}{\text{d}{v}}  \leq c_0 + 1.
\end{equation}
Then from (\ref{upbound}) and (\ref{lowbound}) we have \begin{equation*}\int_{{\mathbb{S}^2}} {{{\left| {\nabla u_n} \right|}^2}{\text{d}{v}}}  \leq c_0 + 1 - 4\pi \cdot 2 {\overline u _n}
\end{equation*} and
\begin{equation}\label{4.1}
	\begin{aligned}
		4\pi  \cdot 2{\overline u _n} &\geq c_0 - 1 - \int_{{\mathbb{S}^2}} {{{\left| {\nabla {u_n}} \right|}^2}}{\text{d}{v}}  - \int_{{\mathbb{S}^2}}{h_2} {{e^{{u_n}}}{{\left| {{\psi _n}} \right|}^2}} {\text{d}{v}}  \\
		&\geq c_0 - 1 - \int_{{\mathbb{S}^2}} {{{\left| {\nabla {u_n}} \right|}^2}}{\text{d}{v}}  - C{\left( {\int_{{\mathbb{S}^2}} {{e^{2{u_n}}}} }{\text{d}{v}} \right)^{1/2}}{\left( {\int_{{\mathbb{S}^2}} {{{\left| {{\psi _n}} \right|}^4}} {\text{d}{v}}} \right)^{1/2}}  \\
		&\geq c_0 - 1 - \int_{{\mathbb{S}^2}} {{{\left| {\nabla {u_n}} \right|}^2}} {\text{d}{v}} - C\left\| {{\psi _n}} \right\|_{H^{1/2}}^2\sqrt {4\pi } {e^{\frac{1}{2}\left( {\frac{1}{{4\pi }}\int_{{\mathbb{S}^2}} {\left( {{{\left| {\nabla {u_n}} \right|}^2} + 2u} \right){\text{d}{v}}} } \right)}}  \\
		&\geq c_0 - 1 - \int_{{\mathbb{S}^2}} {{{\left| {\nabla {u_n}} \right|}^2}{\text{d}{v}}}  - C{e^{\frac{1}{{8\pi }}\int_{{\mathbb{S}^2}} {{{\left| {\nabla {u_n}} \right|}^2}} {\text{d}{v}}}} .
	\end{aligned}
\end{equation}

If we take the basic Hilbert space \begin{equation*}
	H_{even}^1\left( {{\mathbb{S}^2}} \right) \times {H^{1/2}}\left( {\Sigma {\mathbb{S}^2}} \right),
\end{equation*}
where $H_{even}^1\left( {{\mathbb{S}^2}} \right)$ is the set of all even symmetric functions in $H^{1}(\mathbb{S}^2) $. Then the classical Moser-Trudinger inequality can be considerably improved (see \cite{MR339258}), taking the form
\begin{equation*}
	\int_{{\mathbb{S}^2}} {{e^{2f}}{\text{d}{v}} \leq } 4\pi {e^{\frac{1}{{8\pi }}\int_{{\mathbb{S}^2}} {{{\left| {\nabla f} \right|}^2}{\text{d}{v}}}  + 2\overline f }},\quad \forall f \in H_{even}^1\left( {{\mathbb{S}^2}} \right).
\end{equation*}
As in (\ref{epsilon}), there exists a constant $C$ such that
\begin{equation*}
	0<C<\frac{{4\pi  - \varepsilon {{\left\| \psi_n  \right\|}_{H^{1/2}}^4}}}{{{h_{1,\max}} + {C_\varepsilon }}} \leq \int_{{\mathbb{S}^2}} {{e^{2u_n}}} {\text{d}{v}} \leq 4\pi {e^{\frac{1}{{8\pi }}\int_{{\mathbb{S}^2}} {\left( {{{\left| {\nabla u_n} \right|}^2} } \right){\text{d}{v}+ 2 \bar u_n} }}}.
\end{equation*}
It follows that
\begin{equation*}
	\int_{{\mathbb{S}^2}} {2{u_n}} {\text{d}{v}} \geq  - \frac{1}{2}\int_{{\mathbb{S}^2}} {{{\left| {\nabla {u_n}} \right|}^2}{\text{d}{v}}}  - C.\end{equation*}
In view of the right-hand side of (\ref{lowbound}) we have
\begin{equation}\label{4.2}
	{c_0} + 1 \geq E\left( {{u_n},{\psi _n}} \right) \geq \frac{1}{2}\int_{{\mathbb{S}^2}} {{{\left| {\nabla {u_n}} \right|}^2}{\text{d}{v}}}-C .\end{equation}
Combining (\ref{4.1}) and (\ref{4.2}), we find that $\bar {u}_n\geq -C$. Furthermore we know that $|\bar{u}_n|\leq C$ by (\ref{upbound}). Consequently we have that $\left\| {{u_n}} \right\|_{H^1}$ is bounded. By the Moser-Trudinger and Sobolev inequality, through standard arguments (or see \cite[Lemma3.3]{jevnikar2020existence1}), one can show that \( \mathcal{A} \) is weakly closed. Then from Ekeland's variational principle, we know there is a
\begin{equation*}
	\left( {{u_0},{\psi _0}} \right) \in H_{even}^1\left( {{\mathbb{S}^2}} \right) \times {H^{1/2}}\left( {\Sigma {\mathbb{S}^2}} \right)
\end{equation*}
such that
\begin{equation*}
	{c_0} = E\left( {{u_0},{\psi _0}} \right).
\end{equation*}
Moreover, $\left( {{u_0},{\psi _0}} \right)$ is a critical point to the functional $E\left( {u,\psi } \right)$.

By assumption, \(u^*_{h_1}\) is a minimizer of the functional \(4\pi S[u]\) on the constraint \(\mathcal{A}_0\) and satisfies
\[
-\Delta u = h_1 e^{2u} - 1, \quad \text{on } \mathbb{S}^2.
\]
Then we know
\[E\left( {u_{{h_1}}^*,0} \right) = 4\pi S\left[ {u_{{h_1}}^*} \right] .\]
By assumption, if the first eigenvalue \(\lambda_1(h_2, h_1)\) of the weighted Dirac operator \((h_2 e^{u^*_{h_1}})^{-1} \slashiii{D}\) satisfies \(\lambda_1(h_2, h_1) < 1\), then there exists a spinor \(\phi_1\) such that
\[
\slashiii{D}\phi_1 = \lambda_1(h_2, h_1) \, h_2 e^{u^*_{h_1}} \phi_1, \quad \text{on } \mathbb{S}^2.
\]
Consequently, there exists \(t > 0\) such that
\[
e^{-t} = \lambda_1(h_2, h_1).
\]
For this \(t\), we consider the family of test pairs \((u_t, \psi_t)\) defined by
\[
u_t = u^*_{h_1} - t, \qquad \psi_t = c_t \phi_1.
\]
Clearly, \((u_t, \psi_t)\) satisfies the first two equations of \(\mathcal{A}\). The third volume constraint requires
\[
4\pi(1 - e^{-2t}) = 4\pi - \int_{\mathbb{S}^2} h_1 e^{2(u^*_{h_1}-t)} \, \text{d}v = \int_{\mathbb{S}^2} h_2 e^{u_t} |\psi_t|^2 \, \text{d}v = e^{-t} c^2_t \int_{\mathbb{S}^2} h_2 e^{u^*_{h_1}} |\phi_1|^2 \, \text{d}v.
\]
Hence, it suffices to choose
\[
c^2_t = \frac{4\pi(1 - e^{-2t}) e^{t}}{\int_{\mathbb{S}^2} h_2 e^{u^*_{h_1}} |\phi_1|^2 \, \text{d}v} = \frac{4\pi(1 - \lambda_1^2(h_2, h_1))}{\lambda_1(h_2, h_1) \int_{\mathbb{S}^2} h_2 e^{u^*_{h_1}} |\phi_1|^2 \, \text{d}v},
\]
so that \((u_t, \psi_t) \in \mathcal{A}\). For any \((u, \psi) \in \mathcal{A}\), the energy functional simplifies to
\[
E(u, \psi) = \int_{\mathbb{S}^2} \left( |\nabla u|^2 + 2u - h_1 e^{2u} \right) + 4\pi.
\]
Evaluating this for the test pair yields
\[
E(u_t, \psi_t) = \int_{\mathbb{S}^2} \left( |\nabla u^*_{h_1}|^2 + 2(u^*_{h_1} - t) - h_1 e^{2u^*_{h_1}} e^{-2t} \right) \text{d}v + 4\pi.
\]
Using the identity ${\int_{{\mathbb{S}^2}} {{h_1}{e^{2u^*_{h_1}}}} {\text{d}}v = 4\pi }$, we obtain
\[
E(u_t, \psi_t) = E(u^*_{h_1}, 0) - 8\pi t - 4\pi e^{-2t} + 4\pi.
\]
Define \(f(t) = 4\pi(1 - 2t - e^{-2t})\). As noted, \(f(0) = 0\) and \(f'(t) < 0\) for \(t > 0\), which implies \(E(u_t, \psi_t) < E(u^*_{h_1}, 0)\). Therefore,
\[
c_0 < E(u^*_{h_1}, 0),
\]
so the least-energy solution $\left( u_0,\psi_0\right) $ is nontrivial.\qed

\

\section{Classification of least-energy solutions}
Under the setting in Theorem \ref{Classification}, we define
\[
(u_0,\psi_0)=\left(-\log B,\,\sqrt{1-\frac{A}{B^2}}\varphi_1\right).
\]
It is clear that  \((u_0,\psi_0)\) is a nontrivial solution to equation (\ref{A-B-key-Sphere}), and the energy is
\[
E(u_0,\psi_0)
=4\pi\left(1-2\log B-\frac{A}{B^2}
\right).
\]
We now prove that every nontrivial least-energy solution to equation (\ref{A-B-key-Sphere}) coincides with $(u_0,\psi_0)$ up to conformal automorphisms of \(\mathbb{S}^2\).

\

\noindent\textbf{Proof of Theorem \ref{Classification}. }
Let \((u,\psi)\) be a nontrivial solution as in Theorem \ref{Classification}. Define
\[
Q:=B\int_{\mathbb{S}^2}e^u|\psi|^2{\text{d}}v_{g_0}.
\]
Integrating the first equation in (\ref{Classification}) over \(\mathbb{S}^2\) yields
\begin{equation}\label{Cl-Liouville-Q}
	A\int_{\mathbb{S}^2}e^{2u}{\text{d}}v_{g_0}+Q=4\pi.
\end{equation}
On the other hand, multiplying the Dirac equation by \(\psi\) and integrating gives
\begin{equation}\label{Cl-Dirac-Q}
	\int_{\mathbb{S}^2}\langle \slashiii{D}_{g_0}\psi,\psi\rangle{\text{d}}v_{g_0}=Q.
\end{equation}
Using  \eqref{Cl-Liouville-Q} and \eqref{Cl-Dirac-Q}, the energy functional reduces to
\[
\begin{aligned}
	E(u,\psi)	&=	\int_{\mathbb{S}^2}
	\Big(	|\nabla u|^2+2u-Ae^{2u}
	\Big){\text{d}}v_{g_0}	+4\pi
	+2Q-2Q		\\
	&=\int_{\mathbb{S}^2}
	\Big(|\nabla u|^2+2u		\Big){\text{d}}v_{g_0}+Q.
\end{aligned}
\]
Next, by the sharp Onofri inequality on \(\mathbb{S}^2\),
\begin{equation}\label{Cl-Onofri}
	\int_{\mathbb{S}^2}\Big(	|\nabla u|^2+2u
	\Big){\text{d}}v_{g_0}\ge	4\pi
	\log\left(	\frac1{4\pi}		\int_{\mathbb{S}^2}e^{2u}{\text{d}}v_{g_0}
	\right).
\end{equation}
Combining \eqref{Cl-Liouville-Q} and \eqref{Cl-Onofri}, we deduce
\begin{equation}\label{Cl-energy-low}
	E(u,\psi)\geq 4\pi\log\left(\frac{4\pi-Q}{4\pi A}	\right)	+Q.
\end{equation}
Define
\[
f(Q):=	4\pi\log\left(	\frac{4\pi-Q}{4\pi A}
\right)	+Q.
\]
A direct computation shows that
\[
f'(Q)=1-\frac{4\pi}{4\pi-Q}=-\frac{Q}{4\pi-Q}<0
\qquad\text{for }Q\in(0,4\pi).
\]
Hence \(f\) is strictly decreasing.

We now derive an upper bound for \(Q\). Consider the conformal metric $g=e^{2u}g_0$ and define the transformed spinor
\[
\widetilde\psi=e^{-u/2}\psi.
\]
By conformal covariance of the Dirac operator,
\[
\slashiii{D}_g\widetilde\psi=B\widetilde\psi.
\]
Thus \(B\) is a positive eigenvalue of \(\slashiii{D}_g\), which implies $\lambda_1(\slashiii{D}_g)\le B$, where \(\lambda_1(\slashiii{D}_g)\) denotes the first positive Dirac eigenvalue of \((\mathbb{S}^2,g)\).

Applying B\"ar's sharp inequality \eqref{Bar-inq}, namely
\[
\lambda_1(\slashiii{D}_g)^2\mathrm{Area}(\mathbb{S}^2,g)\ge 4\pi,
\]
we obtain
\begin{equation}\label{Cl-Bar}
	B^2\mathrm{Area}(\mathbb{S}^2,g)\geq 4\pi.
\end{equation}
Since
\[
\mathrm{Area}(\mathbb{S}^2,g)=\int_{\mathbb{S}^2}e^{2u}{\text{d}}v_{g_0}=\frac{4\pi-Q}{A},
\]
it follows from \eqref{Cl-Bar} that
\begin{equation}\label{Cl-upbound}
	Q\le4\pi\left(	1-\frac{A}{B^2}	\right)=:Q_0.
\end{equation}
Since \(f\) is strictly decreasing, \eqref{Cl-energy-low} and \eqref{Cl-upbound} imply $E(u,\psi)\ge f(Q_0)$. A direct computation gives $f(Q_0)=	E(u_0,\psi_0)$. Therefore,
\begin{equation}\label{Cl-sharp}
	E(u,\psi)\ge E(u_0,\psi_0).
\end{equation}
Since \((u,\psi)\) is a least-energy solution, equality must hold in \eqref{Cl-sharp}. Consequently, equality holds simultaneously in the Onofri inequality and in  B\"ar's inequality. Equality in  B\"ar's inequality implies that \((\mathbb{S}^2,e^{2u}g_0)\) is a round sphere and that $\widetilde\psi=e^{-u/2}\psi$ is a first eigenspinor of \(\slashiii{D}_{e^{2u}g_0}\) corresponding to the positive eigenvalue $B$. Therefore, there exist a conformal diffeomorphism \(\Phi\in\mathrm{Conf}(\mathbb{S}^2)\) and a constant \(c_1>0\) such that
\[
g=e^{2u}g_0=c_1\,\Phi^{*}g_0.
\]
Since \(\Phi\) is a diffeomorphism, equality in \eqref{Cl-Bar} yields
\[
\operatorname{Area}(\Phi^{*}g_0)
=\operatorname{Area}(g_0) =4\pi ,
\]
and therefore
\[
\operatorname{Area}(g)=c_1\operatorname{Area}(g_0) =4\pi c_1.
\]
Combining this with \(B^{2}\operatorname{Area}(g)=4\pi\), we obtain $c_1=B^{-2}$. Hence
\[
u=\frac12\log\det(d\Phi)-\log B.
\]
Since \(\lambda_1(D_g)=B\) is a Killing spinor, there exists \(c_2\in\mathbb R\) such that
\[
\widetilde\psi=c_2\Phi^*\varphi_1,
\]
where \(\varphi_1\) is a normalized first eigenspinor on \((\mathbb{S}^2,g_0)\). Integrating the first equation of \eqref{Classification} and using
\[
\operatorname{Area}(g)
=\int_{\mathbb{S}^2}e^{2u}\text{d}v_{g_0}
=\frac{4\pi}{B^2},
\]
we obtain
\[
\int_{\mathbb{S}^2}|\widetilde\psi|^2\text{d}v_g
=\frac{4\pi}{B^2}\Bigl(1-\frac{A}{B^2}\Bigr).
\]
Since \(|\Phi^*\varphi_1|\equiv1\),
\[
\int_{\mathbb{S}^2}|\widetilde\psi|^2\text{d}v_g=c_2^2\operatorname{Area}(g)
=c_2^2\frac{4\pi}{B^2},
\]
and therefore
\[
c_2^2=1-\frac{A}{B^2}.
\]
Hence
\[
\widetilde\psi=\sqrt{1-\frac{A}{B^2}}\,\Phi^*\varphi_1.
\]
This completes the proof.\qed


\appendix
\makeatletter
\renewcommand{\@seccntformat}[1]{\ifcsname quad#1\endcsname\csname quad#1\endcsname\else\csname the#1\endcsname\quad\fi}
\renewcommand{\thesection}{Appendix \Alph{section}}
\makeatother
\renewcommand{\thetheorem}{\Alph{section}.\arabic{theorem}}
\newpage

\section{Sharp Spinorial $H^{1/2}$-Sobolev Inequality} 
Let $(\mathbb{S}^n, g_0)$, $n \ge 2$, denote the standard unit sphere, and let $\Sigma \mathbb{S}^n$ be its associated spinor bundle. Since $\slashiii{D}$ is a self-adjoint, unbounded, first-order elliptic differential operator, the spectral theorem implies that its absolute value $|\slashiii{D}| := (\slashiii{D}^2)^{1/2}$ defines a closed, non-negative quadratic form on the spinorial Sobolev space $H^{1/2}(\Sigma \mathbb{S}^n)$. We define the sharp constant associated with the critical embedding $H^{1/2}(\Sigma \mathbb{S}^n) \hookrightarrow L^{\frac{2n}{n-1}}(\Sigma \mathbb{S}^n)$ via the following quotient:
$$
C_{n,|D|} := \inf_{\psi \in H^{1/2}(\Sigma \mathbb{S}^n) \setminus \{0\}} \frac{\displaystyle \int_{\mathbb{S}^n} \langle |\slashiii{D}|\psi, \psi \rangle \, \text{d}v_{g_0}}{\displaystyle \left( \int_{\mathbb{S}^n} |\psi|^{\frac{2n}{n-1}} \text{d}v_{g_0} \right)^{\frac{n-1}{n}}},
$$
where $\langle \cdot, \cdot \rangle$ is the pointwise Hermitian inner product on the fibers of $\Sigma \mathbb{S}^n$, and $\text{d}v_{g_0}$ is the standard Riemannian volume element.
\begin{theorem}\label{Sharp-const}
	For any $n \ge 2$, the geometric optimal constant $C_{n,|D|}$ is given by:
	$$
	C_{n,|D|} = \frac{n-1}{2} \omega_n^{1/n},
	$$
	where $\omega_n := |\mathbb{S}^n|_{g_0} = \frac{2\pi^{(n+1)/2}}{\Gamma((n+1)/2)}$ is the total volume of the standard sphere. Equivalently, for every $\psi \in H^{1/2}(\Sigma \mathbb{S}^n)$, the following sharp inequality holds:
	$$
	\int_{\mathbb{S}^n} \langle |\slashiii{D}|\psi, \psi \rangle \, \text{d}v_{g_0} \ge \frac{n-1}{2} \omega_n^{1/n} \left( \int_{\mathbb{S}^n} |\psi|^{\frac{2n}{n-1}} \text{d}v_{g_0} \right)^{\frac{n-1}{n}}.
	$$
\end{theorem}
Let $P_1$ denote the conformal fractional Laplacian of order $1$ on $\mathbb{S}^n$, defined by:
$$
P_1 := \left( -\Delta_{\mathbb{S}^n} + \frac{(n-1)^2}{4} \right)^{1/2}.
$$
It is known (for example, see Xiong \cite{MR4059997}) that $P_1$ satisfies the spectral decomposition $P_1 Y_\ell = \left( \ell + \frac{n-1}{2} \right) Y_\ell$ for any spherical harmonic $Y_\ell$ of degree $\ell \in \mathbb{N}_0$.
\begin{lemma}\label{Dirac-Laplace}
	For every spinor field $\psi \in H^{1/2}(\Sigma \mathbb{S}^n)$, its pointwise norm satisfies $|\psi| \in H^{1/2}(\mathbb{S}^n)$, and one has the domination in the closed quadratic-form sense:
	$$
	\int_{\mathbb{S}^n} \langle |\slashiii{D}|\psi, \psi \rangle \, \text{d}v \ge \int_{\mathbb{S}^n} |\psi| P_1 |\psi| \, \text{d}v.
	$$
\end{lemma}
\begin{proof}
	The scalar curvature of the standard unit sphere is given by $\operatorname{Scal}_{\mathbb{S}^n} = n(n-1)$. The Lichnerowicz formula yields:
	$$
	\slashiii{D}^2 = \nabla^* \nabla + \frac{\operatorname{Scal}_{\mathbb{S}^n}}{4} = \nabla^* \nabla + \frac{n(n-1)}{4},
	$$
	where $\nabla^* \nabla$ is the connection Laplacian on $\Sigma \mathbb{S}^n$. By the Hess–Schrader–Uhlenbrock domination  theorem \cite{MR458243} for generalized Schrödinger semigroups, the spinorial heat semigroup is pointwise dominated by the scalar heat semigroup: for every $t > 0$, it holds almost everywhere on $\mathbb{S}^n$ that
	$$
	|e^{-t \slashiii{D}^2} \psi| \le e^{-t \left( -\Delta_{\mathbb{S}^n} + \frac{n(n-1)}{4} \right)} |\psi|.
	$$
	Since $\frac{n(n-1)}{4} \ge \frac{(n-1)^2}{4}$ for all $n \ge 2$, the spectral mapping theorem and the positivity of the scalar heat kernel imply:
	$$
	e^{-t \left( -\Delta_{\mathbb{S}^n} + \frac{n(n-1)}{4} \right)} \le e^{-t \left( -\Delta_{\mathbb{S}^n} + \frac{(n-1)^2}{4} \right)} = e^{-t P_1^2}.
	$$
	Thus, we arrive at the pointwise subordination-ready estimate:
	$$
	|e^{-t \slashiii{D}^2} \psi| \le e^{-t P_1^2} |\psi|.
	$$
	By virtue of Bochner's subordination formula \cite[Chapter IX]{MR617913}, for any non-negative self-adjoint operator $A$, the fractional semigroup can be represented via the generic formula:
	$$
	e^{-s A^{1/2}} = \frac{s}{2\sqrt{\pi}} \int_0^\infty t^{-3/2} e^{-\frac{s^2}{4t}} e^{-tA} \, \text{d}t, \quad s > 0.
	$$
	Applying this representation respectively to $A = \slashiii{D}^2$ and $A = P_1^2$, the positivity of the integrand yields:
	$$
	\left| e^{-s \left|\slashiii{D}\right|} \psi\right|  \le e^{-s P_1} |\psi| \quad (\text{a.e. } \mathbb{S}^n, \, \forall s > 0).
	$$
	Set $f=|\psi|\in L^2(\mathbb S^n)$. From the pointwise domination
	\[
	|e^{-s|\slashiii{D}|}\psi|\le e^{-sP_1}f
	\]
	and the positivity of \(e^{-sP_1}\), we obtain
	\[
	\operatorname{Re}\langle e^{-s|\slashiii{D}|}\psi,\psi\rangle_{L^2}
	\le
	\int_{\mathbb S^n}|e^{-s|\slashiii{D}|}\psi|\,|\psi|\,dv
	\le
	\langle e^{-sP_1}f,f\rangle_{L^2}.
	\]
	Hence
	\[
	\frac{\|\psi\|_2^2-\operatorname{Re}\langle e^{-s|\slashiii{D}|}\psi,\psi\rangle}{s}
	\ge
	\frac{\|f\|_2^2-\langle e^{-sP_1}f,f\rangle}{s}.
	\]
	Since \(\psi\in\mathcal D(|\slashiii{D}|^{1/2})=H^{1/2}(\Sigma\mathbb S^n)\), the left-hand side has a finite limit as \(s\to0\). Therefore the right-hand side is uniformly bounded as \(s\to0\), and by the standard semigroup characterization of closed quadratic forms, \(f\in\mathcal D(P_1^{1/2})=H^{1/2}(\mathbb S^n)\). For a non-negative self-adjoint operator \(A\), its closed quadratic form satisfies
	\[
	\|A^{1/2}u\|_2^2
	=
	\lim_{s\to0}
	\frac{\|u\|_2^2-\langle e^{-sA}u,u\rangle_{L^2}}{s},
	\qquad u\in\mathcal D(A^{1/2}),
	\]
	(see \cite[Chapter VI]{MR203473}). Applying this to \(A=|\slashiii{D}|\) and \(A=P_1\) gives
	$$
	\langle |\slashiii{D}|\psi, \psi \rangle_{L^2(\Sigma \mathbb{S}^n)} \ge \langle P_1 f, f \rangle_{L^2(\mathbb{S}^n)}.
	$$
	This completes the proof.
\end{proof}
\begin{lemma}\label{Beckner Inequality}
	For every scalar function $f \in H^{1/2}(\mathbb{S}^n)$, the following inequality holds and the constant is sharp:
	$$\int_{\mathbb{S}^n} f P_1 f \, \text{d}v \ge \frac{n-1}{2} \omega_n^{1/n} \left( \int_{\mathbb{S}^n} |f|^{\frac{2n}{n-1}} \, \text{d}v \right)^{\frac{n-1}{n}}.$$
\end{lemma}
\begin{proof}
	This is the classical Beckner fractional Sobolev inequality on the conformal sphere (see Beckner \cite{MR1230930}). To calibrate the normalization, substituting the constant function $f \equiv 1$ yields $P_1 1 = \frac{n-1}{2}$. Direct integration then shows that the ratio equals exactly $\frac{n-1}{2} \omega_n^{1 - \frac{n-1}{n}} = \frac{n-1}{2} \omega_n^{1/n}$.
\end{proof}

To establish that the constant is optimal, we construct a sequence of localized truncations modeled on the Euclidean fractional Sobolev extremal. Let $U(x) = (1 + |x|^2)^{-\frac{n-1}{2}}$ for $x \in \mathbb{R}^n$ be the standard bubble optimizing the sharp fractional Sobolev inequality on $\mathbb{R}^n$, satisfying:
$$
\frac{\displaystyle \int_{\mathbb{R}^n} U (-\Delta)^{1/2} U \, \text{d}x}{\displaystyle \left( \int_{\mathbb{R}^n} U^{\frac{2n}{n-1}} \, \text{d}x \right)^{\frac{n-1}{n}}} = \frac{n-1}{2} \omega_n^{1/n}.
$$
Let $\eta \in C_c^\infty(\mathbb{R}^n)$ be a radial cut-off function such that $\mathbf{1}_{B_1(0)} \le \eta \le \mathbf{1}_{B_2(0)}$. For $R > 1$, define $U_R(x) := \eta(x/R) U(x)$. 
\begin{lemma}
	Let $n\ge 2$. Let $\eta\in C_c^\infty(\mathbb R^n)$ satisfy
	\[
	0\le \eta\le 1,\qquad
	\eta\equiv 1 \ \text{in } B_1(0),\qquad
	\eta\equiv 0 \ \text{in } \mathbb R^n\setminus B_2(0).
	\]
	For $R>1$, set $\eta_R(x):=\eta(x/R).$ Then, for every $u\in \dot H^{1/2}(\mathbb R^n)$,
	\[
	\eta_R u\to u
	\qquad\text{in } \dot H^{1/2}(\mathbb R^n)
	\]
	as $R\to\infty$. In particular, if
	\[
	U(x):=(1+|x|^2)^{-\frac{n-1}{2}},
	\]
	then
	\[
	\eta_R U\to U
	\qquad\text{in } \dot H^{1/2}(\mathbb R^n).
	\]
\end{lemma}

\begin{proof}
	We use the Gagliardo characterization of $\dot H^{1/2}(\mathbb R^n)$:
	\[
	[w]_{\dot H^{1/2}(\mathbb R^n)}^2:=
	\iint_{\mathbb R^n\times\mathbb R^n}
	\frac{|w(x)-w(y)|^2}{|x-y|^{n+1}}\,dx\,dy.
	\]
	For $u\in \dot H^{1/2}(\mathbb R^n)$, we always take its Sobolev representative, which belongs to
	$L^{2n/(n-1)}(\mathbb R^n)$ by the homogeneous Sobolev embedding.
	
	Set $\chi_R:=1-\eta_R$.	It suffices to prove
	\[
	[\chi_Ru]_{\dot H^{1/2}(\mathbb R^n)}\to0.
	\]
	Let $v_R:=\chi_Ru$.	Then
	\[
	v_R(x)-v_R(y)
	=
	\chi_R(x)\bigl(u(x)-u(y)\bigr)
	+
	u(y)\bigl(\chi_R(x)-\chi_R(y)\bigr).
	\]
	Hence
	\[
	[v_R]_{\dot H^{1/2}}^2
	\le C(I_R+J_R),
	\]
	where
	\[
	I_R
	:=
	\iint_{\mathbb R^n\times\mathbb R^n}
	\chi_R(x)^2
	\frac{|u(x)-u(y)|^2}{|x-y|^{n+1}}
	\,dx\,dy
	\]
	and
	\[
	J_R
	:=
	\iint_{\mathbb R^n\times\mathbb R^n}
	|u(y)|^2
	\frac{|\chi_R(x)-\chi_R(y)|^2}{|x-y|^{n+1}}
	\,dx\,dy.
	\]
	
	We first show that $I_R\to0$. Since $0\le \chi_R\le1$ and
	\[
	\chi_R(x)\to0
	\qquad\text{for every }x\in\mathbb R^n,
	\]
	we have pointwise convergence
	\[
	\chi_R(x)^2
	\frac{|u(x)-u(y)|^2}{|x-y|^{n+1}}
	\to0
	\]
	for almost every $(x,y)$. Moreover,
	\[
	\chi_R(x)^2
	\frac{|u(x)-u(y)|^2}{|x-y|^{n+1}}
	\le
	\frac{|u(x)-u(y)|^2}{|x-y|^{n+1}},
	\]
	and the right-hand side is integrable on
	$\mathbb R^n\times\mathbb R^n$. Therefore, by dominated convergence,
	\[
	I_R\to0.
	\]
	
	It remains to prove that $J_R\to0$. For $y\in\mathbb R^n$, define
	\[
	K_R(y)
	:=
	\int_{\mathbb R^n}
	\frac{|\chi_R(x)-\chi_R(y)|^2}{|x-y|^{n+1}}\,dx.
	\]
	Then, by Tonelli's theorem,
	\[
	J_R
	=
	\int_{\mathbb R^n}
	|u(y)|^2K_R(y)\,dy.
	\]
	
	Write
	\[
	\chi:=1-\eta,
	\qquad
	\chi_R(x)=\chi(x/R).
	\]
	After the change of variables $x=RX$ and with $a:=y/R$, we obtain
	\[
	K_R(y)
	=
	R^{-1}K(y/R),
	\]
	where
	\[
	K(a)
	:=
	\int_{\mathbb R^n}
	\frac{|\chi(X)-\chi(a)|^2}{|X-a|^{n+1}}\,dX.
	\]
	
	We claim that
	\[
	K(a)\le \frac{C}{|a|}
	\qquad\text{for all }a\neq0.
	\]
	First assume $0<|a|\le4$. Since $\chi\in C^\infty(\mathbb R^n)$ and
	$\nabla\chi$ is bounded,
	\[
	|\chi(X)-\chi(a)|
	\le C|X-a|.
	\]
	Therefore,
	\[
	\int_{|X-a|<1}
	\frac{|\chi(X)-\chi(a)|^2}{|X-a|^{n+1}}\,dX
	\le
	C\int_{|X-a|<1}|X-a|^{1-n}\,dX
	\le C.
	\]
	On the other hand, using only the boundedness of $\chi$,
	\[
	\int_{|X-a|\ge1}
	\frac{|\chi(X)-\chi(a)|^2}{|X-a|^{n+1}}\,dX
	\le
	C\int_{|X-a|\ge1}|X-a|^{-n-1}\,dX
	\le C.
	\]
	Thus
	\[
	K(a)\le C
	\qquad\text{for }0<|a|\le4,
	\]
	and hence
	\[
	K(a)\le \frac{C}{|a|}
	\qquad\text{for }0<|a|\le4.
	\]
	Next assume $|a|>4$. Since $\eta\equiv0$ outside $B_2(0)$, we have
	\[
	\chi(a)=1,
	\]
	and $\chi(X)-1=-\eta(X)$ is supported in $B_2(0)$. Hence
	\[
	K(a)
	=
	\int_{B_2(0)}
	\frac{|\chi(X)-1|^2}{|X-a|^{n+1}}\,dX.
	\]
	For $X\in B_2(0)$ and $|a|>4$,
	\[
	|X-a|\ge |a|-|X|\ge |a|-2\ge \frac{|a|}{2}.
	\]
	Consequently,
	\[
	K(a)
	\le
	C|a|^{-n-1}
	\le
	\frac{C}{|a|}.
	\]
	The claim follows.
	
	Thus, for every \(y\neq0\),
	\[
	K_R(y)\le \frac{C}{|y|}.
	\]
	On the other hand, for fixed \(y\neq0\), \(y/R\to0\), and since \(K\) is
	uniformly bounded near the origin, 
	\[
	K_R(y)=R^{-1}K(y/R)\to0.
	\]
	
	We now use the fractional Hardy inequality, see for instance
	Frank et al. \cite{MR2425175}. Since $n\ge2$,
	\[
	\int_{\mathbb R^n}
	\frac{|u(y)|^2}{|y|}\,dy
	\le
	C_n [u]_{\dot H^{1/2}(\mathbb R^n)}^2.
	\]
	Hence
	\[
	|u(y)|^2K_R(y)
	\le
	C\frac{|u(y)|^2}{|y|}
	\in L^1(\mathbb R^n).
	\]
	By dominated convergence,
	\[
	J_R
	=
	\int_{\mathbb R^n}|u(y)|^2K_R(y)\,dy
	\to0.
	\]
	
	Combining the estimates for $I_R$ and $J_R$, we obtain
	\[
	[\chi_Ru]_{\dot H^{1/2}(\mathbb R^n)}^2\to0.
	\]
	Since
	\[
	u-\eta_Ru=\chi_Ru,
	\]
	we conclude that
	\[
	\eta_Ru\to u
	\qquad\text{in } \dot H^{1/2}(\mathbb R^n).
	\]
	
	It remains to verify that \(U\in \dot H^{1/2}(\mathbb R^n)\). 
	Recall that
	\[
	U(x)=(1+|x|^2)^{-\frac{n-1}{2}}.
	\]
	belongs to Bessel-Kernel family. By the classical Fourier transform formula for Bessel kernels \cite{MR304972}, there exist
	constants \(C,c>0\) such that
	\[
	|\widehat U(\xi)|\le C|\xi|^{-1}e^{-c|\xi|}
	\qquad (\xi\neq0).
	\]
	Hence
	\[
	\begin{aligned}
		\|U\|_{\dot H^{1/2}(\mathbb R^n)}^2
		&=
		\int_{\mathbb R^n}|\xi|\,|\widehat U(\xi)|^2\,d\xi  \\
		&\le
		C\int_{\mathbb R^n}|\xi|^{-1}e^{-2c|\xi|}\,d\xi  \\
		&=
		C\omega_{n-1}\int_0^\infty r^{n-2}e^{-2cr}\,dr
		<\infty,
	\end{aligned}
	\]
	since \(n\ge2\). Therefore \(U\in\dot H^{1/2}(\mathbb R^n)\), and the cutoff
	approximation proved above gives
	\[
	\eta_RU\to U
	\qquad\text{in } \dot H^{1/2}(\mathbb R^n).
	\]
\end{proof}

Since
\[
U^{\frac{2n}{n-1}}=(1+|x|^2)^{-n}\in L^1(\mathbb R^n),
\]
we have \(U\in L^{\frac{2n}{n-1}}(\mathbb R^n)\). Moreover, since \(0\le\eta(x/R)\le1\), \(\eta(x/R)\to1\) pointwise, dominated convergence gives
\[
U_R\to U
\quad\text{in}\quad
L^{\frac{2n}{n-1}}(\mathbb R^n).
\]
Consequently,
\[
U_R\to U
\quad\text{in}\quad
L^{\frac{2n}{n-1}}(\mathbb R^n)\cap \dot H^{1/2}(\mathbb R^n).
\]

Fix a base point $p \in \mathbb{S}^n$. Let $\exp_p: B_\rho(0) \subset \mathbb{R}^n \to \mathbb{S}^n$ be the Riemannian normal coordinates around $p$. In this coordinate ball, we pick a synchronous local spin frame and fix a unit spinor $\xi_0$ in the model fiber (so that $|\xi_0|=1$).

For fixed $R$ and a sufficiently small scaling parameter $\varepsilon > 0$, we define the test spinor sequence by:
$$
\psi_{\varepsilon,R}(\exp_p(\varepsilon y)) = \varepsilon^{-\frac{n-1}{2}} U_R(y) \xi_0, \quad y \in B_{\rho/\varepsilon}(0),
$$
and set $\psi_{\varepsilon,R} = 0$ outside the geodesic ball. Since \(U_R\) is compactly supported and, for small \(\varepsilon\), its support is strictly contained in the coordinate ball, the zero extension defines a smooth spinor.

Using the normal coordinate expansion of the volume element $$
\text{d}v_{g_0}(\exp_p(\varepsilon y)) = \left( 1 + O(\varepsilon^2 |y|^2) \right) \varepsilon^n \text{d}y,
$$
and noticing that the critical scaling satisfies $\varepsilon^{-\frac{n-1}{2} \cdot \frac{2n}{n-1}} \cdot \varepsilon^n = 1$, a direct integration shows:
\begin{equation}\label{denominator}
	\lim_{\varepsilon \to 0} \int_{\mathbb{S}^n} |\psi_{\varepsilon,R}|^{\frac{2n}{n-1}} \, \text{d}v = \int_{\mathbb{R}^n} U_R(y)^{\frac{2n}{n-1}} \, \text{d}y.
\end{equation}
Next, we analyze the energy form. 
\begin{lemma}
	Let \(p\in\mathbb S^n\), and choose geodesic normal coordinates centered at
	\(p\). After fixing a smooth local orthonormal spin frame, we identify the
	spinor bundle over the coordinate ball with the trivial bundle
	\(B_\rho(0)\times\mathbb C^N\). Thus a compactly supported spinor in the
	coordinate ball may be viewed as a \(\mathbb C^N\)-valued function. For \(\Phi\in C_c^\infty(B_R(0),\mathbb C^N)\), define
	\[
	\Phi_\varepsilon(\exp_p(\varepsilon y))
	=
	\varepsilon^{-\frac{n-1}{2}}\Phi(y),
	\]
	and set \(\Phi_\varepsilon=0\) outside \(B_{\varepsilon R}(p)\). Then
	\begin{equation}\label{numerator}
	\lim_{\varepsilon\to0}
	\int_{\mathbb S^n}
	\langle |\slashiii{D}|\Phi_\varepsilon,\Phi_\varepsilon\rangle\,dv_{g_0}
	=
	\int_{\mathbb R^n}
	\langle (-\Delta_{\mathbb R^n})^{1/2}\Phi,\Phi\rangle_{\mathbb C^N}\,dy,
	\end{equation}
	where \((-\Delta_{\mathbb R^n})^{1/2}\) denotes the Fourier multiplier
	\[
	\widehat{(-\Delta_{\mathbb R^n})^{1/2}\Phi}(\xi)	=|\xi|\widehat{\Phi}(\xi),
	\]
	acting componentwise on \(\mathbb C^N\)-valued functions.
\end{lemma}

\begin{proof}
	Fix $\chi \in C_c^\infty(B_\rho(p))$ with $\chi \equiv 1$ on
	$B_{\rho/2}(p)$ and $0 \le \chi \le 1$. For all $\varepsilon > 0$ small
	enough that $\varepsilon R < \rho/2$, the support of $\Phi_\varepsilon$ is
	contained in $B_{\rho/2}(p)$, so $\Phi_\varepsilon = \chi\Phi_\varepsilon$.
	Using this and the self-adjointness of $\chi$ as a multiplication operator
	($\langle u,\chi v\rangle = \langle \chi u, v\rangle$), we compute
	\[
	\begin{aligned}
		\langle |\slashiii{D}|\Phi_\varepsilon,\Phi_\varepsilon\rangle_{L^2(\mathbb{S}^n)}
		&= \langle |\slashiii{D}|(\chi\Phi_\varepsilon),\chi\Phi_\varepsilon
		\rangle_{L^2(\mathbb{S}^n)}\\
		&= \langle \chi\,|\slashiii{D}|\,\chi\,\Phi_\varepsilon,\Phi_\varepsilon
		\rangle_{L^2(\mathbb{S}^n)}.
	\end{aligned}
	\]
	Set $T := \chi\,|\slashiii{D}|\,\chi$. In the chosen coordinate chart and spin
	trivialization, $T$ may be represented as a properly supported classical
	pseudodifferential operator of order one on $\mathbb{R}^n$, after an
	arbitrary smooth extension of its local symbol outside the chart. This
	extension does not affect the quadratic form evaluated on $\Phi_\varepsilon$
	for sufficiently small $\varepsilon$, since $\operatorname{supp}\Phi_\varepsilon$
	lies strictly inside the chart. The Schwartz kernel of $T$ equals
	$K_T(x,z) = \chi(x)\,K_{|\slashiii{D}|}(x,z)\,\chi(z)$, which is compactly supported in
	$(x,z)$ but retains the standard conormal singularity along the diagonal
	$\{x = z\}$. The Kohn--Nirenberg symbol of \(T\) admits a classical asymptotic
	expansion. Fix an integer \(N_0\ge n+2\). We write it in the finite form
	\[
	a(x,\xi)
	=
	a_1(x,\xi)+a_0(x,\xi)+a_{-1}(x,\xi)+\cdots+a_{-N_0}(x,\xi)
	+r_{-N_0-1}(x,\xi),
	\]
	where \(a_j\in S^j\) and \(r_{-N_0-1}\in S^{-N_0-1}\). The principal symbol is
	\[
	a_1(x,\xi)=|\xi|_{g(x)}\,\mathrm{Id}_{\mathbb C^N},
	\qquad
	|\xi|_{g(x)}=\sqrt{g^{ij}(x)\xi_i\xi_j}.
	\]
	As usual, this identity for the principal symbol is understood on the
	high-frequency region \(|\xi|\ge c>0\); the low-frequency part is treated
	separately below by the fixed cutoff \(\varphi\). The lower-order terms \(a_j\), \(0\ge j\ge -N_0\), encode curvature,
	spin-connection, coordinate, and quantization contributions. The remainder
	\(r_{-N_0-1}\) is of sufficiently negative order.
	
	Representing $T$ by its Kohn--Nirenberg symbol up to a smoothing operator
	$R_\infty \in \Psi^{-\infty}$:
	\begin{equation}\label{eq:osc}
		\langle T\Phi_\varepsilon,\Phi_\varepsilon\rangle_{L^2}
		= (2\pi)^{-n}
		\iiint
		e^{i(x-z)\cdot\xi}\,
		\bigl\langle a(x,\xi)\Phi_\varepsilon(z),\Phi_\varepsilon(x)\bigr\rangle
		\sqrt{g(x)}\,\mathrm{d}z\,\mathrm{d}\xi\,\mathrm{d}x
		\;+\; \langle R_\infty\Phi_\varepsilon,\Phi_\varepsilon\rangle_{L^2}.
	\end{equation}
	Since $R_\infty$ has smooth kernel $K_{R_\infty}$, the remainder satisfies
	$|\langle R_\infty\Phi_\varepsilon,\Phi_\varepsilon\rangle|
	\le \|K_{R_\infty}\|_{L^\infty}\|\Phi_\varepsilon\|_{L^1}^2$.
	Changing variables $x = \exp_p(\varepsilon X)$ and using
	$\Phi_\varepsilon(\exp_p(\varepsilon X)) = \varepsilon^{-(n-1)/2}\Phi(X)$
	together with $\mathrm{d}v_{g_0} = \varepsilon^n\sqrt{g(\varepsilon X)}\,\mathrm{d}X$
	gives
	\[
	\|\Phi_\varepsilon\|_{L^1(\mathbb{S}^n)}
	= \varepsilon^{(n+1)/2}\bigl(\|\Phi\|_{L^1} + O_R(\varepsilon^2)\bigr),
	\]
	so $\|\Phi_\varepsilon\|_{L^1}^2 = O_R(\varepsilon^{n+1}) \to 0$, and thus
	the smoothing remainder is $o(1)$.
	
	In the main integral of \eqref{eq:osc}, substitute $x = \varepsilon X$,
	$z = \varepsilon Z$, $\xi = \varepsilon^{-1}\eta$, so
	$(x-z)\cdot\xi = (X-Z)\cdot\eta$ and
	$\mathrm{d}x\,\mathrm{d}z\,\mathrm{d}\xi = \varepsilon^n\,\mathrm{d}X\,\mathrm{d}Z\,\mathrm{d}\eta$.
	The two factors $\varepsilon^{-(n-1)/2}$ from $\Phi_\varepsilon(\varepsilon X)$
	and $\Phi_\varepsilon(\varepsilon Z)$, combined with $\varepsilon^n$, give the
	overall prefactor $\varepsilon^{-(n-1)/2}\cdot\varepsilon^{-(n-1)/2}\cdot
	\varepsilon^n = \varepsilon$. Hence
	\begin{equation}\label{eq:rescaled}
		\langle T\Phi_\varepsilon,\Phi_\varepsilon\rangle_{L^2}
		= (2\pi)^{-n}
		\iiint
		e^{i(X-Z)\cdot\eta}\,
		\bigl\langle
		\varepsilon\,a(\varepsilon X,\varepsilon^{-1}\eta)\,\Phi(Z),
		\Phi(X)
		\bigr\rangle
		\sqrt{g(\varepsilon X)}\,\mathrm{d}Z\,\mathrm{d}\eta\,\mathrm{d}X
		+ o(1).
	\end{equation}
	
	The density factor $\sqrt{g(\varepsilon X)}$ is a smooth $S^0$-multiplier
	satisfying $\sqrt{g(\varepsilon X)} = 1 + O_R(\varepsilon^2)$ uniformly for
	$|X| \le R$. Multiplying the rescaled symbols below by this factor therefore
	preserves all symbol-seminorm bounds, up to an additional $O_R(\varepsilon^2)$
	contribution that is subsumed in the $o_R(1)$ error.
	
	\smallskip
	\noindent\emph{(I) Lower-order terms.}
	For each \(j=0,-1,\ldots,-N_0\), put
	\[
	b_{\varepsilon,j}(X,\eta)
	:=
	\varepsilon\,a_j(\varepsilon X,\varepsilon^{-1}\eta).
	\]
	We split \(b_{\varepsilon,j}\) into low and high frequencies. Let
	\(\varphi\in C_c^\infty(\mathbb R^n)\) satisfy \(\varphi\equiv1\) near
	\(\eta=0\).
	
	For the low-frequency part we do not use an \(L^2\)-operator norm estimate.
	Since \(\Phi\) is fixed and compactly supported, \(\varphi b_{\varepsilon,j}\)
	is integrated only over compact \(X,Z,\eta\)-sets. On the support of \(\Phi\) and \(\varphi\), we have \(|X|\le R\) and
	\(\eta\in \operatorname{supp}\varphi\). Since \(a_j\in S^j\) with \(j\le0\),
	\[
	\|a_j(\varepsilon X,\varepsilon^{-1}\eta)\|_{\operatorname{End}(\mathbb C^N)}
	\le C_R.
	\]
	Hence
	\[
	\|b_{\varepsilon,j}(X,\eta)\|_{\operatorname{End}(\mathbb C^N)}
	=
	\varepsilon
	\|a_j(\varepsilon X,\varepsilon^{-1}\eta)\|_{\operatorname{End}(\mathbb C^N)}
	\le C_R\varepsilon.
	\]
	Therefore, using \(|\langle Av,w\rangle|\le \|A\|\,|v|\,|w|\) and integration by separation of variables,
	\[
	\begin{aligned}
		\left|
		\left\langle
		\operatorname{Op}(\varphi b_{\varepsilon,j})\Phi,\Phi
		\right\rangle
		\right|
		&\le
		C
		\iiint
		|\varphi(\eta)|
		\|b_{\varepsilon,j}(X,\eta)\|
		|\Phi(Z)|\,|\Phi(X)|
		\,dZ\,d\eta\,dX  \\
		&\le
		C_R\varepsilon
		\|\varphi\|_{L^1}
		\|\Phi\|_{L^1}^2
		=o_R(1).
	\end{aligned}
	\]
	
	For the high-frequency part \((1-\varphi)b_{\varepsilon,j}\), the estimates are
	uniform away from \(\eta=0\). For all multi-indices \(\alpha,\beta\),
	\[
	\partial_X^\alpha\partial_\eta^\beta b_{\varepsilon,j}(X,\eta)
	=
	\varepsilon^{1+|\alpha|-|\beta|}
	(\partial_x^\alpha\partial_\xi^\beta a_j)(\varepsilon X,\varepsilon^{-1}\eta).
	\]
	On \(\operatorname{supp}(1-\varphi)\), \(|\eta|\ge c>0\), and therefore
	\[
	|\partial_X^\alpha\partial_\eta^\beta b_{\varepsilon,j}(X,\eta)|
	\le
	C_{\alpha\beta R}\varepsilon^{1-j}(1+|\eta|)^{j-|\beta|}.
	\]
	Thus
	\[
	(1-\varphi)b_{\varepsilon,j}\in \varepsilon^{1-j}S^j
	\subset \varepsilon S^0 .
	\]
	Combining the low-frequency quadratic-form estimate with the high-frequency
	operator estimate, we obtain
	\[
	\bigl|\langle\operatorname{Op}(b_{\varepsilon,j})\Phi,\Phi\rangle\bigr|
	=o_R(1).
	\]
	Since the number of indices \(j=0,-1,\ldots,-N_0\) is finite, summing over
	these terms still gives an \(o_R(1)\) contribution to the quadratic form.	
	
	It remains to treat the finite-order remainder. Define
	\[
	b_{\varepsilon,\mathrm{rem}}(X,\eta)
	:=
	\varepsilon\,r_{-N_0-1}(\varepsilon X,\varepsilon^{-1}\eta).
	\]
	We split it again into low and high frequencies. The low-frequency contribution
	is estimated directly as above:
	\[
	\left|
	\left\langle
	\operatorname{Op}(\varphi b_{\varepsilon,\mathrm{rem}})\Phi,\Phi
	\right\rangle
	\right|
	=o_R(1),
	\]
	because \(\varphi b_{\varepsilon,\mathrm{rem}}\) is integrated only over compact
	\(X,Z,\eta\)-sets and satisfies
	\[
	\|\varphi(\eta)b_{\varepsilon,\mathrm{rem}}(X,\eta)\|_{End(\mathbb{C}^N)}
	\le C_R\varepsilon .
	\]
	For the high-frequency part, since \(r_{-N_0-1}\in S^{-N_0-1}\), the same
	rescaling calculation gives
	\[
	(1-\varphi)b_{\varepsilon,\mathrm{rem}}
	\in
	\varepsilon^{N_0+2}S^{-N_0-1}
	\subset	o_R(1)S^0.
	\]
	Thus its \(L^2\to L^2\) operator norm is \(o_R(1)\), and hence
	\[
	\left|
	\left\langle
	\operatorname{Op}((1-\varphi)b_{\varepsilon,\mathrm{rem}})\Phi,\Phi
	\right\rangle
	\right|
	=o_R(1).
	\]
	\smallskip
	\noindent\emph{(II) Principal symbol error.}
	Define
	\[
	r_\varepsilon(X,\eta) := \sqrt{g(\varepsilon X)}\,|\eta|_{g(\varepsilon X)} - |\eta|.
	\]
	We decompose $|\eta| = (1-\varphi(\eta))|\eta| + \varphi(\eta)|\eta|$, where
	$\varphi \in C_c^\infty(\mathbb{R}^n)$ satisfies $\varphi \equiv 1$ near
	$\eta = 0$. The high-frequency part $(1-\varphi)|\eta|$ is a classical symbol
	of order one, smooth away from the origin and supported in $|\eta| \ge c > 0$.
	
	The low-frequency part is treated directly at the level of quadratic forms.
	On the compact \(X\)-support of \(\Phi\) and on \(\operatorname{supp}\varphi\),
	the normal-coordinate expansion gives
	\[
	g^{ij}(\varepsilon X)=\delta^{ij}+O_R(\varepsilon^2),
	\qquad
	\sqrt{g(\varepsilon X)}=1+O_R(\varepsilon^2).
	\]
	Hence, on $\operatorname{supp}\varphi$,
	\[
	\left|	\sqrt{g(\varepsilon X)}\,|\eta|_{g(\varepsilon X)}-|\eta|
	\right|\leq \left| {(\sqrt {g(\varepsilon X)}  - 1)|\eta {|_{g(\varepsilon X)}}} \right| + \left| {|\eta {|_{g(\varepsilon X)}} - |\eta |} \right|	\le C_R\varepsilon^2|\eta|
	\le C_R\varepsilon^2.
	\]
	Therefore
	\[
	\begin{aligned}
		&\left|	(2\pi)^{-n}	\iiint
		e^{i(X-Z)\cdot\eta}
		\varphi(\eta)	r_\varepsilon(X,\eta)
		\langle \Phi(Z),\Phi(X)\rangle
		\,dZ\,d\eta\,dX	\right| \\
		&\qquad\le	C_R\varepsilon^2
		\|\varphi\|_{L^1}\|\Phi\|_{L^1}^2
		=o_R(1).
	\end{aligned}
	\] 
	The low-frequency contribution has therefore been controlled directly at the
	quadratic-form level. In the remaining high-frequency region, \(|\eta|\) is a
	smooth classical symbol. 
	
	For the
	high-frequency part, $g_{ij}(\varepsilon X) = \delta_{ij} + O_R(\varepsilon^2)$
	uniformly for $|X| \le R$, and the same chain-rule argument shows that the high-frequency error symbol \((1-\varphi)r_\varepsilon\) belongs to
	\(\varepsilon^2 S^1\) on the compact \(X\)-support. Therefore
	\[
	\operatorname{Op}((1-\varphi)r_\varepsilon):
	H^{1/2}(\mathbb R^n;\mathbb C^N)
	\to
	H^{-1/2}(\mathbb R^n;\mathbb C^N)
	\]
	has operator norm \(O_R(\varepsilon^2)\), and since
	\(\Phi\in C_c^\infty\subset H^{1/2}\),
	\[
	\bigl|	\langle	\operatorname{Op}((1-\varphi)r_\varepsilon)\Phi,\Phi	\rangle	\bigr|	\le	C_R\varepsilon^2\|\Phi\|_{H^{1/2}}^2
	=	O_R(\varepsilon^2).
	\]
	Collecting all contributions:
	\begin{equation}\label{eq:combined}
		\langle |\slashiii{D}|\Phi_\varepsilon,\Phi_\varepsilon\rangle_{L^2(\mathbb{S}^n)}
		= (2\pi)^{-n}
		\iiint
		e^{i(X-Z)\cdot\eta}\,|\eta|\,
		\bigl\langle\Phi(Z),\Phi(X)\bigr\rangle
		\,\mathrm{d}Z\,\mathrm{d}\eta\,\mathrm{d}X
		+ o_R(1).
	\end{equation}
	
	The triple integral in \eqref{eq:combined} is interpreted as the quadratic
	form of the Fourier multiplier $\operatorname{Op}(|\eta|)$. The Fourier multiplier \(\operatorname{Op}(|\eta|)\) is understood in the
	standard quadratic-form sense; the preceding cutoff decomposition justifies
	passing from the smooth symbol pieces to this multiplier. Since
	$\Phi \in C_c^\infty$, we have $\hat\Phi \in \mathcal{S}(\mathbb{R}^n;\mathbb{C}^N)$,
	and Plancherel's theorem gives
	\[
	(2\pi)^{-n}
	\iiint
	e^{i(X-Z)\cdot\eta}\,|\eta|\,
	\langle\Phi(Z),\Phi(X)\rangle
	\,\mathrm{d}Z\,\mathrm{d}\eta\,\mathrm{d}X
	= \langle\operatorname{Op}(|\eta|)\Phi,\Phi\rangle
	= (2\pi)^{-n}\int_{\mathbb{R}^n}|\eta|\,|\hat\Phi(\eta)|^2\,\mathrm{d}\eta.
	\]
	This equals $\int_{\mathbb{R}^n}\langle(-\Delta)^{1/2}\Phi,\Phi\rangle\,\mathrm{d}X$.
	The identification is consistent with the flat-space Dirac operator
	$\slashiii{D}_{\mathbb{R}^n}$, which satisfies
	\[
	\slashiii{D}_{\mathbb{R}^n}^2 = -\Delta\,\mathrm{Id}_{\mathbb{C}^N},
	\qquad\text{hence}\qquad
	|\slashiii{D}_{\mathbb{R}^n}|
	= \bigl(\slashiii{D}_{\mathbb{R}^n}^2\bigr)^{1/2}
	= (-\Delta)^{1/2}\,\mathrm{Id}_{\mathbb{C}^N}.
	\]
	Inserting into \eqref{eq:combined} and letting $\varepsilon \to 0$ completes
	the proof.
\end{proof}

\noindent\textbf{Proof of Theorem \ref{Sharp-const}. } Combining Lemma \ref{Dirac-Laplace} and Lemma \ref{Beckner Inequality} with $f = |\psi|$, we deduce that for any non-zero $\psi \in H^{1/2}(\Sigma \mathbb{S}^n)$:
$$
\int_{\mathbb{S}^n} \langle |\slashiii{D}|\psi, \psi \rangle \, \text{d}v \ge \int_{\mathbb{S}^n} |\psi| P_1 |\psi| \, \text{d}v \ge \frac{n-1}{2} \omega_n^{1/n} \left( \int_{\mathbb{S}^n} |\psi|^{\frac{2n}{n-1}} \, \text{d}v \right)^{\frac{n-1}{n}}.
$$
Taking the infimum over all non-zero spinors, we obtain:
$$
C_{n,|D|} \ge \frac{n-1}{2} \omega_n^{1/n}.
$$
Combining the convergence of the numerator (\ref{numerator}) and denominator (\ref{denominator}), we have for any fixed $R > 1$:
$$
\lim_{\varepsilon \to 0} \frac{\displaystyle \int_{\mathbb{S}^n} \langle |\slashiii{D}|\psi_{\varepsilon,R}, \psi_{\varepsilon,R} \rangle \, \text{d}v}{\displaystyle \left( \int_{\mathbb{S}^n} |\psi_{\varepsilon,R}|^{\frac{2n}{n-1}} \text{d}v \right)^{\frac{n-1}{n}}} = \frac{\displaystyle \int_{\mathbb{R}^n} U_R (-\Delta)^{1/2} U_R \, \text{d}y}{\displaystyle \left( \int_{\mathbb{R}^n} U_R^{\frac{2n}{n-1}} \text{d}y \right)^{\frac{n-1}{n}}}.
$$
Letting $R \to \infty$, the right-hand side approaches the Euclidean sharp quotient $\frac{n-1}{2} \omega_n^{1/n}$. 
Choose \(R_k\to\infty\). For each \(k\), choose \(\varepsilon_k>0\) sufficiently
small so that the quotient of \(\psi_{\varepsilon_k,R_k}\) differs from the
corresponding Euclidean quotient of \(U_{R_k}\) by at most \(1/k\). Then the
resulting sequence \(\psi_k:=\psi_{\varepsilon_k,R_k}\) satisfies
\[
\mathop {\lim }\limits_{k \to \infty} \frac{{\int_{{\mathbb{S}^n}} {\left\langle {\left|\slashiii{D} \right|{\psi _k},{\psi _k}} \right\rangle } {\mkern 1mu} {\text{d}}v}}{{{{\left( {{{\int_{{\mathbb{S}^n}} {\left| {{\psi _k}} \right|} }^{\frac{{2n}}{{n - 1}}}}{\text{d}}v} \right)}^{\frac{{n - 1}}{n}}}}}=
\frac{n-1}{2}\omega_n^{1/n}.
\]
Therefore
$$
C_{n,|D|} \le \frac{n-1}{2} \omega_n^{1/n}.
$$
The matching upper and lower bounds yield the desired conclusion.\qed

\begin{theorem}
	Let $n \ge 2$. The infimum $C_{n,|D|}$ is not attained by any non-zero spinor $\psi \in H^{1/2}(\Sigma \mathbb{S}^n)$.
\end{theorem}

\begin{proof}
	Suppose, for the sake of contradiction, that there exists a minimizer $\psi_0 \in H^{1/2}(\Sigma \mathbb{S}^n) \setminus \{0\}$. Let $p = \frac{2n}{n-1}$ and 
	\[
	B = \left(-\Delta + \frac{n(n-1)}{4}\right)^{1/2}.
	\]
	Repeating the semigroup-domination argument in Lemma \ref{Dirac-Laplace} but without
	replacing the potential \(n(n-1)/4\) by \((n-1)^2/4\), Bochner subordination gives
	\[
	|e^{-s|\slashiii{D}|}\psi|
	\le
	e^{-sB}|\psi|.
	\]
	Using the same difference-quotient argument as above, we obtain
	\begin{equation}\label{eq:form_domination}
		\int_{{\mathbb{S}^n}} {\left\langle {|\slashiii{D}|\psi ,\psi } \right\rangle } {\mkern 1mu} {\text{d}}v \geq \int_{{\mathbb{S}^n}} {\left\langle {B\left| \psi  \right|,\left| \psi  \right|} \right\rangle } {\mkern 1mu} {\text{d}}v.
	\end{equation}
	
	Next, we establish a strict spectral comparison between $B$ and $P_1$. Let \(f\in H^{1/2}(\mathbb S^n)\setminus\{0\}\). Interpreting both sides as
	closed quadratic forms, write \(f=\sum_{\ell=0}^\infty f_\ell\) according to its orthogonal decomposition into spherical harmonics of degree $\ell$. The spectral mappings for $P_1$ and $B$ yield
	\[
	P_1 f_\ell = \left(\ell + \frac{n-1}{2}\right) f_\ell, \quad \text{and} \quad B f_\ell = \left[ \left(\ell + \frac{n-1}{2}\right)^2 + \frac{n-1}{4} \right]^{1/2} f_\ell.
	\]
	Since $n \ge 2$, the inequality $\frac{n-1}{4} > 0$ holds strictly. Consequently, for every $\ell \in \mathbb{N}_0$, we have the strict eigenvalue inequality:
	\[
	\left[ \left(\ell + \frac{n-1}{2}\right)^2 + \frac{n-1}{4} \right]^{1/2} > \ell + \frac{n-1}{2}.
	\]
	Summing over the spectrum and invoking Parseval's identity, we deduce that for any non-zero $f \in H^{1/2}(\mathbb{S}^n)$, the quadratic forms satisfy the strict global inequality:
	\begin{equation}\label{eq:strict_operator_comparison}
		\int_{{\mathbb{S}^n}} {\left\langle {Bf,f} \right\rangle } {\mkern 1mu} {\text{d}}v > \int_{{\mathbb{S}^n}} {\left\langle {P_1 f,f} \right\rangle } {\mkern 1mu} {\text{d}}v.
	\end{equation}

	We now return to the hypothetical minimizer $\psi_0$. We normalize \(\psi_0\) so that
	\[
	\left(\int_{\mathbb S^n}|\psi_0|^p\,dv\right)^{2/p}
	=\|\psi_0\|_{L^p}^2
	=1,
	\] 
	and set $f_0 = |\psi_0| \in H^{1/2}(\mathbb{S}^n)$, which is clearly non-zero. Since $\psi_0$ achieves the sharp geometric constant $C_{n,|D|}$, its spinorial energy is locked at
	\begin{equation}\label{eq:attainment_value}
		\int_{{\mathbb{S}^n}} {\left\langle {|\slashiii{D}|{\psi _0},{\psi _0}} \right\rangle } {\mkern 1mu} {\text{d}}v= C_{n,|D|} = \frac{n-1}{2}\omega_n^{1/n}.
	\end{equation}
	On the other hand, concatenating the form domination \eqref{eq:form_domination} and the strict spectral comparison \eqref{eq:strict_operator_comparison} applied to $f_0$, we obtain:
	\begin{equation}\label{eq:contradiction_chain}
		\int_{{\mathbb{S}^n}} {\left\langle {|\slashiii{D}|{\psi _0},{\psi _0}} \right\rangle } {\mkern 1mu} {\text{d}}v \geq \int_{{\mathbb{S}^n}} {\left\langle {B{f_0},{f_0}} \right\rangle } {\mkern 1mu} {\text{d}}v > \int_{{\mathbb{S}^n}} {\left\langle {{P_1}{f_0},{f_0}} \right\rangle } {\mkern 1mu} {\text{d}}v.
	\end{equation}
	By virtue of the sharp scalar Beckner Sobolev inequality on the conformal sphere (see Beckner \cite{MR1230930}), the energy under $P_1$ is bounded from below by:
	\[
	\int_{{\mathbb{S}^n}} {\left\langle {{P_1}{f_0},{f_0}} \right\rangle } {\mkern 1mu} {\text{d}}v \ge \frac{n-1}{2}\omega_n^{1/n} \|f_0\|_{L^{p}(\mathbb{S}^n)}^2 = \frac{n-1}{2}\omega_n^{1/n}.
	\]
	Combining this with \eqref{eq:contradiction_chain} gives
	\[
	\int_{\mathbb S^n}\langle |\slashiii{D}|\psi_0,\psi_0\rangle\,dv
	>
	\frac{n-1}{2}\omega_n^{1/n},
	\]
	contradicting \eqref{eq:attainment_value}.
	
	Thus, the infimum $C_{n,|D|}$ cannot be attained by any non-zero element in $H^{1/2}(\Sigma \mathbb{S}^n)$, completing the proof.
\end{proof}

\small
\bibliographystyle{abbrv}
\bibliography{Sphere}
\end{document}